\documentclass[11pt,reqno]{amsart}
\usepackage{amsmath, amssymb, amsthm}
\usepackage{url}
\usepackage{mathrsfs}
\usepackage[ansinew]{inputenc}
\usepackage[breaklinks]{hyperref}
\usepackage[top=1in, bottom=1in, left=1in, right=1in]{geometry}
\newcommand{\pFq}[5]{\ensuremath{{}_{#1}F_{#2} \left( \genfrac{}{}{0pt}{}{#3}
{#4} \bigg| {#5} \right)}}

 \renewcommand{\a}{\alpha}
\renewcommand{\b}{\beta}

\renewcommand{\d}{{\delta}}
\newcommand{\g}{\gamma}
\newcommand{\G}{\Gamma}
\renewcommand{\l}{\lambda}

\renewcommand{\(}{\left\(}
\renewcommand{\)}{\right\)}
\renewcommand{\[}{\left\[}
\renewcommand{\]}{\right\]}
\newtheorem{remark}[]{Remark}
\numberwithin{equation}{section}
 \theoremstyle{plain}
\newtheorem{theorem}{Theorem}[section]
\newtheorem{lemma}[theorem]{Lemma}

\newtheorem{corollary}[theorem]{Corollary}

   \makeatletter
\def\proof{\@ifnextchar[{\@oproof}{\@nproof}}
\def\@oproof[#1][#2]{\trivlist\item[\hskip\labelsep\textit{#2 Proof of\
#1.}~]\ignorespaces}
\def\@nproof{\trivlist\item[\hskip\labelsep\textit{Proof.}~]\ignorespaces}

\makeatother
\let\dotlessi=\i

\begin{document}
\title[Explicit transformations of certain Lambert series]{Explicit transformations of certain Lambert series}
\dedicatory{In memory of Srinivasa Ramanujan}

\author{Atul Dixit}
\address{Discipline of Mathematics, Indian Institute of Technology Gandhinagar, Palaj, Gandhinagar 382355, Gujarat, India} 
\email{adixit@iitgn.ac.in}

\author{Aashita Kesarwani}
\address{Computing and Information Services, Harvey Mudd College, 301 Platt Boulevard, Claremont, CA 91711, USA}
\email{contact@aashitak.com}

\author{Rahul Kumar}
\address{Discipline of Mathematics, Indian Institute of Technology Gandhinagar, Palaj, Gandhinagar 382355, Gujarat, India} 
\curraddr{Center for Geometry and Physics, Institute for Basic Science (IBS), Pohang 37673, Republic of Korea.}
\email{rahul@ibs.re.kr}
\begin{abstract}
An exact transformation, which we call the \emph{master identity}, is obtained for the first time for the series $\sum_{n=1}^{\infty}\sigma_{a}(n)e^{-ny}$ for $a\in\mathbb{C}$ and Re$(y)>0$. New modular-type transformations when $a$ is a non-zero even integer are obtained as its special cases. The precise obstruction to modularity is explicitly seen in these transformations. These include a novel companion to Ramanujan's famous formula for $\zeta(2m+1)$. The Wigert-Bellman identity arising from the $a=0$ case of the master identity is derived too. When $a$ is an odd integer, the well-known modular transformations of the Eisenstein series on $\textup{SL}_{2}\left(\mathbb{Z}\right)$, that of the Dedekind eta function as well as Ramanujan's formula for $\zeta(2m+1)$ are derived from the master identity. The latter identity itself is derived using Guinand's version of the Vorono\"{\dotlessi} summation formula and an integral evaluation of N.~S.~Koshliakov involving a generalization of the modified Bessel function $K_{\nu}(z)$. Koshliakov's integral evaluation is proved for the first time. It is then generalized using a well-known kernel of Watson to obtain an interesting two-variable generalization of the modified Bessel function. This generalization allows us to obtain a new modular-type transformation involving the sums-of-squares function $r_k(n)$. Some results on functions self-reciprocal in the Watson kernel are also obtained.
\end{abstract}
\maketitle
\vspace{-1cm}
\tableofcontents
\vspace{-1.3cm}
\section{Introduction}\label{intro}
The Bessel functions of the first and second kinds of order $\nu$ are defined by \cite[p.~40, 64]{watson-1944a}
\begin{align}
	J_{\nu}(z)&:=\sum_{m=0}^{\infty}\frac{(-1)^m(z/2)^{2m+\nu}}{m!\Gamma(m+1+\nu)} \hspace{9mm} (z,\nu\in\mathbb{C}),\label{sumbesselj}\\
	Y_{\nu}(z)&:=\frac{J_{\nu}(z)\cos(\pi \nu)-J_{-\nu}(z)}{\sin{\pi \nu}}\hspace{5mm}(z\in\mathbb{C}, \nu\notin\mathbb{Z}),\label{ybesse}
	\end{align}
	along with $Y_n(z)=\lim_{\nu\to n}Y_\nu(z)$ for $n\in\mathbb{Z}$. Here $\G(s)$ denotes Euler's gamma function.
The modified Bessel functions of the first and second kinds are defined by \cite[p.~77, 78]{watson-1944a}
\begin{align}
I_{\nu}(z)&:=
\begin{cases}
e^{-\frac{1}{2}\pi\nu i}J_{\nu}(e^{\frac{1}{2}\pi i}z), & \text{if $-\pi<\arg(z)\leq\frac{\pi}{2}$,}\\
e^{\frac{3}{2}\pi\nu i}J_{\nu}(e^{-\frac{3}{2}\pi i}z), & \text{if $\frac{\pi}{2}<\arg(z)\leq \pi$,}
\end{cases}\label{besseli}\\
K_{\nu}(z)&:=\frac{\pi}{2}\frac{I_{-\nu}(z)-I_{\nu}(z)}{\sin\nu\pi}\label{kbesse}
\end{align}
respectively. When $\nu=n\in\mathbb{Z}$, $K_{n}(z)$ is interpreted as the limit $\nu\to n$ of the right-hand side of \eqref{kbesse}.

Several generalizations of the Bessel functions exist in the literature, for example, see \cite{dkmt}, \cite{dk2}, \cite{jekhowsky1}, \cite{jekhowsky2}, \cite{masparpog}, to list a few. A new generalization of the modified Bessel function of the second kind $K_{\nu}(z)$ was obtained by N.~S.~Koshliakov in \cite{kosh1938} through an interesting integral evaluation, and it seems to have gone unnoticed. One of the goals of this paper is to bring it to light. It has an application in obtaining an explicit modular-type transformation for an infinite series involving the generalized divisor function $\sigma_a(n):=\sum_{d|n}d^{a}$, where $a\in\mathbb{C}$. See Theorems \ref{main} and \ref{extendedid}. This is the main focus of this paper, for, this transformation gives, as corollaries, not only important results in the theory of modular forms but also surprising identities that were hitherto unknown. 

In \cite{kosh1938}, Koshliakov studied integrals involving the kernels
\begin{align}
 &\cos(\pi\nu) M_{2\nu}(2 \sqrt{xt}) -
\sin(\pi \nu) J_{2\nu}(2 \sqrt{xt}),\label{kk1}\\
&\sin(\pi\nu) J_{2\nu}(2 \sqrt{xt}) -\cos(\pi\nu) L_{2\nu}(2 \sqrt{xt}),\label{kk2}
\end{align} 
where \begin{align}\label{mldef}
{M_{\nu} }(x) := \frac{2}{\pi }{K_{\nu} }(x) - {Y_{\nu} }(x),\hspace{6mm}
{L_{\nu} }(x) := -\frac{2}{\pi }{K_{\nu} }(x) - {Y_{\nu} }(x).
\end{align}
The special case $\nu=0$ of some of his integrals involving the kernel in \eqref{kk1} were previously obtained by Dixon and Ferrar \cite{dixfer3}.

In the same paper \cite[Equation (8)]{kosh1938}, Koshliakov proved a remarkable result that $K_{\nu}(t)$ is self-reciprocal in the kernel \eqref{kk1}, that is, for $-\tfrac{1}{2}<\nu<\tfrac{1}{2}$,
\begin{equation}\label{koshlyakov-1}
\int_{0}^{\infty} K_{\nu}(t) \left( \cos(\pi\nu) M_{2\nu}(2 \sqrt{xt}) -
\sin(\pi\nu) J_{2\nu}(2 \sqrt{xt}) \right)\, dt = K_{\nu}(x).
\end{equation}
It is easy to see though that this identity is valid for complex $\nu$ such that $-\tfrac{1}{2}<\textup{Re}(\nu)<\tfrac{1}{2}$. Prior to this, the special case $\nu=0$ of \eqref{koshlyakov-1} was stated and proved by Dixon and
Ferrar \cite[p.~164, Equation (4.1)]{dixfer3} though they surmise in their paper that although Koshliakov does not give a reference to the formula itself, it is clear that he is familiar with the result.

In \cite[p.~897]{bdrz}, the kernel in \eqref{kk1} is called the first Koshliakov kernel, and the integral 
\begin{equation*}
\int_{0}^{\infty}f(t, \nu)\left(\cos(\pi\nu) M_{2\nu}(2 \sqrt{xt}) -\sin(\pi \nu) J_{2\nu}(2 \sqrt{xt})\right)\, dt,
\end{equation*}
the first Koshliakov transform of $f(t, \nu)$. It is important in the context of the Vorono\"{\dotlessi} summation formula for $\sigma_a(n)$, see for example \cite{bdrz}. It must be mentioned here that there are few functions in literature whose Koshliakov transforms have closed-form evaluations. However, it is always desirable to have such closed form evaluations, whenever possible, in view of their applications in number theory. Dixon and Ferrar \cite[p.~161]{dixfer3} evaluated such integrals for the special case $\nu=0$ of the Koshliakov kernel, and recently further such evaluations were obtained in \cite{dkmt}, \cite{dk2} and \cite{dixitroy1} along with their applications given.

In the same paper \cite[Equation (15)]{kosh1938}, Koshliakov gives a more general result\footnote{Koshliakov inadvertently missed the factor $\pi$ in front on the right-hand side.} of which \eqref{koshlyakov-1} is a special case, that is, for\footnote{It will be shown later that when $\mu\neq-\nu$, this result actually holds for $\nu\in\mathbb{C}\backslash\left(\mathbb{Z}\backslash\{0\}\right)$, $\textup{Re}(\mu)>-1/2$, $\textup{Re}(\nu)>-1/2$ and $\textup{Re}(\mu+\nu)>-1/2$; otherwise, it holds for $-1/2<\textup{Re}(\nu)<1/2$.} $\mu>-1/2$ and $\nu>-\tfrac{1}{2}+|\mu|$, 
\begin{align}\label{koshlyakovg-1}
&\int_{0}^{\infty} K_{\mu}(t)t^{\mu+\nu} \left( \cos(\pi\nu) M_{2\nu}(2 \sqrt{xt}) -
\sin(\pi\nu) J_{2\nu}(2 \sqrt{xt}) \right)\, dt \nonumber\\
&=\frac{\pi 2^{\mu+\nu-1}}{\sin(\nu\pi)}\bigg\{\left(\frac{x}{2}\right)^{-\nu}\frac{\Gamma(\mu+\tfrac{1}{2})}{\Gamma(1-\nu)\Gamma(\tfrac{1}{2}-\nu)}\pFq12{\mu+\tfrac{1}{2}}{\tfrac{1}{2}-\nu,1-\nu}{\frac{x^2}{4}}\nonumber\\
&\quad\quad\quad\quad\quad\quad-\left(\frac{x}{2}\right)^{\nu}\frac{\Gamma(\mu+\nu+\tfrac{1}{2})}{\Gamma(1+\nu)\Gamma(\tfrac{1}{2})}\pFq12{\mu+\nu+\tfrac{1}{2}}{\tfrac{1}{2},1+\nu}{\frac{x^2}{4}}\bigg\},
\end{align}
where, for $b_j\notin\mathbb{Z}^{-}\cup\{0\}, 1\leq j\leq q$, and $(a)_n:=a(a+1)\cdots(a+n-1)=\G(a+n)/\G(a)$,
\begin{equation}\label{pfqgen}
\pFq{p}{q}{a_1, a_2, \cdots, a_p}{b_1,b_2, \cdots, b_q}{z}:=\sum_{n=0}^{\infty}\frac{(a_1)_n(a_2)_n\cdots(a_p)_n}{(b_1)_n(b_2)_n\cdots(b_q)_n}\frac{z^n}{n!}
\end{equation}
is the generalized hypergeometric series. It is well-known \cite[p.~62, Theorem 2.1.1]{aar} that the above series converges absolutely for all $z$ if $p\leq q$ and for $|z|<1$ if $p=q+1$, and it diverges for all $z\neq0$ if $p>q+1$ and the series does not terminate.

Indeed, letting $\mu=-\nu$ in \eqref{koshlyakovg-1} gives \eqref{koshlyakov-1} as the right-hand side reduces to
\begin{align}\label{kdef}
&\frac{\pi}{2\sin(\nu\pi)}\left\{\frac{(x/2)^{-\nu}}{\G(1-\nu)}\pFq01{-}{1-\nu}{\frac{x^2}{4}}-\frac{(x/2)^{\nu}}{\G(1+\nu)}\pFq01{-}{1+\nu}{\frac{x^2}{4}}\right\}\nonumber\\
&=K_{\nu}(x)\hspace{6mm} (\text{using}\hspace{1mm}\eqref{sumbesselj}, \eqref{besseli}\hspace{1mm}\text{and}\hspace{1mm}\eqref{kbesse}).
\end{align}
Thus, either side of \eqref{koshlyakovg-1} can be conceived to be a one-variable generalization of $K_{\nu}(x)$. 

Equation \eqref{koshlyakovg-1} is the identity of Koshliakov \cite{kosh1938} whose importance seems to have gone unnoticed. Koshliakov does not prove the identity but only indicates that it can be obtained by considering 
$\int_{0}^{\infty}\left(\sin(\nu\pi)J_{\nu}(xu)+\cos(\nu\pi)Y_{\nu}(xu)\right)\frac{u^{2\mu+\nu}\, du}{(u^2+1)^{\mu+\nu+1}}$.
He does not even evaluate the above integral, unlike its special case $\mu=-\nu$, for which he \cite[Equation (7)]{kosh1938} gives the evaluation as $-K_{\nu}(x)$, again without proof. Since \eqref{koshlyakovg-1} has never been proved in the literature, we do so by deriving it as a special case of a more general result. See Corollary \ref{koshlyakovfir} of Section \ref{proof}. Its number-theoretic application is given in Section \ref{mr}.

Our generalization of Koshliakov's result \eqref{koshlyakovg-1} naturally leads us to a new two-variable generalization of the modified Bessel function. The latter allows us to derive a novel modular-type transformation involving $r_k(n)$, the number of representations of $n$ as a sum of $k$ squares, where representations with different orders or signs of the summands are regarded as distinct, i.e.,
\begin{equation}\label{rkndef}
r_k(n):=\#\{(a_1,a_2,\cdots, a_k)\in\mathbb{Z}^{k}: n=a_1^2+a_2^2+\cdots+a_k^2\}.
\end{equation}
See Theorem \ref{rk(n)identity} below.

The two-variable generalization of $K_{\nu}(x)$ is achieved by first obtaining a common extension of the first and the second Koshliakov kernels in \eqref{kk1} and \eqref{kk2} respectively. This common extension is in terms of the hyper-Bessel functions ${}_0F_{3}$ whose theory, in the general case, that is for ${}_0F_{n}$, was initiated by Delerue in a series of five papers \cite{delerue1}-\cite{delerue5}. The hyper-Bessel functions have been found useful in many applications, for example, they are used to understand the wavefields and the elements of the non-adiabatic transition matrix and the tunnelling loss matrix \cite{witte}.

Let $x>0$, $\nu\in\mathbb{C}\backslash\left(\mathbb{Z}\backslash\{0\}\right)$ and $w\in\mathbb{C}$. We define the generalization of the two Koshliakov kernels in \eqref{kk1} and \eqref{kk2} by
 \begin{multline}\label{genkoshk}
\mathscr{G}_{\nu}(x,w) 
:= \frac{\pi}{\sin(\nu \pi)} \left(\frac{x}{4}\right)^w\\
\times\Bigg[ \left(\frac{x}{4}\right)^{-\nu} \frac{1}{\Gamma(1-\nu)\Gamma\left(w+\frac{1}{2}\right) \Gamma\left(w+\frac{1}{2}-\nu\right)} 
\pFq03{-}{1-\nu, w+\frac{1}{2}, w+\frac{1}{2}-\nu}{\frac{x^2}{16} }\\
-\left(\frac{x}{4}\right)^{\nu} \frac{1}{\Gamma(1+\nu)\Gamma\left(w+\frac{1}{2}\right) \Gamma\left(w+\frac{1}{2}+\nu\right)} 
\pFq03{-}{1+\nu, w+\frac{1}{2}, w+\frac{1}{2}+\nu}{\frac{x^2}{16}}
\Bigg].
\end{multline}
The two expressions inside the square brackets in \eqref{genkoshk} are entire functions of\footnote{One might as well take $x$ in the definition of $\mathscr{G}_{\nu}(x, w)$ to be such that $-\pi<\arg(x)<\pi$, thereby having analyticity in $x$ as well. However, in this paper, we will be working with $x>0$ only.} $\nu$ and $w$. As a function of $\nu$, $\mathscr{G}_{\nu}(x, w)$ has a simple pole at every non-zero integer and a removable singularity at $0$.

The kernel $\mathscr{G}_{\nu}(x, w)$ is not new. In fact, it is a special case of the well-known kernel introduced by Watson in \cite{watsonself}, namely,
\begin{equation*}
\varpi_{\mu, \nu}(xy)=A(xy)^{1/2}\int_{0}^{\infty}J_{\nu}(xt)J_{\mu}\left(\frac{Ay}{t}\right)\frac{dt}{t}.
\end{equation*}
Watson formally studied how this kernel gives rise to a transform. Letting $A=y=1$ in the above kernel, using Hanumanta Rao's formula \cite{watsonself} with $\rho=0, a=x$, and $b=1$ to express $\varpi_{\mu, \nu}(x)$ in the form of two ${}_0F_{3}s$, replacing $\nu$ by $w+\nu-\frac{1}{2}$ and then $\mu$ by $w-\nu-\frac{1}{2}$ in the resulting identity, we see that for $|\textup{Re}(\nu)|<\textup{Re}(w)+1$,
\begin{equation}\label{kernelequi}
\mathscr{G}_{\nu}(x, w)=\varpi_{w-\nu-\frac{1}{2}, w+\nu-\frac{1}{2}}(x).
\end{equation}
However, note that the definition of the function $\mathscr{G}_{\nu}(x, w)$ in \eqref{genkoshk} itself makes sense for any $\nu\in\mathbb{C}\backslash\left(\mathbb{Z}\backslash\{0\}\right)$ and $w\in\mathbb{C}$. Henceforth, we call $\mathscr{G}_{\nu}(x, w)$ as the Watson kernel only.

Observe that $\varpi_{\mu, \nu}(xy)$ is the resultant\footnote{For the definition of a resultant of two kernels, see \cite{hardyresultant}.} of the kernels $\sqrt{t}J_{\nu}(t)$ and $t^{-3/2}J_{\mu}(1/t)$, and that a simple change of variable shows that 
\begin{equation}\label{symmetry}
\varpi_{\mu, \nu}(xy)=\varpi_{\nu, \mu}(xy).
\end{equation}

The literature on Watson's kernel is vast. Bhatnagar \cite{bhatnagarbelgique}, \cite{bhatnagar1954f} obtained a necessary and sufficient condition for a function in a certain class $A(\omega, a)$ considered by Hardy and Titchmarsh \cite{hartitch} to be self-reciprocal in $\varpi_{\mu, \nu}(xy)$. Among other things, he \cite[Theorem 8]{bhatnagarbelgique} also obtained a theorem where functions reciprocal in the Hankel transform give rise to other functions reciprocal in the Watson kernel and vice-versa. See also \cite[Theorem 7]{bhatnagar1954f}. In the same paper \cite{bhatnagar1954f}, Bhatnagar also obtained the inverse Mellin transform representation of $\varpi_{\mu, \nu}(x)$ although in Theorem \ref{gkkmellin} below, we re-derive it with maximum possible domain where it is valid. Watson's kernel was further generalized by him in \cite{bhatnagarbelgique} and \cite{bhatnagar1954}, and in a different direction by Olkha and Rathie in \cite{olkharathie}. An expansion of the generalized Watson kernel $\varpi_{\mu_1, \cdots, \mu_n}(x)$ in terms of Jacobi polynomials was obtained by Dahiya \cite{dahiyajapan}. A representation of $\varpi_{\mu, \nu}(x)$ (more generally for $\varpi_{\mu_1, \cdots, \mu_n}(x)$) in terms of the Meijer-$G$ function was given by Narain \cite[Section 4(b)]{narain}.

That $\varpi_{\mu, \nu}(x)=O(x^{-1/4})$ as $x\to\infty$ was established by Watson \cite[p.~308]{watsonself}. Mainra and Singh \cite{mainrasingh} derived the differential equation satisfied by $\varpi_{\mu, \nu}(x)$ and then used it to obtain its asymptotic expansion as $x\to\infty$. Numerous integrals containing $\varpi_{\mu, \nu}(x)$ in its integrand have been evaluated by Singh in \cite{singhrajasthan}, \cite{singhgerman}.

In spite of so much work done on Watson's kernel, its importance from the point of view of number theory has not been recognized before. 
For example, while it is known \cite[p.~5, Example 3]{hardyresultant} that for $x>0$,
\begin{equation}\label{gkkhalf}
\mathscr{G}_{\frac{1}{2}}(x,w)=J_{2w-1}(2\sqrt{x}),
\end{equation}
it has not been noticed before that when $w=0$ and $1$, the kernel $\varpi_{w-\nu-\frac{1}{2}, w+\nu-\frac{1}{2}}(x)$, or equivalently $\mathscr{G}_{\nu}(x, w)$, reduces respectively to the first and second Koshliakov kernels defined in \eqref{kk1} and \eqref{kk2}. See Theorem \ref{gkkkoshliakov} of Section \ref{genknow}. Earlier, this was known \cite[Section 4]{mainrasingh} only when we additionally consider $\nu=0$. These kernels are useful in number theory.

Our generalization of Koshliakov's modified Bessel function is defined for $\nu\in\mathbb{C}\backslash\left(\mathbb{Z}\backslash\{0\}\right)$, and $z, \mu, w\in\mathbb{C}$ such that $\mu+w\neq-\frac{1}{2}, -\frac{3}{2}, -\frac{5}{2},\cdots$, by
\begin{align}\label{def2varbessel}
{}_{\mu}K_{\nu}(z, w)&:=\frac{\pi z^w 2^{\mu+\nu-1}}{\sin(\nu\pi)}\bigg\{\left(\frac{z}{2}\right)^{-\nu}\frac{\Gamma(\mu+w+\tfrac{1}{2})}{\Gamma(1-\nu)\Gamma(w+\tfrac{1}{2}-\nu)}\pFq12{\mu+w+\tfrac{1}{2}}{w+\tfrac{1}{2}-\nu,1-\nu}{\frac{z^2}{4}}\nonumber\\
&\quad\quad\quad\quad\quad\quad-\left(\frac{z}{2}\right)^{\nu}\frac{\Gamma(\mu+\nu+w+\tfrac{1}{2})}{\Gamma(1+\nu)\Gamma(w+\tfrac{1}{2})}\pFq12{\mu+\nu+w+\tfrac{1}{2}}{w+\tfrac{1}{2},1+\nu}{\frac{z^2}{4}}\bigg\},
\end{align}
with ${}_{\mu}K_{0}(z, w)=\lim_{\nu\to0}{}_{\mu}K_{\nu}(z, w)$.
When $w=0$, ${}_{\mu}K_{\nu}(z, w)$ reduces to \emph{Koshliakov's generalized modified Bessel function of the second kind} whose special case when $z$ is a positive real number $x$ occurs on the right-hand side of \eqref{koshlyakovg-1}. Clearly, from \eqref{kdef} and \eqref{def2varbessel}, 
\begin{equation}\label{mueq-nu}
{}_{-\nu}K_{\nu}(z, w)=z^{w}K_{\nu}(z).
\end{equation}






\section{Main results}\label{mr}

Our first result is a generalization of Koshliakov's identity \eqref{koshlyakovg-1}. It gives an integral representation for ${}_{\mu}K_{\nu}(z, w)$. 
\begin{theorem}\label{koshlyakovg-wthm}
Let $x>0$, $\textup{Re}(w)>-1/2$ and $\nu\in\mathbb{C}\backslash\left(\mathbb{Z}\backslash\{0\}\right)$. Let $\mathscr{G}_{\nu}(x, w)$ be defined in \eqref{genkoshk}. If $\mu\neq-\nu$, then for $\textup{Re}(\mu), \textup{Re}(\nu), \textup{Re}(\mu+\nu)>-\textup{Re}(w)-\frac{1}{2}$, we have
\begin{align}\label{koshlyakovg-w}
\int_{0}^{\infty} K_{\mu}(t)t^{\mu+\nu+w} \mathscr{G}_{\nu}(xt, w)  \, dt={}_{\mu}K_{\nu}(x, w),
\end{align}
otherwise, for $-\textup{Re}(w)-\frac{1}{2}<\textup{Re}(\nu)<\textup{Re}(w)+\frac{1}{2}$, we have
\begin{align}\label{selff}
\int_{0}^{\infty}t^{w}K_{\nu}(t)\mathscr{G}_{\nu}(xt, w)\, dt=x^wK_{\nu}(x).
\end{align}
\end{theorem}
The above integral evaluations are proved using the theory of Mellin transforms, in particular, Parseval's formula \eqref{parseval-1}. Equations \eqref{koshlyakovg-w} and \eqref{selff}, and the special case \eqref{koshlyakovg-1} due to Koshliakov, are crucial in obtaining the main theorems of this section. 

It must be noted that the second part of Theorem \ref{koshlyakovg-wthm}, that is \eqref{selff}, is not new. In fact, Bhatnagar \cite[Section 3, (1) (i)]{bhatnagar1954} obtained it as a special case of a more general identity proved by him in \cite{bhatnagar1954f}, but only for real values of $w$ and $\nu$. Varma \cite[p.~103, Example 2]{varma} considered the special case $w=\nu$ of \eqref{selff} whereas Mitra \cite[Example 2]{mitra} did the same for $w=1/2$. However, to the best of our knowledge, \eqref{koshlyakovg-w} has not appeared in the literature before.

We now give a new modular-type transformation between two infinite series involving $r_k(n)$ as an application of Theorem \ref{koshlyakovg-wthm}.
\begin{theorem}\label{rk(n)identity}
Let $k\in\mathbb{N},k\geq2$ and $\textup{Re}(z)>0$. Let $r_k(n)$ be defined in \eqref{rkndef}. Define $\mathfrak{R}(\mu, k, z)$ by
\begin{equation}\label{rmuk}
\mathfrak{R}(\mu, k, z):=\begin{cases}
0,\hspace{22mm}\text{if}\hspace{1mm}\textup{Re}(\mu)>-\frac{1}{2},\\
\frac{1}{\sqrt{2z}}\pi^{\frac{1-k}{2}}\Gamma\left(\frac{k}{2}\right),\hspace{1mm}\text{if}\hspace{1mm}\mu=-\frac{1}{2}.
\end{cases}
\end{equation}
Then for $\mathrm{Re}(\mu)>-\frac{1}{2}$ or $\mu=-1/2$, the following transformation holds:
\begin{align}\label{rk(n)identityeqn}
&\sum_{n=1}^\infty r_k(n)n^{\mu+1}K_{\mu}(nz)-\frac{\pi^{\frac{k+1}{2}}2^{\mu}\Gamma\left(\mu+\frac{k}{4}+\frac{1}{2}\right)}{z^{\mu+\frac{k}{2}+1}\Gamma\left(\frac{k}{4}\right)}\nonumber\\
&=\frac{\pi}{z^{\mu+\frac{k}{4}+\frac{3}{2}}}\sum_{n=1}^\infty r_k(n)n^{\frac{1}{2}-\frac{k}{4}}{}_\mu K_{\frac{1}{2}}\left(\frac{\pi^2n}{z},\frac{k}{4}\right)-\frac{\pi^{\frac{k}{2}}}{\G\left(\frac{k}{2}\right)}\mathfrak{R}(\mu, k, z).
\end{align}
\end{theorem}
There are a number of transformations involving infinite series of $r_k(n)$ and modified Bessel functions in the literature. See, for example, \cite[Equation (4)]{watsonself}, \cite[Equation (3.12)]{dixfer2} or its generalization \cite[Corollary 4.6]{bdkz}. However, a transformation for the series $\sum_{n=1}^\infty r_k(n)n^{\mu+1}K_{\mu}(nz)$, given in Theorem \ref{rk(n)identity}, is obtained for the first time.

Theorem \ref{rk(n)identity} gives, as a special case, the well-known modular transformation \cite[Equation (64)]{cn} recorded below.
\begin{corollary}\label{rkncor}
Let $k$ be an integer such that $k\geq2$ and let $\textup{Re}(z)>0$. Then
\begin{equation*}
\sum_{n=0}^{\infty}r_k(n)e^{-nz}=\left(\frac{\pi}{z}\right)^{k/2}\sum_{n=0}^{\infty}r_k(n)e^{-\pi^2 n/z}.
\end{equation*}
\end{corollary} 
Let $a\in\mathbb{C}$. Consider the series 
\begin{equation}\label{equivalence}
\sum_{n=1}^{\infty}\sigma_a(n)e^{-ny}=\sum_{n=1}^{\infty}\frac{n^a}{e^{ny}-1}.
\end{equation}
When $a=2m+1, m\in\mathbb{N}$ and $y=-2\pi iz$ (so that $z\in\mathbb{H}$, the upper half-plane), either of the above series essentially represents the Eisenstein series of weight $2m+2$ on the full modular group $\textup{SL}_{2}\left(\mathbb{Z}\right)$. When $a=-2m-1, m\in\mathbb{N}$, the series $\sum_{n=1}^{\infty}\sigma_a(n)e^{2\pi inz}$ represents the Eichler integral corresponding to the weight $2m+2$ Eisenstein series \cite[Section 5]{berndtstraubzeta}. Moreover, while $\sum_{n=1}^{\infty}\sigma_{1}(n)e^{2\pi inz}$ is essentially the quasi-modular form $E_2(z)$, that is, the Eisenstein series of weight $2$, the series $\sum_{n=1}^{\infty}\sigma_{-1}(n)e^{2\pi inz}$ is what appears in the modular transformation of logarithm of the Dedekind eta function $\eta(z)$ \cite[Equation (3.10)]{bls}. 

In all of the above cases, that is, when $a$ is an odd integer, the series $\sum_{n=1}^\infty\sigma_a(n)e^{-ny}$ satisfies a modular transformation, and hence plays a fundamental role in the theory of modular forms. However, an explicit transformation for the series $\sum_{n=1}^\infty\sigma_a(n)e^{-ny}$ when $a$ is an \emph{even integer} is conspicuously absent from the literature except in the special case $a=0$; see, for example, the paper by Kanemitsu, Tanigawa and Tsukada \cite[Theorem 2]{kanemitsu} or \cite[p.~33]{ober1}, \cite[Equation (3.7)]{yak1}. Using the concept of transseries \cite{transseries}, Dorigoni and Kleinschmidt \cite{dorigoni} have considered the case when $a$ is \emph{negative} even integer (see \cite[Equation (2.43)]{dorigoni}). However, as special cases of our master identity (see Theorem \ref{extendedid} below), we not only obtain the identity for negative even integer $a$ (see Theorem \ref{trans-2m} below) but also when $a$ is a non-negative even integer.

For Re$(a)>2$, Ramanujan \cite[p.~416]{berndt1} obtained a transformation for the series in \eqref{equivalence} which shows how the modularity is obstructed by means of the appearance of an extra integral. Later, Guinand \cite[Section 7]{guinand1944} also studied the series $\sum_{n=1}^\infty\sigma_{2k}(n)e^{-ny}$ for $k\in\mathbb{R},$ such that $k+\frac{1}{2}\notin\mathbb{Z}$ and 
obtained \cite[Theorem 8]{guinand1944} similar results for $\textup{Re}(y)>0$. However, these transformations are not explicit in the sense that apart from the series getting transformed and the residual terms, certain extra objects appear, for example, contour integrals or principal value integrals etc. 

Berndt \cite[Theorem 2.2]{berndtrocky} derived an important modular-type transformation for a vast generalization of Eisenstein series. We rewrite one case of this transformation in the version given by O'Sullivan \cite[Equation (1.23)]{sullivan}, namely, for $z\in\mathbb{H}$ and $k\in\mathbb{Z}$,
{\allowdisplaybreaks\begin{align}\label{br}
&z^k(1+(-1)^k)\sum_{n=1}^{\infty}\sigma_{k-1}(n)e^{2\pi inz}-(1+(-1)^k)\sum_{n=1}^{\infty}\sigma_{k-1}(n)e^{2\pi in(-1/z)}\nonumber\\
&=(2\pi i)^{1-k}\sum_{u, v\in\mathbb{Z}_{\geq 0}\atop u+v=2-k}\frac{B_u}{u!}\frac{B_v}{v!}z^{1-v}-\begin{cases}
\pi i-\log z\hspace{12mm}\text{if}\hspace{1mm}k=0;\\
(z^{k}-(-1)^{k})\zeta(1-k)\hspace{5mm}\text{if}\hspace{1mm}k\neq0,
\end{cases}
\end{align}}%
where $\zeta(s)$ is the Riemann zeta function. When $k=1-2m, m\in\mathbb{Z}$, it results in
\begin{equation}\label{zeta(2m)}
\zeta(2 m ) = \begin{cases}
(-1)^{m +1} \displaystyle\frac{(2\pi)^{2 m}B_{2 m }}{2 (2 m)!},\hspace{2mm}\text{if}\hspace{1mm} m\geq 0,\\
0, \hspace{35mm}\text{if}\hspace{1mm} m<0,
\end{cases}
\end{equation}
thus incorporating Euler's formula \cite[p.~5, Equation (1.14)]{temme} whereas on letting $k=-2m,m\in\mathbb{Z}\backslash\{0\}$, it gives Ramanujan's formula for $\zeta(2m+1)$ \cite[p.~173, Ch. 14, Entry 21(i)]{ramnote}, \cite[p.~319-320, formula (28)]{lnb}, \cite[p.~275-276]{bcbramsecnote}, namely, for\footnote{Ramanujan's formula is actually valid for any complex $\a, \b$ such that $\textup{Re}(\a)>0, \textup{Re}(\b)>0$ and $\a\b=\pi^2$.} $\alpha,\ \beta>0$ with $\alpha\beta=\pi^2$, 
\begin{align}\label{rameqn}
\alpha^{-m}\left\{\frac{1}{2}\zeta(2m+1)+\sum_{n=1}^\infty \frac{n^{-2m-1}}{e^{2n\alpha}-1}\right\}&=(-\beta)^{-m}\left\{\frac{1}{2}\zeta(2m+1)+\sum_{n=1}^\infty\frac{n^{-2m-1}}{e^{2n\beta}-1}\right\}\nonumber\\
&\qquad-2^{2m}\sum_{k=0}^{m+1}\frac{(-1)^{k}B_{2k}B_{2m+2-2k}}{(2k)!(2m+2-2k)!}\alpha^{m+1-k}\beta^k.
\end{align}
Note, however, that \eqref{br} does not give any transformation for $\sum_{n=1}^{\infty}\sigma_{-2m}(n)e^{-ny}$, $m\in\mathbb{Z}$.

Ramanujan's formula encodes the transformation properties of Eisenstein series on $\textup{SL}_2\left(\mathbb{Z}\right)$ as well as their Eichler integrals \cite{berndtstraubzeta}. Guinand \cite[Theorem 9 (iv)]{guinand1944} rediscovered the above formula. 

Let $y=\log(1/q)$ and let $\mathcal{P}$ denote the set of integer partitions. Then for functions $f:\mathcal{P}\to\mathbb{C}$, the $q$-brackets are the power series
\begin{align*}
\langle f\rangle_q:=\frac{\sum_{\lambda\in\mathcal{P}}f(\lambda)q^{|\lambda|}}{\sum_{\lambda\in\mathcal{P}}q^{|\lambda|}}\in\mathbb{C}[[q]],
\end{align*}
which represent ``weighted average" of $f$. Also let $\mathcal{H}_t(\l)$ be the multiset of partition hook numbers which are multiples of $t$, where the hook number of a hook in the Ferrers diagram of a partition is the number of dots in its arm and leg together. Recently Han \cite{han1}, Han and Ji \cite{han2} showed that the series $\sum_{n=1}^{\infty}\sigma_{1-a}(n)e^{-ny}$ are the $q$-brackets of $f_{a,t}(\l)$, where
\begin{equation*}
f_{a, t}(\l):=t^{a-1}\sum_{h\in\mathcal{H}_t(\l)}\frac{1}{h^{a}}.
\end{equation*}
Bringmann, Ono and Wagner \cite{ono} showed that for even $a\geq2$, the series $\sum_{n=1}^{\infty}\sigma_{1-a}(n)e^{-ny}$ are pieces of weight $2-a$ sesquiharmonic and harmonic Maass forms, whereas for odd $a\leq-1$, they are holomorphic quantum modular forms. For more details, see \cite[Theorems 1.2, 1.6]{ono}.

The series $\sum_{n=1}^{\infty}d(n)e^{-ny}$, or, in general, $\sum_{n=1}^{\infty}\sigma_a(n)e^{-ny}$, where $a\in\mathbb{C}$, is also instrumental in the study of moments of the Riemann zeta function and have long been employed for this purpose. See, for example, \cite[p.~163, Theorem 7.15]{titch} or \cite{laurincikas}. Lukkarinen \cite{mari} obtained meromorphic continuation of the modified Mellin transform of $\left|\zeta(1/2+ix)\right|^{2}$, defined for $\textup{Re}(s)>1$, by
\begin{equation*}
\int_{1}^{\infty}\left|\zeta\left(\frac{1}{2}+ix\right)\right|^{2}x^{-s}\, dx,
\end{equation*}
by using a transformation for $\sum_{n=1}^{\infty}d(n)e^{-ny}$. Period functions, in general, are real analytic functions $\Psi(x)$ satisfying the three-term recurrence relations $\Psi(x)=\Psi(x+1)+(x+1)^{-2s}\Psi\left(\frac{x}{x+1}\right)$, where $s=1/2+it$.  In an interesting paper, Bettin and Conrey \cite{betcon} studied the period function of the series $\sum_{n=1}^{\infty}\sigma_a(n)e^{2\pi inz}$, where $a\in\mathbb{C}$ and $z\in\mathbb{H}$, showing that it can be analytically continued to $|\arg(z)|<\pi$, and as an application of their result they gave a simple proof of the Vorono\"{\dotlessi} summation formula
as well as an exact asymptotic expansion for the smoothly weighted second moment of the Riemann zeta function on the critical line, namely,
\begin{align*}
\int_0^\infty\left|\zeta\left(\frac{1}{2}+it\right)\right|^2e^{-\delta t}\ dt,
\end{align*}
as $\delta\to0$, in the form of a \emph{convergent} asymptotic series.

However, in the above results, either an asymptotic estimate for the series $\sum_{n=1}^{\infty}\sigma_a(n)e^{2\pi inz}$ is given or a transformation involving a line integral, and hence an explicit transformation is missing. We fill this gap and obtain an explicit transformation for the series $\sum_{n=1}^{\infty}\sigma_a(n)e^{-ny}$ first, for any $a\in\mathbb{C}$ such that $\textup{Re}(a)>-1$, and then for Re$(a)>-2m-3$, where $m\in\mathbb{N}\cup\{0\}$, by analytic continuation. We then obtain, as corollaries, not only the well-known results in the theory of modular forms but also a transformation between $\sum_{n=1}^{\infty}\sigma_{2m}(n)e^{-ny}$ and the series 
\begin{equation}\label{seriesimp}
\sum_{n=1}^\infty\sigma_{2m}(n)\Bigg\{\sinh\left(\frac{4\pi^2n}{y}\right)\mathrm{Shi}\left(\frac{4\pi^2n}{y}\right)-\cosh\left(\frac{4\pi^2n}{y}\right)\mathrm{Chi}\left(\frac{4\pi^2n}{y}\right)+\sum_{j=1}^m(2j-1)!\left(\frac{4\pi^2n}{y}\right)^{-2j}\Bigg\},
\end{equation}
where $m$ is any integer and Re$(y)>0$. Note, of course, that when $m$ is a non-positive integer, the finite sum over $j$ in the summand of \eqref{seriesimp} is empty. Here, the functions $\mathrm{Shi}(z)$ and $\mathrm{Chi}(z)$ are the \emph{hyperbolic sine and cosine integrals} defined by \cite[p.~150, Equation (6.2.15), (6.2.16)]{nist}
\begin{align}\label{shichi}
\mathrm{Shi}(z):=\int_0^z\frac{\sinh(t)}{t}\ dt,\hspace{3mm}
\mathrm{Chi}(z):=\gamma+\log(z)+\int_0^z\frac{\cosh(t)-1}{t}\ dt,
\end{align}
where $\g$ is Euler's constant. (Observe that $\mathrm{Shi}'(z)=\sinh(z)/z$ and $\mathrm{Chi}'(z)=\cosh(z)/z$.) We note that $\textup{Shi}(z)$ is an entire function of $z$ whereas $\textup{Chi}(z)$ is analytic in $\mathbb{C}\backslash(-\infty,0]$. The series in \eqref{seriesimp} is a natural analogue of the series
\begin{equation*}
\sum_{n=1}^\infty\sigma_{2m+1}(n)\Bigg\{\sinh\left(\frac{4\pi^2n}{y}\right)-\cosh\left(\frac{4\pi^2n}{y}\right)\Bigg\}=-\sum_{n=1}^\infty\sigma_{2m+1}(n)e^{-4\pi^2n/y},
\end{equation*}
which appears in the corresponding modular transformation satisfied by $\sum_{n=1}^{\infty}\sigma_{2m+1}(n)e^{-ny}$.

Let $\Phi_{k, \ell}(z):=\sum_{n=1}^{\infty}n^{k}\sigma_{\ell-k}(n)q^{n}$ be a function considered by Ramanujan \cite{ram1916}. In \cite[Corollary 1]{bs}, Bhand and Shankhadhar recently showed that for $k+\ell$ even and any $c\in\mathbb{C}$, the function $c+\Phi_{k,\ell}(z)$ is not a quasimodular form of any weight and depth (and hence, obviously, not a modular form). Note that when $k=0$, this gives $\sum_{n=1}^{\infty}\sigma_{\ell}(n)q^n$ for $\ell$ even. However, as expressed in the concluding remarks of \cite{bs}, no transformation property for the latter series is known. In Theorem \ref{a=2mcase} below, we obtain the precise transformation that this series satisfies. It shows explicitly the obstruction to modularity.

When $a=2m$, $m\in\mathbb{Z}^{-}$, or $a=0$, we also obtain simplified versions of the transformations for the series in \eqref{seriesimp}. This is achieved using a recent result \cite[Theorem 2.2]{dgkm}, valid for $\textup{Re}(u)>0$:
\begin{align}\label{dgkmresult}
\sum_{n=1}^\infty\int_0^\infty\frac{t\cos(t)}{t^2+n^2u^2}\ dt=\frac{1}{2}\left\{\log\left(\frac{u}{2\pi}\right)-\frac{1}{2}\left(\psi\left(\frac{iu}{2\pi}\right)+\psi\left(-\frac{iu}{2\pi}\right)\right)\right\},
\end{align}
where $\psi(z):=\G'(z)/\G(z)$ is the logarithmic derivative of $\G(z)$. See \eqref{kanot} and \eqref{equiforma=-2mano} for the same.

Our transformation for the series $\sum_{n=1}^\infty  \sigma_a(n)e^{-ny}$ is now given. It is proved using Guinand's version of the Vorono\"{\dotlessi} summation formula (see Theorem \ref{guinandSummationFormula} below).
\begin{theorem}\label{main}
Let $\textup{Re}(y)>0$. For $\textup{Re}(a)>-1$, the following transformation holds:
\begin{align}\label{maineqn}
&\sum_{n=1}^\infty  \sigma_a(n)e^{-ny}+\frac{1}{2}\left(\left(\frac{2\pi}{y}\right)^{1+a}\mathrm{cosec}\left(\frac{\pi a}{2}\right)+1\right)\zeta(-a)-\frac{1}{y}\zeta(1-a)\nonumber\\
&=\frac{2\pi}{y\sin\left(\frac{\pi a}{2}\right)}\sum_{n=1}^\infty \sigma_{a}(n)\Bigg(\frac{(2\pi n)^{-a}}{\Gamma(1-a)} {}_1F_2\left(1;\frac{1-a}{2},1-\frac{a}{2};\frac{4\pi^4n^2}{y^2} \right) -\left(\frac{2\pi}{y}\right)^{a}\cosh\left(\frac{4\pi^2n}{y}\right)\Bigg).
\end{align}
\end{theorem}
\begin{remark}\label{sigmatrans}
With ${}_\mu K_{\nu}(z,w)$ defined in \eqref{def2varbessel}, the right-hand side of the above transformation is nothing but $
2\sqrt{2\pi}y^{-1-\frac{a}{2}}\sum_{n=1}^{\infty}\sigma_a(n)n^{-\frac{a}{2}}{}_{\frac{1}{2}}K_{\frac{a}{2}}\left(\frac{4\pi^2n}{y},0\right)$.
\end{remark}

Through analytic continuation, the identity in Theorem \ref{main} can be extended to the half-plane $\mathrm{Re}(a)>-2m-3$, where $m$ is any non-negative integer.  
\begin{theorem}\label{extendedid}
Let $m\in\mathbb{N}\cup\{0\}$ and $\textup{Re}(y)>0$. Then for $\mathrm{Re}(a)>-2m-3$, we have
{\allowdisplaybreaks\begin{align}\label{extendedideqn}
&\sum_{n=1}^\infty  \sigma_a(n)e^{-ny}+\frac{1}{2}\left(\left(\frac{2\pi}{y}\right)^{1+a}\mathrm{cosec}\left(\frac{\pi a}{2}\right)+1\right)\zeta(-a)-\frac{\zeta(1-a)}{y}\nonumber\\
&=\frac{2\sqrt{2\pi}}{y^{1+\frac{a}{2}}}\sum_{n=1}^\infty\sigma_a(n)n^{-\frac{a}{2}}\left\{{}_{\frac{1}{2}}{K}_{\frac{a}{2}}\left(\frac{4\pi^2n}{y},0\right)-\frac{\pi2^{\frac{3}{2}+a}}{\sin\left(\frac{\pi a}{2}\right)}\left(\frac{4\pi^2n}{y}\right)^{-\frac{a}{2}-2}A_m\left(\frac{1}{2},\frac{a}{2},0;\frac{4\pi^2n}{y}\right)\right\}\nonumber\\
&\qquad-\frac{y(2\pi)^{-a-3}}{\sin\left(\frac{\pi a}{2}\right)}\sum_{k=0}^m\frac{\zeta(a+2k+2)\zeta(2k+2)}{\Gamma(-a-1-2k)}\left(\frac{4\pi^2}{y}\right)^{-2k},
\end{align}}
where
\begin{align}\label{am}
A_m(\mu,\nu,w;z):=\sum_{k=0}^m\frac{(-1)^{-\mu-w-\frac{1}{2}}\Gamma\left(\mu+w+\frac{1}{2}+k\right)}{k!\Gamma\left(-\nu-\mu-k\right)\Gamma\left(\frac{1}{2}-\nu-\mu-w-k\right)}\left(\frac{z}{2}\right)^{-2k}.
\end{align}
\end{theorem}
It is the above transformation that we call the \emph{master identity}, for, it gives numerous well-known as well as new results as corollaries. Note, of course, that if we let $m=0$ in Theorem \ref{extendedid}, we get Theorem \ref{main}.

We now give a special case of the above theorem.  
\begin{corollary}\label{ram}
Ramanujan's formula \eqref{rameqn} holds for $m>0$.
\end{corollary}
Theorem \ref{main} gives below the modular transformation satisfied by the Eisenstein series on $\textup{SL}_{2}\left(\mathbb{Z}\right)$ as a special case.
\begin{corollary}\label{a=2m-1}
Let $m>1$ be a natural number. If $\textup{Re}(\alpha), \textup{Re}(\beta)>0$ and $\alpha\beta=\pi^2$,
\begin{align}\label{a=2m-1equv}
\alpha^m\sum_{n=1}^\infty\sigma_{2m-1}(n)e^{-2n\alpha}-(-\beta)^m\sum_{n=1}^\infty\sigma_{2m-1}(n)e^{-2n\beta}=\left(\alpha^m-(-\beta)^m\right)\frac{B_{2m}}{4m}.
\end{align}
\end{corollary}
Next, we give a corollary of Theorem \ref{extendedid} which is equivalent to the transformation formula satisfied by the logarithm of the Dedekind eta function \cite[Equation (3.10)]{bls}.
\begin{corollary}\label{maineqn2cor}
If $\a,\b$ are such that $\textup{Re}(\a)>0, \textup{Re}(\b)>0$ and $\a\b=\pi^2$, then
\begin{equation}\label{maineqn2alt}
\sum_{n=1}^{\infty}\sigma_{-1}(n)e^{-2n\a}-\sum_{n=1}^{\infty}\sigma_{-1}(n)e^{-2n\b}=\frac{\b-\a}{12}+\frac{1}{4}\log\left(\frac{\a}{\b}\right),
\end{equation}
\end{corollary}
The limiting case $a\to1$ of the Theorem \ref{main} gives an equivalent form of the transformation satisfied by the weight-$2$ Eisenstein series $E_2(z)$ on $\textup{SL}_{2}\left(\mathbb{Z}\right)$, namely, $E_2\left(\frac{-1}z\right)=z^2E_2(z)+\frac{6z}{\pi i}$.
\begin{corollary}\label{e2transcor}
Let $\a,\b$ be such that $\textup{Re}(\a)>0, \textup{Re}(\b)>0$ and $\a\b=\pi^2$. Then
\begin{align}\label{e2trans}
\a\sum_{n=1}^{\infty}\frac{n}{e^{2n\a}-1}+\b\sum_{n=1}^{\infty}\frac{n}{e^{2n\b}-1}&=\frac{\a+\b}{24}-\frac{1}{4}.
\end{align}
\end{corollary}
Note that Corollary \ref{e2transcor} is also deducible from Ramanujan's formula \eqref{rameqn}. The transformation for $\sum_{n=1}^{\infty}d(n)e^{-ny}$, where $d(n)$ is the number of positive divisors of $n$, also results as a special case of the master identity. 
\begin{corollary}\label{kan}
For $\textup{Re}(y)>0$, we have
\begin{align}\label{kaneqn}
&\sum_{n=1}^{\infty}d(n)e^{-ny}-\frac{1}{4}-\frac{\left(\gamma-\log(y)\right)}{y}\nonumber\\
&=\frac{4}{y}\sum_{n=1}^{\infty}d(n)\left\{\sinh\left(\frac{4\pi^2n}{y}\right)\mathrm{Shi}\left(\frac{4\pi^2n}{y}\right)-\cosh\left(\frac{4\pi^2n}{y}\right)\mathrm{Chi}\left(\frac{4\pi^2n}{y}\right)\right\}.
\end{align}
Equivalently,
\begin{align}\label{kanot}
\sum_{n=1}^{\infty}\frac{1}{e^{ny}-1}-\frac{1}{4}-\frac{\left(\gamma-\log(y)\right)}{y}=\frac{2}{y}\sum_{n=1}^{\infty}\left\{\log\left(\frac{2\pi n}{y}\right)-\frac{1}{2}\left(\psi\left(\frac{2\pi in}{y}\right)+\psi\left(-\frac{2\pi in}{y}\right)\right)\right\},
\end{align}
\end{corollary}
The equivalent version of \eqref{kaneqn} given above was first discovered by Wigert \cite[p.~203, Equation~(A)]{wig0}, who considered it as `\emph{la formule importante}' (an important formula).

The modular-type transformation for $\sum_{n=1}^{\infty}\sigma_{2m}(n)e^{-ny}, m>0,$ is now given.
\begin{theorem}\label{a=2mcase}
Let $m\in\mathbb{N}$. Then for $\mathrm{Re}(y)>0$, we have
\begin{align}\label{a=2midentity} 
&\sum_{n=1}^\infty \sigma_{2m}(n)e^{-ny}-\frac{(2m)!}{y^{2m+1}}\zeta(2m+1)+\frac{B_{2m}}{2my}\nonumber\\
&=(-1)^m\frac{2}{\pi}\left(\frac{2\pi}{y}\right)^{2m+1}\sum_{n=1}^\infty\sigma_{2m}(n)\Bigg\{\sinh\left(\frac{4\pi^2n}{y}\right)\mathrm{Shi}\left(\frac{4\pi^2n}{y}\right)\nonumber\\
&\quad-\cosh\left(\frac{4\pi^2n}{y}\right)\mathrm{Chi}\left(\frac{4\pi^2n}{y}\right)+\sum_{j=1}^m(2j-1)!\left(\frac{4\pi^2n}{y}\right)^{-2j}\Bigg\}.
\end{align}
\end{theorem}
Note that the $m=1$ case of the series on the left-hand side above is connected with the generating function for plane partitions considered by MacMahon \cite[p.~184]{gea} in that, with $y=\log(1/q)$ and $F(q):=\displaystyle\prod_{n=1}^{\infty}\frac{1}{(1-q^n)^n}$, we have
$\displaystyle\sum_{n=1}^\infty \sigma_{2}(n)e^{-ny}=q\frac{d}{dq}\log F(q)$.

We now state a transformation for $\sum_{n=1}^{\infty}\sigma_{2m}(n)e^{-ny}$ when $m<0$.
\begin{theorem}\label{trans-2m}
Let $m\in\mathbb{N}$. Then for $\mathrm{Re}(y)>0$, the following transformation holds:
{\allowdisplaybreaks\begin{align}\label{resultinsincosh}
&\sum_{n=1}^\infty \sigma_{-2m}(n) e^{-ny}+\left\{\frac{1}{2}+(-1)^m\frac{\gamma}{\pi}\left(\frac{y}{2\pi}\right)^{2m-1}+\frac{(-1)^m}{\pi}\left(\frac{y}{2\pi}\right)^{2m-1}\log\left(\frac{2\pi}{y}\right)\right\}\zeta(2m)\nonumber\\
&+\frac{2(-1)^m}{\pi}\left(\frac{y}{2\pi}\right)^{2m-1}\left(-\frac{1}{2}\zeta'(2m)+\frac{2\pi^2}{y^2}\sum_{k=0}^{m-1}(-1)^{k+1}\zeta(2k+3)\zeta(2m-2k-2)\left(\frac{2\pi}{y}\right)^{2k}\right)\nonumber\\
&=\frac{2(-1)^{m}}{\pi}\left(\frac{y}{2\pi}\right)^{2m-1}\sum_{n=1}^\infty\sigma_{-2m}(n)\left\{\sinh\left(\frac{4\pi^2n}{y}\right)\mathrm{Shi}\left(\frac{4\pi^2n}{y}\right)-\cosh\left(\frac{4\pi^2n}{y}\right)\mathrm{Chi}\left(\frac{4\pi^2n}{y}\right)\right\}.
\end{align}}
Equivalently,
\begin{align}\label{equiforma=-2mano}
&\frac{1}{2}\zeta(2m)+\sum_{n=1}^\infty\frac{n^{-2m}}{e^{ny}-1}-\frac{1}{y}\sum_{k=0}^{m-1}\frac{B_{2k}}{(2k)!}\zeta(2m-2k+1)y^{2k}\nonumber\\
&=\frac{(-1)^{m+1}}{y}\left(\frac{y}{2\pi}\right)^{2m}\left\{2\g\zeta(2m)+\sum_{n=1}^\infty n^{-2m}\left(\psi\left(\frac{2\pi in}{y}\right)+\psi\left(-\frac{2\pi in}{y}\right)\right)\right\}.
\end{align}
\end{theorem}

The above theorem gives a companion to Ramanujan's formula \eqref{rameqn} which is new. The reason we call this as a companion to \eqref{rameqn} is because its left-hand side  involves even zeta values $\zeta(2m)$ and the series $\sum_{n=1}^\infty\frac{n^{-2m}}{e^{2n\a}-1}$ whereas that in \eqref{rameqn} involves the odd zeta values $\zeta(2m+1)$ and $\sum_{n=1}^\infty\frac{n^{-2m-1}}{e^{2n\a}-1}$. However, we note that while the latter series transforms to itself (with $\alpha$ replaced by $\beta$), the former transforms into  a combination of the \emph{higher Herglotz functions} of Vlasenko and Zagier \cite{vz} (see also Radchenko and Zagier \cite{raza}).
\begin{corollary}\label{companion}
Let $m\in\mathbb{N}$. If $\a$ and $\b$ are complex numbers such that $\textup{Re}(\a)>0$, $\textup{Re}(\b)>0$, and $\a\b=\pi^2$, then
\begin{align}\label{companioneqn}
&\a^{-\left(m-\frac{1}{2}\right)}\left\{\frac{1}{2}\zeta(2m)+\sum_{n=1}^\infty\frac{n^{-2m}}{e^{2n\a}-1}\right\}-\sum_{k=0}^{m-1}\frac{2^{2k-1}B_{2k}}{(2k)!}\zeta(2m-2k+1)\a^{2k-m-\frac{1}{2}}\nonumber\\
&=(-1)^{m+1}\b^{-\left(m-\frac{1}{2}\right)}\left\{\frac{\g}{\pi}\zeta(2m)+\frac{1}{2\pi}\sum_{n=1}^\infty n^{-2m}\left(\psi\left(\frac{in\b}{\pi}\right)+\psi\left(-\frac{in\beta}{\pi}\right)\right)\right\}.
\end{align}
\end{corollary}
Character analogues of Ramanujan's formula for $\zeta(2m+1)$ have been obtained by Katayama \cite{katayama} and Berndt \cite[Section 5]{berndtacta1975}. Berndt \cite[Section 11]{bcbspfunc} and Bradley \cite[Theorem 1]{bradley2002} gave generalizations and/or analogues of Ramanujan's formula for $\zeta(2m+1)$ for periodic sequences. See also Equation (3.12) of Berndt \cite{berndtrocky}. In all these results, series similar to the one on the left-hand side of \eqref{companioneqn} appear. 
 Yet, \eqref{companioneqn} is quite different in nature from each of these results. This is because, in the series similar to that on the left-hand side of \eqref{companioneqn} in the aforementioned results of Katayama, Berndt and Bradley, one cannot put the periodic function (or, the Dirichlet character, in the case of Dirichlet $L$-functions) to be identically equal to $1$ as their results in this setting hold when the periodic function or the Dirichlet character is odd.

New results concerning functions self-reciprocal in Watson's kernel are given in Section \ref{sffwk}.

\section{Nuts and bolts}
We will be frequently using the duplication and reflection formulas for $\Gamma(s)$ \cite[p.~46]{temme}:
\begin{align}
\G(s)\G\left(s+\frac{1}{2}\right)&=\frac{\sqrt{\pi}}{2^{2s-1}}\G(2s),\label{dupl}\\
\G(s)\G(1-s)&=\frac{\pi}{\sin(\pi s)}\hspace{5mm}(s\notin\mathbb{Z}).\label{refl}
\end{align}
Let $\mathfrak{F}(s)=M[f;s]$ and $\mathfrak{G}(s)=M[g;s]$ denote the Mellin transforms of functions $f(x)$ and $g(x)$ respectively. Let $c=\textup{Re}(s)$. If $M[f;1-c-it]\in L(-\infty, \infty)$ and $x^{c-1}g(x)\in L[0,\infty)$, then Parseval's formula \cite[p.~83]{kp} is given by
\begin{equation}
\int_{0}^{\infty} f(x) g(x) dx = \frac{1}{2 \pi i } 
\int_{(c)}\mathfrak{F}(1-s)\mathfrak{G}(s)\, ds,
\label{parseval-1}
\end{equation}
Here the vertical line Re$(s) = c$ lies in the common strip of 
analyticity of the Mellin transforms $\mathfrak{F}(1-s)$ and $\mathfrak{G}(s)$.

We need Slater's theorem \cite[p.~56-59]{marichev} to evaluate inverse Mellin transforms of certain functions. We give its statement below to make the paper self-contained. We need some notations to begin with. Let $|\arg(z)|<\pi$. Let $\pFq{p}{q}{a_1, a_2, \cdots, a_p}{b_1,b_2, \cdots, b_q}{z}$ be defined in \eqref{pfqgen}. Let
\begin{equation*}
\Gamma\genfrac[]{0pt}{}{a_1, a_2, \dots, a_A}{b_1, b_2, \dots, b_B}  \equiv \G[(a);(b)]=\frac{\Gamma \left( a_1 \right) \Gamma \left( a_2 \right) \dots \Gamma \left( a_A \right) }{\Gamma \left( b_1 \right) \Gamma \left( b_2 \right) \dots \Gamma \left( b_B \right) }, 
\end{equation*}
\begin{equation*}
(a) + s := a_1+s, a_2+s, \dots, a_A+s,
\end{equation*}
\begin{equation*}
(b)' - b_k := b_1-b_k, \dots, b_{k-1}-b_k,  b_{k+1}-b_k, \dots, b_B-b_k,
\end{equation*}
\begin{equation*}
\Sigma_A(z) := \sum_{j=1}^A z^{a_j} \Gamma \genfrac[]{0pt}{}{(a)' - a_j, (b) +  a_j}{(c) - a_j, (d) +  a_j} 
{}_{B+C}F_{A+D-1}
 \left( \genfrac{}{}{0pt}{}{(b)+a_j, 1+a_j-(c)}{1+a_j-(a)', (d)+a_j} 
 \bigg| {(-1)^{C-A} z} \right),
\end{equation*}
\begin{equation*}
\Sigma_B(1/z) := \sum_{k=1}^B z^{-b_k} \Gamma \genfrac[]{0pt}{}{(b)' - b_k, (a) +  b_k}{(d) - b_k, (c) +  b_k} 
{}_{A+D}F_{B+C-1}
 \left( \genfrac{}{}{0pt}{}{(a)+b_k, 1+b_k-(d)}{1+b_k-(b)', (c)+b_k} 
 \bigg| {\frac{(-1)^{D-B}}{z}} \right),
\end{equation*}
\begin{theorem}[Slater's Theorem]\label{slater}
Let 
\begin{equation}\label{functionHs}
\mathscr{H}(s) = \Gamma \genfrac[]{0pt}{}{(a)+s, (b)-s}{(c)+s, (d)-s},
\end{equation}
where the vectors $(a)$, $(b)$, $(c)$, and $(d)$ have, respectively, A, B, C, and D components $a_j$, $b_k$, $c_l$, and $d_m$. Then if the following two groups of conditions hold:
\begin{equation}\label{condition1}
-\textup{Re}(a_j) < \textup{Re}(s)  < \textup{Re}(b_k)  \quad (j = 1,2, \dots, A, \quad  k=1,2, \dots, B),\\
\end{equation}
\begin{equation}\label{condition2}
\begin{cases}
A+B > C+D, \\
A+B = C+D, \quad \textup{Re}(s(A+D-B-C)) < -\textup{Re}(\eta) \\
A=C, \quad B=D, \quad \textup{Re}(\eta) <0,
\end{cases}
\end{equation}
where
\begin{equation*}
\eta := \sum_{j=1}^A a_j + \sum_{k=1}^B b_k - \sum_{l=1}^C c_l - \sum_{m=1}^D d_m,
\end{equation*}
then for these $s$ we have
\begin{equation*}
\mathscr{H}(s) = 
\begin{cases}
\displaystyle\int_{0}^{\infty} x^{s-1} \Sigma_A(x)\, dx, \text{ if } A+D>B+C, \\
\displaystyle\int_{0}^{1} x^{s-1} \Sigma_A(x)\, dx  + \int_{1}^{\infty} x^{s-1} \Sigma_B(1/x)\, dx, \text{ if } A+D=B+C,\\
\displaystyle\int_{0}^{\infty} x^{s-1} \Sigma_B(1/x)\, dx, \text{ if } A+D<B+C,
\end{cases}
\end{equation*}
$\Sigma_A(1) = \Sigma_B(1)$ if $A+D=B+C$, \textup{Re}$(\eta) +C-A+1<0, A \geq C$.
\end{theorem}
\begin{corollary}\cite[p.~58]{marichev} \label{SlatersCor}
Under the conditions \eqref{condition1} and \eqref{condition2}, the inverse Mellin transform of the function in \eqref{functionHs} is a function $H(x)$ of hypergeometric type given by
\begin{equation*}
H(x) = 
\begin{cases}
\Sigma_A(x) \text{ for } x>0, \quad \text{ if } A+D>B+C, \\
\Sigma_A(x) \text{ for } 0<x<1, \quad \text{ or }  \quad \Sigma_B(1/x) \text{ for } x>1, \quad  \text{ if } A+D=B+C,\\
\Sigma_B(1/x) \text{ for } x>0, \quad \text{ if } A+D<B+C,
\end{cases}
\end{equation*}
$\mathscr{H}(1) = \Sigma_A(1) = \Sigma_B(1)$ if $A+D=B+C$, Re$(\eta) +C-A+1<0, A \geq C$.
\end{corollary}
\begin{remark}\label{slaterextension}
As given in \cite[Remark 2]{marichev}, whenever we have $|A+D-B-C|>1, A+B=C+D$, the condition $\textup{Re}(s(A+D-B-C)) < -\textup{Re}(\eta)$ can be weakened to $\textup{Re}(s(A+D-B-C)) < 1/2-\textup{Re}(\eta)$. 
\end{remark}
We will also be making use of the following asymptotic formula for it as $z\to\infty$ \cite[Equation (28)]{witte}, namely,
\begin{equation}\label{hypbessasym}
\frac{1}{\G(a)\G(b)\G(c)}\pFq03{-}{a, b, c}{z}\sim\frac{z^{\theta/4}}{2(2\pi)^{3/2}}\left\{\exp{\left(4z^{1/4}\right)}+2\cos\left(4z^{1/4}+\frac{\pi}{2}\theta\right)\right\},
\end{equation}
where $\theta=3/2-a-b-c$. 

The asymptotic formulas for the modified Bessel function $K_{\mu}(z)$ are required as well and hence are noted below. As $|z|\to\infty, |\arg(z)|<\pi$, we have \cite[p.~202]{watson-1944a}
\begin{equation}\label{kbessasym}
K_{\mu}(z)\sim \sqrt{\frac{\pi}{2 z}}e^{-z},
\end{equation}
whereas as for $z\to0$, we have \cite[p.~375, Equations (9.6.8), (9.6.9)]{stab}, \cite[Equation (2.14)]{dunster}
\begin{equation}
K_{\mu}(z)\sim\begin{cases}\label{kbessasym0}
\frac{1}{2}\G(\mu)\left(\frac{z}{2}\right)^{-\mu},\hspace{1mm}\text{if}\hspace{1mm}\textup{Re}(\mu)>0,\\
-\log z, \hspace{1mm}\text{if}\hspace{1mm}\mu=0,\\
-\sqrt{\frac{\pi}{y\sinh(\pi y)}}\sin\left(y\log\left(\frac{z}{2}\right)-\arg(\G(1+iy))\right),\hspace{1mm}\text{if}\hspace{1mm}\mu=iy, y>0.
\end{cases}
\end{equation}
We will also frequently make use of the well-known fact that
\begin{equation}\label{khalfz}
K_{\frac{1}{2}}(z)=\sqrt{\frac{\pi}{2z}}e^{-z}.
\end{equation}
\section{The Watson kernel $\mathscr{G}_{\nu}(x, w)$}\label{genknow}
As discussed in the introduction, the kernel $\mathscr{G}_{\nu}(x, w)$ defined in \eqref{genkoshk} is a simultaneous generalization of the Koshliakov kernels and (essentially) the Hankel kernel, all of which are quite important from the point of view of their applications in number theory. This section is devoted to obtaining some properties of this kernel.
\begin{theorem}\label{gkkkoshliakov}
Let $x>0$. Let $\mathscr{G}_{\nu}(x, w)$ be defined in \eqref{genkoshk}, and $M_{\nu}(x)$ and $L_{\nu}(x)$ in \eqref{mldef}. Then
\begin{align}
\mathscr{G}_{\nu}(x,0)&=\cos(\pi\nu) M_{2\nu}(2 \sqrt{x}) -\sin(\pi \nu) J_{2\nu}(2 \sqrt{x}),\label{1st}\\
\mathscr{G}_{\nu}(x,1)&=\sin(\pi\nu) J_{2\nu}(2 \sqrt{x}) -\cos(\pi\nu) L_{2\nu}(2 \sqrt{x}).\label{2nd}
\end{align}
\end{theorem}
\begin{proof}
We only prove \eqref{1st}. Then, \eqref{2nd} can be proved similarly. In the second step, use \eqref{dupl} twice and the fact $2^{2k}k!(1/2)_k=(2k)!$ to see that 
{\allowdisplaybreaks\begin{align}\label{0f3bess}
\pFq03{-}{1\pm\nu, \frac{1}{2}, \frac{1}{2}\pm\nu}{\frac{x^2}{16}}&=\sum_{k=0}^{\infty}\frac{\G(1\pm\nu)\G\left(\frac{1}{2}\pm\nu\right)}{k!\left(\frac{1}{2}\right)_k\G(1\pm\nu+k)\G\left(\frac{1}{2}\pm\nu+k\right)}\left(\frac{x}{4}\right)^{2k}\nonumber\\
&=\sum_{k=0}^{\infty}\frac{x^{2k}}{(2k)!(1\pm2\nu)_{2k}}\nonumber\\
&=\frac{1}{2}\left(\pFq01{-}{1\pm2\nu}{x}+\pFq01{-}{1\pm2\nu}{-x}\right).
\end{align}}
However, from the definitions \eqref{sumbesselj} and \eqref{besseli},
\begin{align}\label{0f1bess}
J_{\pm2\nu}(2\sqrt{x})=\frac{x^{\pm\nu}}{\G(1\pm2\nu)}\pFq01{-}{1\pm2\nu}{-x},\hspace{2mm}I_{\pm2\nu}(2\sqrt{x})=\frac{x^{\pm\nu}}{\G(1\pm2\nu)}\pFq01{-}{1\pm2\nu}{x}
\end{align}
so that from \eqref{genkoshk}, \eqref{0f3bess} and \eqref{0f1bess},
\begin{align}\label{althere}
\mathscr{G}_{\nu}(x,0)&=\frac{1}{2\sin(\nu\pi)}\left\{\left(I_{-2\nu}(2\sqrt{x})-I_{2\nu}(2\sqrt{x})\right)+\left(J_{-2\nu}(2\sqrt{x})-J_{2\nu}(2\sqrt{x})\right)\right\}\nonumber\\
&=\frac{2}{\pi}\cos(\nu\pi)K_{2\nu}(2\sqrt{x})+\frac{1}{2\sin(\nu\pi)}\left\{J_{-2\nu}(2\sqrt{x})-J_{2\nu}(2\sqrt{x})\right\},
\end{align}
where in the last step we used \eqref{kbesse}. Now the fact that
\begin{align}\label{althere1}
\frac{1}{2\sin(\nu\pi)}\left\{J_{-2\nu}(2\sqrt{x})-J_{2\nu}(2\sqrt{x})\right\}=-\cos(\nu\pi)Y_{2\nu}(2\sqrt{x})-\sin(\nu\pi)J_{2\nu}(2\sqrt{x})
\end{align}
can be proved using the definition of $Y_{\nu}(z)$ in \eqref{ybesse}. Finally, \eqref{althere} and \eqref{althere1} together imply \eqref{1st}.
\end{proof}
In \cite[p.~842]{bdrz}, it was remarked that $\cos(\pi\nu) M_{2\nu}(2 \sqrt{x}) -\sin(\pi \nu) J_{2\nu}(2 \sqrt{x})$, that is, the first Koshliakov kernel, is an even function of $\nu$. While it is not trivial to prove this directly using properties of Bessel functions, it \emph{is} so when viewed through the Watson kernel in \eqref{genkoshk} since, by definition, 
\begin{equation*}
\mathscr{G}_{-\nu}(x,w)=\mathscr{G}_{\nu}(x, w)\hspace{5mm}(w\in\mathbb{C}),
\end{equation*}
and hence, in particular, this is true for $w=0$.

The following big-O result for small positive values of $x$ is used in the sequel. The bound for the case $\nu=0$ has not appeared in the literature although the bound for the case $\nu\in\mathbb{C}\backslash\mathbb{Z}$ has been considered by Dahiya, see  \cite[Section 1 (iii)]{dahiyaromanian} in conjunction with \eqref{kernelequi} and \eqref{symmetry}. However, we prove both the cases to make the paper self-contained.
\begin{theorem}\label{gkkx0thm}
Let $x>0$ and $w\in\mathbb{C}$. As $x\to0^{+}$,
\begin{equation*}
\mathscr{G}_{\nu}(x, w)=\begin{cases}
O_{\nu, w}\left(x^{\textup{Re}(w)}(1+|\log x|)\right),\hspace{1mm}\text{if}\hspace{1mm}\nu=0,\\
O_{\nu, w}\left(x^{\textup{Re}(w)-|\textup{Re}(\nu)|}\right),\hspace{9mm}\text{if}\hspace{1mm}\nu\in\mathbb{C}\backslash\mathbb{Z}.
\end{cases}
\end{equation*}
\end{theorem}
\begin{proof}
By the definition of a ${}_0F_3$, as $x\to0$, 
\begin{equation*}
\frac{1}{\Gamma(1\pm\nu)\Gamma\left(w+\frac{1}{2}\right) \Gamma\left(w+\frac{1}{2}\pm\nu\right)} \pFq03{-}{1\pm\nu, w+\frac{1}{2}, w+\frac{1}{2}\pm\nu}{\frac{x^2}{16} }=O_{\nu,w}(1),
\end{equation*}
we see that $\mathscr{G}_{\nu}(x, w)=O_{\nu, w}\left(x^{\textup{Re}(w)-|\textup{Re}(\nu)|}\right)$ when $\nu\in\mathbb{C}\backslash\mathbb{Z}$. 

However, for $\nu=0$, observe that $\mathscr{G}_{\nu}(x, w)$ admits $\frac{0}{0}$ form. So we, of course, mean 
\begin{equation*}
\mathscr{G}_{0}(x, w)=\lim_{\nu\to0}\mathscr{G}_{\nu}(x, w).
\end{equation*}
For brevity, let $f(\nu)$ be defined by
\begin{equation*}
f(\nu):=\pi\left(\frac{x}{4}\right)^{w-\nu}\frac{1}{\Gamma(1-\nu)\Gamma\left(w+\frac{1}{2}\right) \Gamma\left(w+\frac{1}{2}-\nu\right)} \pFq03{-}{1-\nu, w+\frac{1}{2}, w+\frac{1}{2}-\nu}{\frac{x^2}{16}}.
\end{equation*}
Then 
\begin{align*}\label{g0xw}
\mathscr{G}_{0}(x, w)=\lim_{\nu\to0}\frac{f(\nu)-f(-\nu)}{\sin(\nu\pi)}=\frac{2f'(0)}{\pi}.
\end{align*}
Clearly, the ${}_0F_{3}$ converges uniformly in a neighborhood of $\nu=0$. Hence by \cite[p.~152, \textbf{7.17}]{rudin},
{\allowdisplaybreaks\begin{align}
f'(\nu)&=\pi\left(\frac{x}{4}\right)^{w}\bigg\{-\frac{(x/4)^{-\nu}\log\left(\frac{x}{4}\right)}{\G(1-\nu)\G\left(w+\frac{1}{2}\right)\G\left(w+\frac{1}{2}-\nu\right)}\pFq03{-}{1-\nu, w+\frac{1}{2}, w+\frac{1}{2}-\nu}{\frac{x^2}{16}}\nonumber\\
&\quad+\left(\frac{x}{4}\right)^{-\nu}\sum_{k=0}^{\infty}\frac{d}{d\nu}\left(\frac{(x^2/16)^{k}}{k!\G(1-\nu+k)\G\left(w+\frac{1}{2}+k\right)\G\left(w+\frac{1}{2}-\nu+k\right)}\right)\bigg\}\nonumber\\
&=\pi\left(\frac{x}{4}\right)^{w}\bigg\{-\frac{(x/4)^{-\nu}\log\left(\frac{x}{4}\right)}{\G(1-\nu)\G\left(w+\frac{1}{2}\right)\G\left(w+\frac{1}{2}-\nu\right)}\pFq03{-}{1-\nu, w+\frac{1}{2}, w+\frac{1}{2}-\nu}{\frac{x^2}{16}}\nonumber\\
&\qquad\qquad\quad+\left(\frac{x}{4}\right)^{-\nu}\sum_{k=0}^{\infty}\frac{(x^2/16)^{k}\left(\psi(w+\frac{1}{2}-\nu+k)+\psi(1-\nu+k)\right)}{k!\G\left(w+\frac{1}{2}+k\right)\G(1-\nu+k)\G(w+\frac{1}{2}-\nu+k)}\bigg\},
\end{align}}
where $\psi(x)=\G'(x)/\G(x)$. Now let $\nu\to0$ on both sides and use \cite[p.~149, \textbf{7.11}]{rudin} so that
\begin{align}\label{fd0}
f'(0)&=\pi\left(\frac{x}{4}\right)^{w}\bigg\{-\frac{\log\left(\frac{x}{4}\right)}{\G^{2}\left(w+\frac{1}{2}\right)}\pFq03{-}{1, w+\frac{1}{2}, w+\frac{1}{2}}{\frac{x^2}{16}}\nonumber\\
&\qquad\qquad\quad+\sum_{k=0}^{\infty}\frac{(x^2/16)^{k}\left(\psi(w+\frac{1}{2}+k)+\psi(1+k)\right)}{k!\G^{2}\left(w+\frac{1}{2}+k\right)\G(1+k)}\bigg\}.
\end{align}
From \eqref{g0xw} and \eqref{fd0}, we conclude that as $x\to0$,
\begin{align*}
\mathscr{G}_0(x,w)=O\left(x^{\textup{Re}(w)}(1+|\log x|)\right).
\end{align*}
\end{proof}
\begin{remark}
Theorem \textup{\ref{gkkx0thm}} with $w=0$ gives a well-known result \cite[Equation (2.11)]{drrz02} for $\nu\in\mathbb{C}\backslash\left(\mathbb{Z}\backslash\{0\}\right)$.
\end{remark}
The Mellin transform of $\mathscr{G}_{\nu}(x, w)$ was first given by Bhatnagar in \cite[p.~23]{bhatnagarganita}\footnote{The condition for the validity of the Mellin transform given in this paper, namely $-\textup{Re}(w)\pm\textup{Re}(\nu)<\textup{Re}(s)<\frac{1}{4}$, is too restrictive. Here we extend the region of validity.} However, he did not give a proof of it. We prove it here using Theorem \ref{slater} of Slater. Here, and throughout the sequel, we use $\int_{(c)}$ to denote the line integral $\int_{c-i\infty}^{c+i\infty}$.
\begin{theorem}\label{gkkmellin}
For $-\textup{Re}(w)\pm\textup{Re}(\nu)<\textup{Re}(s)<\frac{3}{4}$,
\begin{align}\label{mellin-kkk}
\int_{0}^{\infty} t^{s-1} \mathscr{G}_{\nu}(xt, w)\, dt
= 2^{2s-1} x^{-s}
\frac{\G\left(\frac{s-\nu+w}{2}\right)\G\left(\frac{s+\nu+w}{2}\right)}{\G\left(\frac{1-s-\nu+w}{2}\right)\G\left(\frac{1-s+\nu+w}{2}\right)}.
\end{align}
\end{theorem}
\begin{proof}
We prove the equivalent statement, namely, for $-\textup{Re}(w)\pm\textup{Re}(\nu)<c=\textup{Re}(s)<\frac{3}{4}$,
\begin{align*}
\frac{1}{2\pi i}\int_{(c)}\frac{\G\left(\frac{s-\nu+w}{2}\right)\G\left(\frac{s+\nu+w}{2}\right)}{\G\left(\frac{1-s-\nu+w}{2}\right)\G\left(\frac{1-s+\nu+w}{2}\right)}2^{2s-1}(xt)^{-s}\, ds=\mathscr{G}_{\nu}(xt, w).
\end{align*}
This is obtained by first employing the change of variable $s=2s_1$ in the integral on the left-hand side of the above equation. Then use Theorem \ref{slater} with $A=2, B=0, C=0, D=2, a_1=\frac{1}{2}(w-\nu), a_2=\frac{1}{2}(w+\nu), d_1=\frac{1}{2}(1-\nu+w), d_2=\frac{1}{2}(1+\nu+w)$, and $z=(xt)^2/16$. This establishes \eqref{mellin-kkk} for $-\textup{Re}(w)\pm\textup{Re}(\nu)<\textup{Re}(s)<\frac{1}{2}$. However, observe that the Remark \ref{slaterextension} applies here and thus we can extend the region of validity to $-\textup{Re}(w)\pm\textup{Re}(\nu)<\textup{Re}(s)<\frac{3}{4}$.
\end{proof}
When $w=0$ and $1$, the above result respectively gives Lemmas 5.1 and 5.2 from \cite{dixitmoll} as special cases as can be seen from Theorem \ref{gkkkoshliakov} and the standard properties of $\G(s)$.

\section{Proofs of Theorem \ref{koshlyakovg-wthm} and its corollaries}\label{proof}
In this section, we prove Theorem \ref{koshlyakovg-wthm} and Corollaries \ref{koshlyakovfir}, \ref{koshlyakovsec} and \ref{besinteval}. 

\begin{proof}[Theorem \textup{\ref{koshlyakovg-wthm}}][]
First, let $\mu\neq-\nu$. We prove \eqref{koshlyakovg-w}. We begin by showing that the integral on the left-hand side of \eqref{koshlyakovg-w} converges for the said values of the variables $\mu, \nu$ and $w$. That the convergence at the upper limit of integration is secured can be clearly seen from \eqref{hypbessasym} and \eqref{kbessasym}. To show that it is also secured at $0$, we consider two cases - first, $\nu\in\mathbb{C}\backslash\mathbb{Z}$ and the second, $\nu=0$.

\underline{Case 1:} $\nu\in\mathbb{C}\backslash\mathbb{Z}$

(i) Let $\textup{Re}(\mu)>0$ and Re$(\nu)>0$. From \eqref{kbessasym0} and Theorem \ref{gkkx0thm}, we see that the integrand
\begin{align*}
 K_{\mu}(t)t^{\mu+\nu+w} \mathscr{G}_{\nu}(xt, w)&= O_{\mu, \nu, w, x}\left(t^{-\textup{Re}(\mu)+\textup{Re}(\mu+\nu+w)+\textup{Re}(w)-|\textup{Re}(\nu)|}\right)\nonumber\\
&=O_{\mu, \nu, w, x}\left(t^{2\textup{Re}(w)}\right),
\end{align*}
as $t\to0$, so that Re$(w)>-\frac{1}{2}$ is needed.

(ii) Let $\textup{Re}(\mu)<0$ and Re$(\nu)>0$. Then we replace $\mu$ by $-\mu$ in \eqref{kbessasym0} and then use the fact $K_{-\mu}(t)=K_{\mu}(t)$ so that as $t\to0$,
\begin{align*}
 K_{\mu}(t)t^{\mu+\nu+w} \mathscr{G}_{\nu}(xt, w)&= O_{\mu, \nu, w, x}\left(t^{\textup{Re}(\mu)+\textup{Re}(\mu+\nu+w)+\textup{Re}(w)-|\textup{Re}(\nu)|}\right)\nonumber\\
&=O_{\mu, \nu, w, x}\left(t^{2\textup{Re}(\mu+w)}\right),
\end{align*}
so that we need Re$(\mu)>-\textup{Re}(w)-\frac{1}{2}$.

(iii) When $\textup{Re}(\mu)>0$ and Re$(\nu)<0$, one can similarly see that we require Re$(\nu)>-\textup{Re}(w)-\frac{1}{2}$.

(iv) Now let $\textup{Re}(\mu)<0$ and Re$(\nu)<0$. As $t\to0$,
\begin{align*}
 K_{\mu}(t)t^{\mu+\nu+w} \mathscr{G}_{\nu}(xt, w)&= O_{\mu, \nu, w, x}\left(t^{\textup{Re}(\mu)+\textup{Re}(\mu+\nu+w)+\textup{Re}(w)+\textup{Re}(\nu)}\right)\nonumber\\
&=O_{\mu, \nu, w, x}\left(t^{2\textup{Re}(\mu+\nu+w)}\right),
\end{align*}
so that Re$(\mu+\nu)>-\textup{Re}(w)-\frac{1}{2}$ is required.

Similarly in the remaining five cases $\textup{Re}(\mu)=0, \textup{Re}(\pm\nu)>0$  or $\textup{Re}(\pm\mu)>0, \textup{Re}(\nu)=0$ or $\textup{Re}(\mu)=\textup{Re}(\nu)=0$, we see that we require $\textup{Re}(w)>-1/2, \textup{Re}(\mu)>-\textup{Re}(w)-\frac{1}{2}, \textup{Re}(\nu)>-\textup{Re}(w)-\frac{1}{2}$ and Re$(\mu+\nu)>-\textup{Re}(w)-\frac{1}{2}$.

\underline{Case 2:} $\nu=0$

(i) Let $\textup{Re}(\mu)\geq0$. Using Theorem \ref{gkkx0thm} and \eqref{kbessasym0}, it is clear that as $t\to0$,
\begin{equation*}
K_{\mu}(t)t^{\mu+\nu+w} \mathscr{G}_{\nu}(xt, w)= O_{\mu, w, x}\left(t^{2\textup{Re}(w)}(1+|\log t|)\right).
\end{equation*}
Hence as long as $\textup{Re}(w)>-1/2$, the convergence is secured.

(ii) Similarly when $\textup{Re}(\mu)<0$, it can be seen that we need $\textup{Re}(\mu+w)>-1/2$.

Thus we conclude that the integral on the left-hand side of \eqref{koshlyakovg-w} converges when  $\textup{Re}(w)>-1/2$ and  $\textup{Re}(\mu), \textup{Re}(\nu), \textup{Re}(\mu+\nu)>-\textup{Re}(w)-1/2$.

We now prove \eqref{koshlyakovg-w}. Apart from the conditions 
\begin{align}\label{con1}
\textup{Re}(w)>-1/2, \hspace{3mm}\text{and}\hspace{1mm}\textup{Re}(\mu), \textup{Re}(\nu), \textup{Re}(\mu+\nu)>-\textup{Re}(w)-1/2,
\end{align}
we initially require one more condition 
\begin{align}\label{con2}
\textup{Re}(\nu)<\textup{Re}(w)+\frac{1}{2}.
\end{align}
This additional restriction is removed later.

Let $I=I(\mu, \nu, x, w)$ denote the integral on the left-hand side of \eqref{koshlyakovg-w}. For Re$(s)>\pm\textup{Re}(\mu)$ and $a>0$, we have \cite[p.~115, Equation (11.1)]{ober},
\begin{equation}\label{kbessmel}
\int_{0}^{\infty}t^{s-1}K_{\mu}(at)\, dt=2^{s-2}a^{-s}\G\left(\frac{s-\mu}{2}\right)\G\left(\frac{s+\mu}{2}\right).
\end{equation}
From Theorem \ref{gkkmellin}, \eqref{kbessmel}, and an application of Parseval's formula \eqref{parseval-1}, we see that for $c=$Re$(s)<\min{(\textup{Re}(1+\nu+w), \textup{Re}(1+2\mu+\nu+w))}$ and Re$(-w\pm\nu)<c=\textup{Re}(s)<3/4$,
\begin{equation*}
I= \frac{1}{2\pi i} \int_{(c)} 
\frac{ \Gamma\left(\frac{s-\nu+w}{2}\right)
\Gamma\left(\frac{s+\nu+w}{2}\right)
\Gamma\left(\frac{1-s+\nu+w}{2}+\mu\right)}
{\Gamma\left(\frac{1-s-\nu+w}{2}\right)}2^{\mu+\nu+w+s-2}  x^{-s} \ ds
\end{equation*}
Now employ the change of variable $s=2\xi+\nu$ so that for $\max{\left(-\textup{Re}\left(\frac{w}{2}\right), -\textup{Re}\left(\frac{w}{2}+\nu\right)\right)}<c_1=\textup{Re}(\xi)<\min{\left(\textup{Re}\left(\frac{1+w}{2}\right), \textup{Re}\left(\mu+\frac{w+1}{2}\right), \frac{3}{8}-\frac{1}{2}\textup{Re}(\nu)\right)}$, we have
\begin{align}\label{aftcond}
I=\frac{2^{\mu+\nu+w-1}}{2\pi i} \left(\frac{x}{2}\right)^{-\nu} \int_{(c_1)}   
\left(\frac{x^2}{4}\right)^{-\xi} 
\frac{ \Gamma\left( \tfrac{w+1}{2}+\mu-\xi\right)\Gamma(\xi + \tfrac{w}{2})\Gamma(\xi+\nu+\tfrac{w}{2})}
{\Gamma\left( \tfrac{w+1}{2}-\nu-\xi\right)} \,d\xi.
\end{align}
It is important to note here that \eqref{con1} and \eqref{con2} imply that the strip in the $\xi$-complex plane in the sentence above \eqref{aftcond} is of non-zero width.

Now we evaluate the inverse Mellin transform in \eqref{aftcond} using Slater's theorem, that is, Theorem \ref{slater}, with $A=2, B=1, C=0, D=1, a_1=w/2, a_2=\nu+w/2, b_1=\mu+(w+1)/2$ and $d_1=-\nu+(w+1)/2$. Note that \eqref{con1} and \eqref{con2} imply that the conditions \eqref{condition1} in the hypotheses of Slater's theorem are satisfied. 

This indeed gives \eqref{koshlyakovg-w} upon simplification when the conditions in \eqref{con1} and \eqref{con2} are satisfied. 
 
However, using \eqref{kbessasym0}, Theorem \ref{gkkx0thm} and \cite[p.~30, Theorem 2.3]{temme}, we see that the left-hand side of \eqref{koshlyakovg-w} is analytic in $\nu$ as long as $\nu\in\mathbb{C}\backslash\left(\mathbb{Z}\backslash\{0\}\right)$ and the conditions in \eqref{con1} are met. Since the right-hand side of \eqref{koshlyakovg-w} too is analytic in $\nu$ when these conditions hold, by the principle of analytic continuation, \eqref{koshlyakovg-w} holds when $\nu\in\mathbb{C}\backslash\left(\mathbb{Z}\backslash\{0\}\right)$, $\textup{Re}(w)>-1/2$ and  $\textup{Re}(\mu), \textup{Re}(\nu), \textup{Re}(\mu+\nu)>-\textup{Re}(w)-1/2$.

Equation \eqref{selff} can be derived by simply letting $\mu=-\nu$ in \eqref{koshlyakovg-w} and then using \eqref{mueq-nu}.


\end{proof}
Koshliakov's \eqref{koshlyakov-1} and \eqref{koshlyakovg-1} can now be derived simply as special cases of Theorem \ref{koshlyakovg-wthm}.

\begin{corollary}\label{koshlyakovfir}
Let $x>0$. For $\mu\neq-\nu$, Equation \eqref{koshlyakovg-1} holds for $\textup{Re}(\mu)>-1/2$, $\textup{Re}(\nu)>-1/2$, $\nu\notin\mathbb{Z}^{+}$, and $\textup{Re}(\mu+\nu)>-1/2$, otherwise, Equation \eqref{koshlyakov-1} holds for $-\tfrac{1}{2}<\textup{Re}(\nu)<\tfrac{1}{2}$.
\end{corollary}
\begin{proof}
Let $w=0$ in Theorem \ref{koshlyakovg-wthm} and make use of the first identity in Theorem \ref{gkkkoshliakov}, namely, Equation \eqref{1st}. 
\end{proof}
A companion to \eqref{koshlyakovg-1} can also be derived from Theorem \ref{koshlyakovg-wthm}.
\begin{corollary}\label{koshlyakovsec}
Let $x>0$. For $\mu\neq-\nu$, if $\textup{Re}(\mu)>-3/2$, $\textup{Re}(\nu)>-3/2, \nu\notin\mathbb{Z}^{+}\cup\{-1\}$, and $\textup{Re}(\mu+\nu)>-3/2$, we have
{\allowdisplaybreaks\begin{align*}
&\int_{0}^{\infty} K_{\mu}(t)t^{\mu+\nu+1} \left( \sin(\pi\nu) J_{2\nu}(2 \sqrt{xt}) -
\cos(\pi\nu) L_{2\nu}(2 \sqrt{xt}) \right)\, dt \nonumber\\
&=\frac{\pi x 2^{\mu+\nu-1}}{\sin(\nu\pi)}\bigg\{\left(\frac{x}{2}\right)^{-\nu}\frac{\Gamma(\mu+\tfrac{3}{2})}{\Gamma(1-\nu)\Gamma(\tfrac{3}{2}-\nu)}\pFq12{\mu+\tfrac{3}{2}}{\tfrac{3}{2}-\nu,1-\nu}{\frac{x^2}{4}}\nonumber\\
&\quad\quad\quad\quad\quad\quad-\left(\frac{x}{2}\right)^{\nu}\frac{\Gamma(\mu+\nu+\tfrac{3}{2})}{\Gamma(1+\nu)\Gamma(\tfrac{3}{2})}\pFq12{\mu+\nu+\tfrac{3}{2}}{\tfrac{3}{2},1+\nu}{\frac{x^2}{4}}\bigg\},
\end{align*}}
otherwise, for $-3/2<\textup{Re}(\nu)<3/2$, we have\footnote{This integral evaluation was obtained by Koshliakov \cite[Equation (13)]{kosh1938}.}
\begin{equation*}
\int_{0}^{\infty} tK_{\nu}(t) \left( \sin(\nu\pi) J_{2 \nu}(2 \sqrt{xt}) -
\cos(\nu\pi) L_{2\nu}(2 \sqrt{xt}) \right)\, dt = xK_{\nu}(x),
\end{equation*}
\end{corollary}
\begin{proof}
Let $w=1$ in Theorem \ref{koshlyakovg-wthm} use the second identity in Theorem \ref{gkkkoshliakov}, that is, Equation \eqref{2nd}.
\end{proof} 
Yet another corollary of Theorem \ref{koshlyakovg-wthm} gives explicit evaluations of the Hankel transforms of well-known functions.
\begin{corollary}\label{besinteval}
Let $x>0$. If $\mu\neq-1/2$, then for $\textup{Re}(w)>-1/2$, $\textup{Re}(\mu)>-\textup{Re}(w)-1/2$, 
\begin{align}\label{5.12}
\int_{0}^{\infty}t^{\mu+w+\frac{1}{2}}K_{\mu}(t)J_{2w-1}(2\sqrt{xt})\, dt&=2^{\mu}\sqrt{\pi}x^{w-\frac{1}{2}}\bigg\{\frac{\Gamma\left(\mu+w+\frac{1}{2}\right)}{\G(w)}\pFq12{\mu+w+\frac{1}{2}}{\frac{1}{2}, w}{\frac{x^2}{4}}\nonumber\\
&\quad-x\frac{\G(\mu+w+1)}{\G(w+\frac{1}{2})}\pFq12{\mu+w+1}{\frac{3}{2}, w+\frac{1}{2}}{\frac{x^2}{4}}\bigg\},
\end{align}
else, for $\textup{Re}(w)>0$,
\begin{equation*}
\int_{0}^{\infty}e^{-t}t^{w-\frac{1}{2}}J_{2w-1}(2\sqrt{xt})\, dt=e^{-x}x^{w-\frac{1}{2}}.
\end{equation*}
\end{corollary}
\begin{proof}
Let $\nu=1/2$ in Theorem \ref{koshlyakovg-wthm} and use \eqref{gkkhalf} to obtain \eqref{5.12}. In the second half of the theorem, use \eqref{mueq-nu} and \eqref{khalfz}.
\end{proof}

\section{Self-reciprocal functions in the Watson kernel and a modular relation}\label{sffwk}

In \eqref{selff}, we found an example of a function which is self-reciprocal in the Watson kernel $\mathscr{G}_{\nu}(xt, w)$. Letting $\nu=z/2$ in \eqref{selff} and replacing $x$ and $t$ by $2\pi x$ and $2\pi t$ respectively, we have an equivalent form
\begin{equation}\label{selffano}
2\pi\int_{0}^{\infty}t^wK_{\frac{z}{2}}(2\pi t)\mathscr{G}_\frac{z}{2}(4\pi^2xt,w)\, dt=x^wK_{\frac{z}{2}}(2\pi x),
\end{equation}
valid for $x>0$ and Re$(w)>\textup{max}\left(-\frac{1}{2}, \frac{1}{2}(|\textup{Re}(z)|-1)\right)$.

Several criteria have been derived for a function to be self-reciprocal in the Watson kernel, for example, see \cite{bhatnagarbelgique} and \cite{bhatnagar1954f}. Here, in the case of a \emph{non-negative integer} $w$, we obtain a new criterion that seems to have been missed. This result is given in the following theorem. 
\begin{theorem}\label{selfrecigkk}
Let $w$ be a non-negative integer and $z\in\mathbb{C}$ such that $|\textup{Re}(z)|<2w+1$. Suppose there exist functions $f(t, z, w)$ and $F(s, z, w)$ with the following three properties:

\textup{(i)} $F(s, z, w)=F(1-s, z, w)$;

\textup{(ii)} For $-w\pm\textup{Re}\left(\frac{z}{2}\right)<\textup{Re}(s)<1+w\pm\textup{Re}\left(\frac{z}{2}\right)$, the Mellin transform of $f(t, z, w)$ is
\begin{align}\label{ftzweven}
\int_{0}^{\infty}t^{s-1}f(t, z, w)\, dt=F(s, z, w)\zeta\left(1-s-\frac{z}{2}\right)\zeta\left(1-s+\frac{z}{2}\right)\left(\frac{s}{2}-\frac{z}{4}\right)_{\frac{w}{2}}\left(\frac{s}{2}+\frac{z}{4}\right)_{\frac{w}{2}},
\end{align}
if $w$ is even, and
\begin{align}\label{ftzwodd}
\int_{0}^{\infty}t^{s-1}f(t, z, w)\, dt&=\frac{F(s, z, w)}{2\left(\cos\left(\frac{\pi z}{2}\right)+\cos(\pi s)\right)}\zeta\left(1-s-\frac{z}{2}\right)\zeta\left(1-s+\frac{z}{2}\right)\nonumber\\
&\quad\times\left(\frac{s}{2}-\frac{z}{4}+\frac{1}{2}\right)_{\frac{w-1}{2}}\left(\frac{s}{2}+\frac{z}{4}+\frac{1}{2}\right)_{\frac{w-1}{2}},
\end{align}
if $w$ is odd;

\textup{(iii)} The Mellin transform of $f(t, z, w)$, namely $M[f;c+it]$ and given by the right-hand sides of \eqref{ftzweven} and \eqref{ftzwodd}, satisfies $M[f;1-c-it]\in L(-\infty,\infty)$.

Then $f$ is self-reciprocal (up to the constant $2\pi$) in the Watson kernel, that is,
\begin{equation*}
f(x, z, w)=2\pi\int_{0}^{\infty}f(t, z, w)\mathscr{G}_{\frac{z}{2}}(4\pi^2xt,w)\, dt.
\end{equation*}
\end{theorem}
When $w=0$ and $w=1$, Theorem \ref{selfrecigkk} reduces to Theorems 5.3 and 5.5 from \cite{dixitmoll} respectively. 
\begin{proof}[Theorem \textup{\ref{selfrecigkk}}][]
We prove only the case when $w$ is even. The case when $w$ is odd can be similarly proved. Replacing $\nu$ by $z/2$, $x$ by $2\pi x$ and $t$ by $2\pi t$ in \eqref{mellin-kkk}, we see that if $-w\pm\textup{Re}\left(\frac{z}{2}\right)<\textup{Re}(s)<\frac{3}{4}$, then
\begin{equation}\label{mellin-kkk1}
\int_{0}^{\infty}t^{s-1}\mathscr{G}_{\frac{z}{2}}(4\pi^2xt,w)\, dt=\frac{1}{2(\pi^2x)^{s}}\frac{\G\left(\frac{s+w}{2}-\frac{z}{4}\right)\G\left(\frac{s+w}{2}+\frac{z}{4}\right)}{\G\left(\frac{1-s+w}{2}-\frac{z}{4}\right)\G\left(\frac{1-s+w}{2}+\frac{z}{4}\right)}.
\end{equation}
Observe that the hypothesis $|\textup{Re}(z)|<2w+1$ ensures that $-w\pm\textup{Re}\left(\frac{z}{2}\right)<\frac{1}{2}<\frac{3}{4}$  and hence we can use \eqref{mellin-kkk1}.
 
From \eqref{ftzweven}, hypothesis (iii), \eqref{mellin-kkk1} and Parseval's formula \eqref{parseval-1}, for $-w\pm\textup{Re}\left(\frac{z}{2}\right)<c=\textup{Re}(s)<\min{\left(\frac{1}{2}, 1+w\pm\textup{Re}\left(\frac{z}{2}\right)\right)}$, we have
{\allowdisplaybreaks\begin{align}\label{prodzeta0}
&2\pi\int_{0}^{\infty}f(t, z, w)\mathscr{G}_{\frac{z}{2}}(4\pi^2xt,w)\, dt\nonumber\\
&=\frac{1}{i}\int_{(c)}F(1-s, z, w)\zeta\left(s-\frac{z}{2}\right)\zeta\left(s+\frac{z}{2}\right)\left(\frac{1-s}{2}-\frac{z}{4}\right)_{\frac{w}{2}}\left(\frac{1-s}{2}+\frac{z}{4}\right)_{\frac{w}{2}}\nonumber\\
&\qquad\quad\times\frac{1}{2(\pi^2x)^{s}}\frac{\G\left(\frac{s+w}{2}-\frac{z}{4}\right)\G\left(\frac{s+w}{2}+\frac{z}{4}\right)}{\G\left(\frac{1-s+w}{2}-\frac{z}{4}\right)\G\left(\frac{1-s+w}{2}+\frac{z}{4}\right)}\, ds\nonumber\\
&=\frac{1}{i}\int_{(c)}F(s, z, w)\zeta\left(s-\frac{z}{2}\right)\zeta\left(s+\frac{z}{2}\right)\left(\frac{s}{2}-\frac{z}{4}\right)_{\frac{w}{2}}\left(\frac{s}{2}+\frac{z}{4}\right)_{\frac{w}{2}}\nonumber\\
&\qquad\quad\times\frac{1}{2(\pi^2x)^{s}}\frac{\G\left(\frac{s}{2}-\frac{z}{4}\right)\G\left(\frac{s}{2}+\frac{z}{4}\right)}{\G\left(\frac{1-s}{2}-\frac{z}{4}\right)\G\left(\frac{1-s}{2}+\frac{z}{4}\right)}\, ds,
\end{align}}
where in the last step we used $F(s, z, w)=F(1-s, z, w)$.

Next, use the asymmetric form of the functional equation of the Riemann zeta function \cite[p.~259, Theorem 12.7]{apostol}
\begin{equation}\label{zetafealt}
\zeta(1-s)=2^{1-s}\pi^{-s}\G(s)\zeta(s)\cos\left(\frac{\pi s}{2}\right)
\end{equation}
twice, once with $s$ replaced by $s+\frac{z}{2}$, and then by $s-\frac{z}{2}$ resulting in
\begin{align}\label{prodzeta}
\zeta\left(s-\frac{z}{2}\right)\zeta\left(s+\frac{z}{2}\right)&=2^{2s}\pi^{2s-2}\G\left(1-s-\frac{z}{2}\right)\G\left(1-s+\frac{z}{2}\right)\zeta\left(1-s-\frac{z}{2}\right)\nonumber\\
&\quad\times\zeta\left(1-s+\frac{z}{2}\right)\sin\left(\frac{\pi}{2}\left(s-\frac{z}{2}\right)\right)\sin\left(\frac{\pi}{2}\left(s+\frac{z}{2}\right)\right).
\end{align}
Now substitute \eqref{prodzeta} in \eqref{prodzeta0} and use both \eqref{dupl} and \eqref{refl} twice thereby deducing
\begin{align*}
&2\pi\int_{0}^{\infty}f(t, z, w)g_{\frac{z}{2}}(4\pi^2xt,w)\, dt\nonumber\\
&=\frac{1}{2\pi i}\int_{(c)}x^{-s}F(s, z, w)\zeta\left(1-s-\frac{z}{2}\right)\zeta\left(1-s+\frac{z}{2}\right)\left(\frac{s}{2}-\frac{z}{4}\right)_{\frac{w}{2}}\left(\frac{s}{2}+\frac{z}{4}\right)_{\frac{w}{2}}\, ds\nonumber\\
&=f(x, z, w),
\end{align*}
where the last step follows from \eqref{ftzweven}.
\end{proof}
\begin{remark}
In the case when $w$ is a non-negative integer, \eqref{selffano} can also be derived as a corollary of Theorem \textup{\ref{selfrecigkk}}. To see this, take $f(t, z, w)=t^{w}K_{\frac{z}{2}}(2\pi t)$ and 
\begin{align}\label{fszw2def}
F(s, z, w)=\begin{cases}
\frac{2^{w-2}\pi^{-s}\G\left(\frac{s}{2}-\frac{z}{4}\right)\G\left(\frac{s}{2}+\frac{z}{4}\right)}{\zeta\left(1-s-\frac{z}{2}\right)\zeta\left(1-s+\frac{z}{2}\right)},\text{if}\hspace{1mm}w\hspace{1mm}\text{is even},\\
\frac{2^{w-1}\pi^{-s}\G\left(\frac{s}{2}-\frac{z}{4}+\frac{1}{2}\right)\G\left(\frac{s}{2}+\frac{z}{4}+\frac{1}{2}\right)\left(\cos\left(\frac{\pi z}{2}\right)+\cos(\pi s)\right)}{\zeta\left(1-s-\frac{z}{2}\right)\zeta\left(1-s+\frac{z}{2}\right)},\text{if}\hspace{1mm}w\hspace{1mm}\text{is odd},
\end{cases}
\end{align}
and show that the hypotheses of the theorem are satisfied if $|\textup{Re}(z)|<2w+1$ and $-w\pm\textup{Re}\left(\frac{z}{2}\right)<\textup{Re}(s)<1+w\pm\textup{Re}\left(\frac{z}{2}\right)$. Thus, we note that the function $F(s, z, w)$ may have two different expressions depending on the parity of $w$.
\end{remark}
\begin{remark}
Note that \eqref{selffano} is valid for any $w$ such that \textup{Re}$(w)>\textup{max}\left(-\frac{1}{2}, \frac{1}{2}(|\textup{Re}(z)|-1)\right)$ whereas Theorem \textup{\ref{selfrecigkk}} gives \eqref{selffano} only for a \emph{natural number} $w$ because of its hypotheses.
\end{remark}
Our last result of this section is the following modular relation between two integrals.
\begin{theorem}\label{mttwthm}
Let $w$ be a non-negative integer and $z\in\mathbb{C}$ such that $|\textup{Re}(z)|<2w+1$. Let $f(t, z, w)$ be as in the previous theorem. Then for $\a, \b>0$ such that $\a\b=1$, we have
\begin{equation*}
\a^{w+\frac{1}{2}}\int_{0}^{\infty}x^wK_{\frac{z}{2}}(2\pi\a x)f(x, z, w)\, dx=\b^{w+\frac{1}{2}}\int_{0}^{\infty}x^wK_{\frac{z}{2}}(2\pi\b x)f(x, z, w)\, dx.
\end{equation*}
\end{theorem}
When $w=0$ and $w=1$, the above theorem reduces to Corollaries 5.4 and 5.6 in \cite{dixitmoll}. An example of a function self-reciprocal in the first Koshliakov kernel, other than $K_{\frac{z}{2}}(x)$, has been considered in \cite[Theorem 15.6]{bdrz}, which leads to the above modular relation in the special case $w=0$.

\begin{proof}[Theorem \textup{\ref{mttwthm}}][]
We are given $w\in\mathbb{N}\cup\{0\}$, $\a,\b>0$ such that $\a\b=1$. Replace $x$ and $t$ in \eqref{selffano} by $\a x$ and $\b y$ respectively so that for $|\textup{Re}(z)|<2w+1$,
\begin{equation*}
\a^{w}x^wK_{\frac{z}{2}}(2\pi\a x)=2\pi\b^{w+1}\int_{0}^{\infty}y^wK_{\frac{z}{2}}(2\pi \b y)\mathscr{G}_\frac{z}{2}(4\pi^2xy,w)\, dt,
\end{equation*}
Multiply both sides of the above equation by $\sqrt{\a}f(x, z, w)$, integrate both sides with respect to $x$ and use the fact $\a\b=1$ so as to get
{\allowdisplaybreaks\begin{align*}
&\a^{w+\frac{1}{2}}\int_{0}^{\infty}x^wK_{\frac{z}{2}}(2\pi\a x)f(x, z, w)\, dx\nonumber\\
&=2\pi\b^{w+\frac{1}{2}}\int_{0}^{\infty}\int_{0}^{\infty}y^wK_{\frac{z}{2}}(2\pi\b y)f(x, z, w)\mathscr{G}_\frac{z}{2}(4\pi^2xy,w)\, dy\, dx\nonumber\\
&=\b^{w+\frac{1}{2}}\int_{0}^{\infty}y^wK_{\frac{z}{2}}(2\pi\b y)\left(2\pi\int_{0}^{\infty}f(x, z, w)\mathscr{G}_\frac{z}{2}(4\pi^2xy,w)\, dx\right)\, dy\nonumber\\
&=\b^{w+\frac{1}{2}}\int_{0}^{\infty}y^wK_{\frac{z}{2}}(2\pi\a y)f(y, z, w)\, dy,
\end{align*}}
where in the last step, we used Theorem \ref{selfrecigkk}. The interchange of the order of integration in the second step is justified because of \eqref{kbessasym}, \eqref{genkoshk} and \eqref{hypbessasym}.
\end{proof}

\section{A transformation involving series of the sums-of-squares function $r_k(n)$}

In this section, we prove Theorem \ref{rk(n)identity}. We begin with a lemma which gives an asymptotic estimate for ${}_\mu{K}_{\nu}\left(z,w\right)$ as $z\to\infty$. This will be crucially required in the results established in both this and the next sections. 
\begin{lemma}\label{asymukhalfo}
Let $|\arg(-z)|\leq\pi$ and $m\in\mathbb{N}\cup\{0\}$. As $z\to\infty$,
\begin{align*}
{}_\mu{K}_{\nu}\left(z,w\right)&= \frac{\pi 2^{3\mu+2\nu+2w}}{\sin(\pi\nu)z^{w+2\mu+\nu+1}}\left\{A_m(\mu,\nu,w;z)-B_m(\mu,\nu,w;z)+O_{\mu,\nu,w}\left(|z|^{-2m-2}\right)\right\},
\end{align*}
where $A_m(\mu, \nu,w;z)$ is defined in \eqref{am} and
\begin{align}\label{bm}
B_m(\mu,\nu,w;z):=\sum_{k=0}^m\frac{(-1)^{-\mu-\nu-w-\frac{1}{2}}\Gamma\left(\mu+\nu+w+\frac{1}{2}+k\right)}{k!\Gamma\left(-\mu-\nu-k\right)\Gamma\left(\frac{1}{2}-\mu-w-k\right)}\left(\frac{z}{2}\right)^{-2k}.
\end{align}
\end{lemma}
\begin{proof}
From \cite[p. 412, Formula 16.11.8]{nist}, for $|\arg(-z)|\leq\pi$, as $z\to\infty$,
\begin{align}\label{1f21}
\frac{\Gamma(a_1)}{\Gamma(b_1)\Gamma(b_2)}{}_1F_2\left(a_1;b_1,b_2;z\right)\sim H_{1,2}(-z)+E_{1,2}\left(-ze^{\pi i}\right)+E_{1,2}\left(-ze^{-\pi i}\right),
\end{align}
where
\begin{align}\label{1f2h1o}
H_{1,2}(z)=\sum_{k=0}^\infty\frac{(-1)^k}{k!}\frac{\Gamma(a_1+k)}{\Gamma(b_1-a_1-k)\Gamma(b_2-a_1-k)}z^{-a_1-k},
\end{align}
$\lambda=a_1-b_1-b_2+\frac{1}{2}$, and 
\begin{align}\label{1f2e1o}
E_{1,2}\left(z\right)=\frac{2^{-\lambda-1/2}}{\sqrt{2\pi}}e^{2z^{\frac{1}{2}}}(2z^{\frac{1}{2}})^\lambda\sum_{k=0}^\infty c_k\left(2z^{\frac{1}{2}}\right)^{-k},
\end{align}
with $c_0=1$, $b_3=1$ and 
\begin{align}\label{cek}
\ c_k&=-\frac{1}{4k}\sum_{m=0}^{k-1} c_me_{k,m},\nonumber\\
e_{k,m}&=\sum_{j=1}^3\frac{(a_1-bj)}{\displaystyle\prod_{\substack{\ell=1\\ \ell\neq j
                  }}^{3}(b_\ell-b_j)}(1-\lambda-2b_j+m)_{k+2-m}.
\end{align}
Let $a_1=\mu+w+\frac{1}{2},\ b_1=w+\frac{1}{2}-\nu,\ b_2=1-\nu$ and replace $z$ by $\frac{-z^2}{4}$ in \eqref{1f2h1o} and \eqref{1f2e1o} to get
\begin{align}\label{h120}
H_{1,2}\left(-\frac{z^2}{4}\right)=\sum_{k=0}^m\frac{(-1)^k}{k!}\frac{\Gamma\left(\mu+w+\frac{1}{2}+k\right)(-z^2/4)^{-\mu-w-\frac{1}{2}-k}}{\Gamma\left(-\nu-\mu-k\right)\Gamma\left(\frac{1}{2}-\nu-\mu-w-k\right)}+O_{\mu,w,\nu}\left(|z|^{-2\mu-2w-2m-3}\right),
\end{align}
Taking $e^{-\pi i}=-1$, we see that
\begin{align}\label{e1201}
E_{1,2}\left(-\frac{z^2}{4}e^{\pi i}\right)=\frac{2^{-\mu-2\nu}}{\sqrt{2\pi}}z^{\mu+2\nu-\frac{1}{2}}e^z\sum_{k=0}^\infty c_kz^{-k},
\end{align}
where
$c_0=1$ and 
\begin{align}\label{ck}
e_{k,m}&=-\frac{1}{4k}\Bigg[\frac{\left(\frac{1}{2}-\mu-2w+m\right)_{k+2-m}(\mu+\nu)}{\left(\frac{1}{2}-w\right)\left(\frac{1}{2}-w+\nu\right)}+\frac{\left(m-\frac{1}{2}-\mu\right)_{k+2-m}\left(\mu+w+\nu-\frac{1}{2}\right)}{\left(w-\frac{1}{2}\right)\nu}\nonumber\\
&\quad+\frac{\left(m-\frac{1}{2}-\mu-2\nu\right)_{k+2-m}\left(\mu+w-\frac{1}{2}\right)}{\left(w-\nu-\frac{1}{2}\right)(-\nu)}\Bigg].
\end{align}
From \eqref{1f21}, \eqref{h120} and \eqref{e1201}, as $z\to\infty$,
\begin{align}\label{1f101}
&\frac{\Gamma\left(\mu+w+\frac{1}{2}\right)}{\Gamma\left(w+\frac{1}{2}-\nu\right)\Gamma\left(1-\nu\right)}{}_1F_2\left(\mu+w+\frac{1}{2}; w+\frac{1}{2}-\nu, 1-\nu;\frac{z^2}{4}\right)\nonumber\\
&=H_{1,2}\left(-\frac{z^2}{4}\right)+E_{1,2}\left(-\frac{z^2}{4}e^{-\pi i}\right)+E_{1,2}\left(-\frac{z^2}{4}e^{\pi i}\right),
\end{align}
where 
\begin{align*}
E_{1,2}\left(-\frac{z^2}{4}e^{-\pi i}\right)&=\frac{2^{-\mu-2\nu}}{\sqrt{2\pi}}e^{\sqrt{-z^2e^{-\pi i}}}\left(\sqrt{-z^2e^{-\pi i}}\right)^{\mu+2\nu-\frac{1}{2}}+O\left(\frac{e^{-z}}{x^{\frac{1}{2}-\mu}}\right)\nonumber\\
&=\frac{2^{-\mu-2\nu}}{\sqrt{2\pi}}e^{-z}\left(-z\right)^{\mu+2\nu-\frac{1}{2}}+O\left(\frac{e^{-z}}{z^{\frac{3}{2}-\mu-2\nu}}\right).
\end{align*}
Similarly, letting $a_1=\mu+\nu+w+\frac{1}{2},\ b_1=w+\frac{1}{2},\ b_2=1+\nu$, replacing $z$ by $\frac{z^2}{4}$ in \eqref{1f21} and simplifying, we see that as $z\to\infty$,
\begin{align}\label{1f2asym2}
&\frac{\Gamma\left(\mu+\nu+w+\frac{1}{2}\right)}{\Gamma\left(w+\frac{1}{2}\right)\Gamma\left(1+\nu\right)}{}_1F_2\left(\mu+\nu+w+\frac{1}{2};w+\frac{1}{2},1+\nu;\frac{z^2}{4}\right)\nonumber\\
&=H_{1,2}^{*}\left(-\frac{z^2}{4}\right)+E_{1,2}^{*}\left(-\frac{z^2}{4}e^{-\pi i}\right)+E_{1,2}^{*}\left(-\frac{z^2}{4}e^{\pi i}\right).
\end{align}
Here
\begin{align*}
H_{1,2}^{*}\left(-\frac{z^2}{4}\right)&=\sum_{k=0}^m\frac{(-1)^k}{k!}\frac{\Gamma\left(\mu+\nu+w+\frac{1}{2}+k\right)(-z^2/4)^{-\mu-\nu-w-\frac{1}{2}-k}}{\Gamma\left(-\mu-\nu-k\right)\Gamma\left(\frac{1}{2}-\mu-w-k\right)}\nonumber\\
&\qquad+O_{\mu,\nu,w}\left(|z|^{-2\mu-2\nu-2w-2m-3}\right),
\end{align*}
and 
\begin{align*}
E_{1,2}^{*}\left(-\frac{z^2}{4}e^{\pi i}\right)=\frac{2^{-\mu}}{\sqrt{2\pi}}z^{\mu-\frac{1}{2}}e^z\sum_{k=0}^\infty c_kz^{-k},
\end{align*}
with $c_k$ defined in \eqref{cek} and $e_{k,m}$ resulting from \eqref{ck}. Also,
\begin{align*}
E_{1,2}^{*}\left(-\frac{z^2}{4}e^{-\pi i}\right)&=\frac{2^{-\mu}}{\sqrt{2\pi}}e^{-z}\left(-z\right)^{\mu-\frac{1}{2}}+O\left(\frac{e^{-z}}{z^{\frac{3}{2}-\mu-2\nu}}\right).
\end{align*}
From the definition of ${}_\mu{K}_{\nu}\left(z,w\right)$ in \eqref{def2varbessel}, \eqref{1f101}, \eqref{1f2asym2} and the observation that the terms containing $e^z$ completely cancel out each other, we deduce that
{\allowdisplaybreaks\begin{align*}
{}_\mu{K}_{\nu}\left(z,w\right)&= \frac{\pi z^w2^{\mu+\nu-1}}{\sin(\pi\nu)}\Bigg\{\sum_{k=0}^m\frac{(-1)^{-\mu-w-\frac{1}{2}}\Gamma\left(\mu+w+\frac{1}{2}+k\right)}{k!\Gamma\left(-\nu-\mu-k\right)\Gamma\left(\frac{1}{2}-\nu-\mu-w-k\right)}\left(\frac{z}{2}\right)^{-2\mu-\nu-2w-2k-1}\nonumber\\
&\quad -\sum_{k=0}^m\frac{(-1)^{-\mu-\nu-w-\frac{1}{2}}\Gamma\left(\mu+\nu+w+\frac{1}{2}+k\right)}{k!\Gamma\left(-\mu-\nu-k\right)\Gamma\left(\frac{1}{2}-\mu-w-k\right)}\left(\frac{z}{2}\right)^{-2\mu-\nu-2w-2k-1}\nonumber\\
&\quad+O_{\mu,\nu,w}\left(|z|^{-2\mu-2\nu-2w-2m-3}\right)+\frac{(-1)^{\mu+2\nu-\frac{1}{2}}2^{-\mu-\nu}}{\sqrt{2\pi}}e^{-z}z^{\mu+\nu-\frac{1}{2}}\nonumber\\
&\quad-\frac{(-1)^{\mu-\frac{1}{2}}2^{-\mu-\nu}}{\sqrt{2\pi}}e^{-z}z^{\mu+\nu-\frac{1}{2}}\Bigg\}.
\end{align*}}
Upon using the definitions of \eqref{am} and \eqref{bm}, we see that the proof of Lemma \ref{asymukhalfo} is complete.
\end{proof}

\begin{lemma}\label{mukhalfintegraliszero}
Let\footnote{This lemma is valid even if $k$ is complex such that $\textup{Re}(k)>0$.} $k\in\mathbb{N}$. Then,
\begin{align*}
\int_0^\infty x^{\frac{k}{4}-\frac{1}{2}}{}_\mu K_{\frac{1}{2}}\left(\pi^2x,\frac{k}{4}\right)\ dx=\begin{cases}
0,\hspace{24mm}\text{if}\hspace{1mm}\textup{Re}(\mu)>-\frac{1}{2},\\
\frac{1}{\sqrt{2}}\pi^{\frac{-k-1}{2}}\Gamma\left(\frac{k}{2}\right),\hspace{1mm}\text{if}\hspace{1mm}\mu=-\frac{1}{2}.
\end{cases}
\end{align*}
\end{lemma}
\begin{proof}
First let $\textup{Re}(\mu)>-\frac{1}{2}$. The idea is to introduce the exponential dampening factor $e^{-\pi^2x^2y}, y>0,$ in the integrand, evaluate the resulting integral explicitly, then let $y\to0$ and finally invoke Lebesgue's dominated convergence theorem.

From \cite[p.~814, Formula 7.522.5]{grad}, for $p<q$ and Re$(s)>0$, we have
\begin{align}\label{intofpfq}
\int_0^\infty e^{-x}x^{s-1}{}_pF_q(a_1,...,a_p;b_1,...,b_q;\d x)\ dx=\Gamma(s){}_{p+1}F_q(s,a_1,...a_p;b_1,...,b_q;\d).
\end{align}
First replace $x$ by $\pi^2yx^2$ and then let $\d=\pi^2/4y,\ p=1,\ q=2$ in the above equation so that
\begin{align}\label{intof1f2}
\int_0^\infty e^{-\pi^2x^2y}x^{2s-1}{}_1F_2\left(a_1;b_1,b_2;\frac{\pi^4x^2}{4}\right)\ dx=\frac{\Gamma(s)y^{-s}}{2\pi^{2s}}{}_2F_2\left(s,a_1;b_1,b_2;\frac{\pi^2}{4y}\right).
\end{align}
Letting $s=\frac{k}{4}$ and $a_1=\mu+\frac{k}{4}+\frac{1}{2},\ b_1=\frac{k}{4}$ and $b_2=\frac{1}{2}$ in the above equation, we get
\begin{align}\label{intof1f2evl1}
\int_0^\infty e^{-\pi^2x^2y}x^{\frac{k}{2}-1}{}_1F_2\left(\mu+\frac{k}{4}+\frac{1}{2};\frac{k}{4},\frac{1}{2};\frac{\pi^4x^2}{4}\right)\, dx
&=\frac{\Gamma\left(\frac{k}{4}\right)}{2\pi^{\frac{k}{2}}}y^{-\frac{k}{4}}{}_1F_1\left(\mu+\frac{k}{4}+\frac{1}{2};\frac{1}{2};\frac{\pi^2}{4y}\right).
\end{align}
Again use \eqref{intof1f2} with $s=\frac{k}{4}+\frac{1}{2}$ and $a_1=\mu+\frac{k}{4}+1,\ b_1=\frac{k}{4}+\frac{1}{2}$ and $b_2=\frac{3}{2}$ to get
\begin{align}\label{intof1f2evl2}
\int_0^\infty e^{-\pi^2x^2y}x^{\frac{k}{2}}{}_1F_2\left(\mu+\frac{k}{4}+1;\frac{k}{4}+\frac{1}{2},\frac{3}{2};\frac{\pi^4x^2}{4}\right)\, dx=\frac{\Gamma\left(\frac{k}{4}+\frac{1}{2}\right)}{2\pi^{\frac{k}{2}+1}}y^{-\frac{k}{4}-\frac{1}{2}}{}_1F_1\left(\mu+\frac{k}{4}+1;\frac{3}{2};\frac{\pi^2}{4y}\right).
\end{align}
Therefore from \eqref{intof1f2evl1}, \eqref{intof1f2evl2} and the definition of ${}_\mu K_{\nu}(z, w)$ in \eqref{def2varbessel}, we have
\begin{align*}
\int_0^\infty x^{\frac{k}{4}-\frac{1}{2}}e^{-\pi^2x^2y}{}_\mu K_{\frac{1}{2}}\left(\pi^2x,\frac{k}{4}\right)\ dx&=\sqrt{\pi}2^{\mu-1}y^{-\frac{k}{4}}\Bigg\{\frac{\Gamma\left(\mu+\frac{k}{4}+\frac{1}{2}\right)}{\pi}{}_1F_2\left(\mu+\frac{k}{4}+\frac{1}{2};\frac{1}{2};\frac{\pi^2}{4y}\right)	\nonumber\\
&\quad-\Gamma\left(\mu+\frac{k}{4}+1\right)y^{-\frac{1}{2}}{}_1F_1\left(\mu+\frac{k}{4}+1;\frac{3}{2};\frac{\pi^2}{4y}\right)\Bigg\}.
\end{align*}
Now express the right-hand side of the above equation in terms of the Tricomi confluent hypergeometric function $U(a;c;z)$ defined by
\cite[p.~176, Equation (7.14)]{temme}
\begin{align}\label{tricomiu}
U(a;c;z):=\frac{\Gamma(1-c)}{\Gamma(a-c+1)}{}_1F_1(a;c;z)+\frac{\Gamma(c-1)}{\Gamma(a)}z^{1-c}{}_1F_1(a-c+1;2-c;z)
\end{align}
so as to get
\begin{align}\label{intxeku}
\int_0^\infty x^{\frac{k}{4}-\frac{1}{2}}e^{-\pi^2x^2y}{}_\mu K_{\frac{1}{2}}\left(\pi^2x,\frac{k}{4}\right)\ dx&=\frac{y^{-\frac{k}{4}}}{\sqrt{\pi}2^{\mu+\frac{k}{2}+1}}\Gamma\left(2\mu+\frac{k}{2}+1\right)U\left(\mu+\frac{k}{4}+\frac{1}{2};\frac{1}{2};\frac{\pi^2}{4y}\right).
\end{align}
Let $y\to 0$ on both sides of \eqref{intxeku}. Using Lemma \ref{asymukhalfo} with $\nu=1/2, w=k/4$ and $z$ replaced by $\pi^2x$ and then invoking Lebesgue's dominated convergence theorem in order to interchange the order of limit and summation, we have
\begin{align}\label{intxku}
\int_0^\infty x^{\frac{k}{4}-\frac{1}{2}}{}_\mu K_{\frac{1}{2}}\left(\pi^2x,\frac{k}{4}\right)\ dx&=\frac{\Gamma\left(2\mu+\frac{k}{2}+1\right)}{\sqrt{\pi}2^{\mu+\frac{k}{2}+1}}\lim_{y\to0}\left\{y^{-\frac{k}{4}}U\left(\mu+\frac{k}{4}+\frac{1}{2};\frac{1}{2};\frac{\pi^2}{4y}\right)\right\}.
\end{align} 
From \cite[p.~174, Equation (7.13)]{temme}, as $z\to\infty$,
\begin{align}\label{asyofu}
U(a;c;z)\sim z^{-a}\sum_{n=0}^\infty \frac{(a)_n(a+c-1)_n}{n!}(-z)^n.
\end{align}
Let $a=\mu+\frac{k}{4}+\frac{1}{2},\ c=\frac{1}{2}$ and $z=\frac{\pi^2}{4y}$ in \eqref{asyofu}; then as $y\to0$,
\begin{align}\label{limitofu}
U\left(\mu+\frac{k}{4}+\frac{1}{2};\frac{1}{2};\frac{\pi^2}{4y}\right)\sim \left(\frac{4y}{\pi^2}\right)^{\mu+\frac{k}{4}+\frac{1}{2}}.
\end{align}
Therefore from \eqref{intxku} and \eqref{limitofu}, for $\mathrm{Re}(\mu)>-\frac{1}{2}$, we arrive at
\begin{align*}
\int_0^\infty x^{\frac{k}{4}-\frac{1}{2}}{}_\mu K_{\frac{1}{2}}\left(\pi^2x,\frac{k}{4}\right)\ dx=0.
\end{align*}
It remains to prove the result for $\mu=-1/2$. From \eqref{mueq-nu}, we have
\begin{align*}
{}_{-\frac{1}{2}} K_{\frac{1}{2}}\left(\pi^2x,\frac{k}{4}\right)=(\pi^2x)^{\frac{k}{4}}K_{\frac{1}{2}}(\pi^2x)
=\frac{1}{\sqrt{2}}\pi^{\frac{k}{2}-\frac{1}{2}}x^{\frac{k}{4}-\frac{1}{2}}e^{-\pi^2x},
\end{align*}
where we used \eqref{khalfz}. Hence,
\begin{align*}
\int_0^\infty x^{\frac{k}{4}-\frac{1}{2}}{}_{-\frac{1}{2}}K_{\frac{1}{2}}\left(\pi^2x,\frac{k}{4}\right)\ dx=\frac{1}{\sqrt{2}}\pi^{\frac{-k-1}{2}}\Gamma\left(\frac{k}{2}\right),
\end{align*}
using the well-known integral representation for Gamma function. This proves Lemma \ref{mukhalfintegraliszero}.
\end{proof}
In order to prove Theorem \ref{rk(n)identity}, we now employ the following transformation of Guinand \cite[p.~264, Equation (10.7)]{guinandconcord} (also known to Popov \cite[Equation (3)]{popov}) in a rigorous formulation given in \cite[Theorem 1.5]{bdkz}. The phrase `$f(x)$ is an integral' means $f$ is an integral of some function, that is, $f$ can be written in the form $f(x)=\int_{a}^{x}h(t)\, dt$ for some function $h$ and $-\infty\leq a<x$.
\begin{theorem}\label{guirknsuminfthm}
Let $k$ be a positive integer greater than $3$ and let $m=\left\lfloor\frac{1}{2}k\right\rfloor-1$. Let $F(x), F'(x)$, $F''(x), \dots, F^{(2m-1)}(x)$ be integrals, and $F(x), xF'(x), x^2F''(x), \dots, x^{2m}F^{(2m)}(x)$ belong to $L^{2}(0,\infty)$. Moreover, as $x\to\infty$, let
\begin{equation}\label{Fbigo}
F(x)=O_{k}\left(x^{-\frac{k}{4}-\frac{1}{2}-\tau}\right),
\end{equation}
for some fixed $\tau>0$.
Let the function $G$ be defined by
\begin{equation}\label{gee}
G(y)=\pi\int_{0}^{\infty}F(t)J_{\frac{k}{2}-1}(2\pi\sqrt{yt})\, dt,
\end{equation}
and assume that it satisfies
\begin{equation}\label{Gbigo}
G(y)=O_{k}\left(y^{-\frac{k}{4}-\frac{1}{2}-\tau}\right),
\end{equation}
for $\tau>0$, as $y\to\infty$. Then
{\allowdisplaybreaks\begin{align}\label{guirknsuminfallk}
\sum_{n=1}^{\infty}r_k(n)n^{\frac{1}{2}-\frac{k}{4}}F(n)-\frac{\pi^{\frac{k}{2}}}{\G(\frac{k}{2})}\int_{0}^{\infty}
x^{\frac{k}{4}-\frac{1}{2}}F(x)\, dx=\sum_{n=1}^{\infty}r_k(n)n^{\frac{1}{2}-\frac{k}{4}}G(n)-\frac{\pi^{\frac{k}{2}}}{\G(\frac{k}{2})}\int_{0}^{\infty}
x^{\frac{k}{4}-\frac{1}{2}}G(x)\, dx.
\end{align}}
For $k=2$ and $3$, \eqref{guirknsuminfallk} holds if $F$ is continuous on $[0,\infty)$,  $F(x), xF'(x)\in L^{2}(0,\infty)$, and  $F$ satisfies \eqref{Fbigo}, and if $G$ is defined in \eqref{gee} and  satisfies \eqref{Gbigo}.
\end{theorem}
\begin{proof}[Theorem \textup{\ref{rk(n)identity}}][]
We first prove the result for $z>0$ and then extend it to $\textup{Re}(z)>0$ by analytic continuation. The idea is to invoke Theorem \ref{guirknsuminfthm} with 
\begin{equation}\label{effx}
F(x)=x^{\mu+\frac{k}{4}+\frac{1}{2}}K_{\mu}(xz).
\end{equation}
To that end, we need to show that all of the hypotheses of Theorem \ref{guirknsuminfthm} are satisfied. 

Since $F$ is infinitely differentiable on $(0,\infty)$, all of its derivatives, in particular, $F(x)$, $F'(x)$, $\dots, F^{(2m-1)}(x)$, are integrals. We next show that $F(x), xF'(x), x^2F''(x), \dots, x^{2m}F^{(2m)}(x)$ belong to $L^{2}(0,\infty)$ provided $\textup{Re}(\mu)>-\frac{k}{8}-\frac{1}{2}$. 

First, \eqref{kbessasym} clearly implies that the convergence of the integral of $\left|x^{n}F^{(n)}(x)\right|^2$ at the upper limit of integration is always secured. Next, from \eqref{kbessasym0} and \cite[p.~36, Formula \textbf{1.14.1.1}]{brch}
\begin{equation*}
\frac{d^n}{dx^{n}}K_{\nu}(xz)=\left(\frac{-z}{2}\right)^{n}\sum_{k=0}^{n}\binom{n}{k}K_{\nu\pm 2k\mp n}(xz),
\end{equation*}
for any $n\in\mathbb{N}\cup\{0\}$, as $x\to 0$, we have
\begin{equation*}
x^{n}F^{(n)}(x)=O_{\mu, k, n}\left(x^{2\mu+\frac{k}{4}+\frac{1}{2}}\right)+O_{\mu, k, n}\left(x^{\frac{k}{4}+\frac{1}{2}}\right).
\end{equation*}
This implies $F(x), xF'(x), x^2F''(x), \dots, x^{2m}F^{(2m)}(x)$ belong to $L^{2}(0,\infty)$, provided $\textup{Re}(\mu)\geq-\frac{k}{8}-\frac{1}{2}$. Further, \eqref{sumbesselj}, \eqref{besseli} and \eqref{kbesse} imply that $F$ is continuous on $[0,\infty)$.   

Replace $x$ by $\pi^2x/z$ in Theorem \ref{koshlyakovg-wthm} and $t$ by $tz$, then let $\nu=\frac{1}{2}$ and $w=\frac{k}4$ in the resulting equation, and use \eqref{gkkhalf}, so that for $\mathrm{Re}(\mu)>-\frac{k}{4}-\frac{1}{2}$,
\begin{align}\label{efftransform}
\int_0^\infty t^{\mu+\frac{k}{4}+\frac{1}{2}}K_\mu(tz)J_{\frac{k}{2}-1}(2\pi\sqrt{xt})\ dt=\frac{1}{z^{\mu+\frac{k}{4}+\frac{3}{2}}}{}_\mu K_{\frac{1}{2}}\left(\frac{\pi^2x}{z},\frac{k}{4}\right).
\end{align}
From \eqref{gee}, \eqref{effx} and \eqref{efftransform},
\begin{equation*}
G(x)=\frac{\pi}{z^{\mu+\frac{k}{4}+\frac{3}{2}}}{}_\mu K_{\frac{1}{2}}\left(\frac{\pi^2x}{z},\frac{k}{4}\right).
\end{equation*}
Equation \eqref{kbessasym} clearly shows that the bound in \eqref{Fbigo} holds for our choice of $F$ in \eqref{effx}. Also, Lemma \ref{asymukhalfo} with $m=0$ implies that as $x\to\infty$, $G(x)=O_{\mu, k, z}\left(x^{-\frac{k}{4}-2\mu-\frac{3}{4}}\right)=O_{\mu, k, z}\left(x^{-\frac{k}{4}-\frac{1}{2}-\tau}\right)$ for $\tau>0$, provided Re$(\mu)	>-1/2$. But when $\mu=-1/2$, along with our assumption $\nu=1/2$ and \eqref{mueq-nu}, we see that $G(x)=\left(\frac{\pi^2x}{z}\right)^{k/4}K_\frac{1}{2}\left(\frac{\pi^2x}{z}\right)=O_{k, z}\left(x^{-\frac{k}{4}-\frac{1}{2}-\tau}\right)$ for $\tau>0$, as can be clearly seen from \eqref{kbessasym}. Hence, \eqref{Gbigo} is also satisfied when Re$(\mu)>-1/2$ or $\mu=-1/2$. 

From the above analysis, we conclude that the hypotheses of Theorem \ref{guirknsuminfthm} are satisfied for any integer $k\geq2$ whenever Re$(\mu)>-1/2$ or $\mu=-1/2$. Therefore invoking Theorem \ref{guirknsuminfthm}, Lemma \ref{mukhalfintegraliszero} and \eqref{kbessmel}, we find that
\begin{align*}
&\sum_{n=1}^\infty r_k(n)n^{\mu+1}K_{\mu}(nz)-\frac{\pi^{\frac{k}{2}}}{\Gamma\left(\frac{k}{2}\right)}2^{\mu+\frac{k}{2}-1}z^{-\mu-\frac{k}{2}-1}\Gamma\left(\frac{k+2}{4}\right)\Gamma\left(\mu+\frac{1}{2}+\frac{k}{4}\right)\nonumber\\
&\qquad=\frac{\pi}{z^{\mu+\frac{k}{4}+\frac{3}{2}}}\sum_{n=1}^\infty r_k(n)n^{\frac{1}{2}-\frac{k}{4}}{}_\mu K_{\frac{1}{2}}\left(\frac{\pi^2n}{z},\frac{k}{4}\right)-\frac{\pi^{\frac{k}{2}}}{\G\left(\frac{k}{2}\right)}\mathfrak{R}(\mu, k, z),
\end{align*}
where $\mathfrak{R}(\mu, k, z)$ is defined in \eqref{rmuk}. The transformation in \eqref{rk(n)identityeqn} now follows for $z>0$ upon using \eqref{dupl} in the second expression on the left-hand side of the above equation. 

To see that the result is valid for Re$(z)>0$, note first that the series on the left-hand side of \eqref{rk(n)identityeqn} is analytic for Re$(z)>0$ as can be seen from \eqref{kbessasym}. Also, Lemma \ref{asymukhalfo} with $m=0$ and the bound $r_k(n)=O_{k}\left(n^{\frac{k}{2}-1+\epsilon}\right)$ for any integer $k\geq2$ and $\epsilon>0$, see for example, \cite[Equation (2.9)]{bdkz}, imply that 
\begin{equation*}
r_k(n)n^{\frac{1}{2}-\frac{k}{4}}{}_\mu K_{\frac{1}{2}}\left(\frac{\pi^2n}{z},\frac{k}{4}\right)=O_{\mu, k, z}\left(\frac{1}{n^{2\textup{Re}(\mu)+2-\epsilon}}\right),
\end{equation*}
which clearly shows that as long as Re$(\mu)\geq-1/2+\delta$, for any $\delta>0$, the series 
\begin{equation*}
\sum_{n=1}^\infty r_k(n)n^{\frac{1}{2}-\frac{k}{4}}{}_\mu K_{\frac{1}{2}}\left(\frac{\pi^2n}{z},\frac{k}{4}\right)
\end{equation*}
converges uniformly in Re$(z)>0$. Since the summand of this series is clearly analytic in Re$(z)>0$, the series itself represents an analytic function in Re$(z)>0$ as long as Re$(\mu)>-1/2$. Moreover, \eqref{mueq-nu} implies that the same is true for $\mu=-1/2$. Since all other functions in \eqref{rk(n)identityeqn} are clearly analytic for Re$(z)>0$, by analytic continuation, \eqref{rk(n)identityeqn} holds for Re$(z)>0$.
\end{proof}

\begin{proof}[Corollary \textup{\ref{rkncor}}][]
Let $\mu=-1/2$ in Theorem \ref{rk(n)identity}, use \eqref{mueq-nu} and simplify.
\end{proof}

\section{A transformation formula associated with the generalized divisor function $\sigma_a(n)$}\label{gdf}

This section is devoted to proving Theorem \ref{main} and several corollaries resulting from it. We first obtain some lemmas which are interesting in their own right. 

\begin{lemma}\label{integralevaluation}
	Let $y>0$. For $\mathrm{Re}(a)>-1$, we have
{\allowdisplaybreaks	\begin{align}\label{integralevaluation1}
&\frac{2\pi}{y\sin\left(\frac{\pi a}{2}\right)}\int_{0}^{\infty}x^{\frac{ a}{2}} \Bigg\{\frac{x^{-\frac{a}{2}}(2\pi)^{-a}}{\Gamma(1-a)} {}_1F_2\left(1;\frac{1-a}{2},1-\frac{a}{2};\frac{4\pi^4x^2}{y^2} \right)-\left(\frac{4\pi^2x}{y^2}\right)^\frac{a}{2}\cosh\left(\frac{4\pi^2x}{y}\right)\Bigg\}\mathrm{d}x\nonumber\\
&=\frac{2^{-2-a}\pi^{-a}}{\Gamma(-a)}\sec\left(\frac{\pi a}{2}\right).
\end{align}}
\end{lemma}

\begin{proof}
Note that one cannot separately integrate each of the two expressions in the integrand as both diverge. As in Lemma \ref{mukhalfintegraliszero}, the idea is to introduce the exponential dampening factor $e^{-ux^2}, u>0$, and evaluate the new integral. Then at the end, we let $u\to 0$ and appeal to Lebesgue's dominated convergence theorem. 

Let $p=1,\ q=2$, $a_1=1, b_1=\frac{1}{2}(1-a)$, $b_2=1-\frac{1}{2}a$ in \eqref{intofpfq}, replace $x$ by $x^2u$, and then let $s=\frac{1}{2},\ \d=\frac{4\pi^4}{uy^2}, u>0,$ in the resulting equation to arrive at
\begin{align}\label{intexp1f2}
\int_0^\infty e^{-ux^2}{}_1F_2\left(1;\frac{1-a}{2},1-\frac{a}{2};\frac{4\pi^4x^2}{y^2} \right)\ dx=\frac{1}{2}\sqrt{\frac{\pi}{u}}{}_2F_2\left(\frac{1}{2},1;\frac{1-a}{2},1-\frac{a}{2};\frac{4\pi^4}{uy^2}\right).
\end{align}
For $\mathrm{Re}(s)>0 $ and $\mathrm{Re}(c)>0$, we have \cite[p.~47, Equation (5.30)]{ober}
\begin{align}\label{mtofcosexp}
\int_0^\infty x^{s-1}e^{-cx^2}\cos(bx)\ dx=\frac{1}{2}c^{-\frac{s}{2}}\Gamma\left(\frac{s}{2}\right)e^{-\frac{b^2}{4c}}{}_1F_1\left(\frac{1-s}{2};\frac{1}{2};\frac{b^2}{4c}\right).
\end{align}
Let $s=a+1,\ c=u$ and $b=\frac{i4\pi^2}{y}$ in \eqref{mtofcosexp} and apply Kummer's transformation ${}_1F_{1}(\a_1;\a_2;z)=e^{z}{}_1F_{1}(\a_2-\a_1;\a_2;-z)$ so as to get, for Re$(a)>-1$,
\begin{align}\label{intxexpcosh}
\int_0^\infty x^ae^{-ux^2}\cosh\left(\frac{4\pi^2x}{y}\right)\ dx=\frac{1}{2}u^{-\frac{a}{2}-\frac{1}{2}}\Gamma\left(\frac{a+1}{2}\right){}_1F_1\left(\frac{1+a}{2};\frac{1}{2};\frac{4\pi^4}{uy^2}\right),
\end{align}
From \eqref{intexp1f2} and \eqref{intxexpcosh},
{\allowdisplaybreaks\begin{align}\label{beforelimit}
&\int_{0}^{\infty}x^{\frac{ a}{2}} e^{-ux^2} \Bigg\{\frac{x^{-\frac{a}{2}}(2\pi)^{-a}}{\Gamma(1-a)} {}_1F_2\left(1;\frac{1-a}{2},1-\frac{a}{2};\frac{4\pi^4x^2}{y^2} \right)-\left(\frac{4\pi^2x}{y^2}\right)^\frac{a}{2}\cosh\left(\frac{4\pi^2x}{y}\right)\Bigg\}\ dx\nonumber\\
&=\frac{(2\pi)^{-a}}{\Gamma(1-a)}\frac{1}{2}\sqrt{\frac{\pi}{u}}{}_2F_2\left(\frac{1}{2},1;\frac{1-a}{2},1-\frac{a}{2};\frac{4\pi^4}{uy^2}\right)\nonumber\\
&\qquad-\left(\frac{4\pi^2}{y^2}\right)^\frac{a}{2}\frac{1}{2}u^{-\frac{a}{2}-\frac{1}{2}}\Gamma\left(\frac{a+1}{2}\right){}_1F_1\left(\frac{1+a}{2};\frac{1}{2};\frac{4\pi^4}{uy^2}\right).
\end{align}}
For $\alpha\neq\mathbb{Z}\cup\{0\}$ and $-\frac{3}{2}\pi<\arg(z)<\frac{\pi}{2}$, Kim \cite{kim} has obtained the following asymptotic expansion as $z\to\infty$:
\begin{align}\label{kimresult}
{}_2F_2(1,\alpha;\rho_1,\rho_2;z)\sim\frac{\Gamma(\rho_1)\Gamma(\rho_2)}{\Gamma(\alpha)}\left(K_{22}(z)+L_{22}(-z)\right),\ 
\end{align}
where, with $ \nu=1+\alpha-\rho_1-\rho_2$,
{\allowdisplaybreaks\begin{align*}
K_{22}(z)&=z^\nu e^z{}_2F_0\left(\rho_1-\alpha,\rho_2-\alpha;-;\frac{1}{z}\right),\nonumber\\
L_{22}(z)&=z^{-1}\frac{\Gamma(\alpha-1)}{\Gamma(\rho_1-1)\Gamma(\rho_2-1)}{}_3F_1\left(1,2-\rho_1,2-\rho_2;2-\alpha;\frac{1}{z}\right)\nonumber\\
&\quad+z^{-\alpha}\frac{\Gamma(\alpha)\Gamma(1-\alpha)}{\Gamma(\rho_1-\alpha)\Gamma(\rho_2-\alpha)}{}_2F_0\left(1+\alpha-\rho_1,1+\alpha-\rho_2;-;\frac{1}{z}\right).
\end{align*}}
Hence letting $\alpha=\frac{1}{2},\ \rho_1=\frac{1-a}{2},\ \rho_2=1-\frac{a}{2}$ and $z=\frac{4\pi^4}{uy^2}$ in \eqref{kimresult}, we find that as $u\to0$,
{\allowdisplaybreaks\begin{align}\label{final2f2}
&{}_2F_2\left(1,\frac{1}{2};\frac{1-a}{2},1-\frac{a}{2};\frac{4\pi^4}{uy^2}\right)\nonumber\\
&=\frac{\Gamma\left(\frac{1-a}{2}\right)\Gamma\left(1-\frac{a}{2}\right)}{\Gamma\left(\frac{1}{2}\right)}\Bigg\{\left(\frac{4\pi^4}{uy^2}\right)^a\exp\left(\frac{4\pi^4}{uy^2}\right){}_2F_0\left(-\frac{a}{2},\frac{1-a}{2};-;\frac{uy^2}{4\pi^4}\right)\nonumber\\
&\quad+\frac{uy^2\sqrt{\pi}}{2\pi^4\Gamma\left(\frac{-a}{2}\right)\Gamma\left(\frac{-a-1}{2}\right)}+O_{a,y}(u^2)+\left(-\frac{4\pi^4}{uy^2}\right)^{-\frac{1}{2}}\frac{\pi}{\Gamma\left(-\frac{a}{2}\right)\Gamma\left(\frac{1-a}{2}\right)}+O_{a,y}(u^{3/2})\Bigg\}.
\end{align}}
From \cite[p.~189, Exercise (7.7)]{temme}, for $-\frac{3}{2}\pi<\arg(z)<\frac{\pi}{2}$, as $z\to\infty$,
\begin{align}\label{1f1asym}
{}_1F_1(b;c;z)\sim\frac{e^zz^{b-c}\Gamma(c)}{\Gamma(b)}\sum_{n=0}^\infty \frac{(c-b)_n(1-b)_n}{n!}z^{-n}+\frac{e^{-\pi ib}z^{-b}\Gamma(c)}{\Gamma(c-b)}\sum_{n=0}^\infty\frac{(b)_n(1+b-c)_n}{n!}(-z)^{-n}.
\end{align}
Let $b=\frac{1+a}{2},\ c=\frac{1}{2}$ and $z=\frac{4\pi^4}{uy^2}$ in \eqref{1f1asym} and observe that the infinite series in the first expression on the right-hand side is nothing but ${}_2F_0\left(-\frac{a}{2},\frac{1-a}{2};-;\frac{4\pi^4}{uy^2}\right)$, so that, as $u\to0$,
\begin{align}\label{1f1asym1}
{}_1F_1\left(\frac{1+a}{2};\frac{1}{2};\frac{4\pi^4}{uy^2}\right)&=\frac{\sqrt{\pi}\exp\left(\frac{4\pi^4}{uy^2}\right)}{\Gamma\left(\frac{1+a}{2}\right)}\left(\frac{4\pi^4}{uy^2}\right)^{\frac{a}{2}}{}_2F_0\left(-\frac{a}{2},\frac{1-a}{2};-;\frac{4\pi^4}{uy^2}\right)\nonumber\\
&\quad+\sqrt{\pi}\frac{\exp\left(\frac{-\pi i(a+1)}{2}\right)}{\Gamma\left(-\frac{a}{2}\right)}\left(\frac{4\pi^4}{uy^2}\right)^{\frac{-a-1}{2}}+O_{a,y}\left(u^{\frac{1}{2}(a+3)}\right).
\end{align}
Using \eqref{final2f2} and \eqref{1f1asym1} and observing that the expressions involving ${}_2F_{0}$ completely cancel out, we deduce that
{\allowdisplaybreaks\begin{align}\label{expr}
&\lim_{u\to0}\Bigg\{\frac{(2\pi)^{-a}}{\Gamma(1-a)}\frac{1}{2}\sqrt{\frac{\pi}{u}}{}_2F_2\left(\frac{1}{2},1;\frac{1-a}{2},1-\frac{a}{2};\frac{4\pi^4}{uy^2}\right)\nonumber\\
&\qquad-\left(\frac{4\pi^2}{y^2}\right)^\frac{a}{2}\frac{1}{2}u^{-\frac{a}{2}-\frac{1}{2}}\Gamma\left(\frac{a+1}{2}\right){}_1F_1\left(\frac{1+a}{2};\frac{1}{2};\frac{4\pi^4}{uy^2}\right)\Bigg\}\nonumber\\
&=\frac{2^{-a-3}y}{i\pi^{a+1}\Gamma(-a)}-\frac{y\exp\left(\frac{-\pi i(a+1)}{2}\right)}{4\pi^{a+\frac{3}{2}}}\frac{\Gamma\left(\frac{a+1}{2}\right)}{\Gamma\left(-\frac{a}{2}\right)}\nonumber\\
&=\frac{y2^{-a-3}}{\Gamma(-a)\pi^{a+1}}\left(\frac{1}{i}+2i\frac{e^{\frac{-\pi ia}{2}}}{\cos\left(\frac{\pi a}{2}\right)}\right)\nonumber\\
&=\frac{y2^{-a-3}}{\Gamma(-a)\pi^{a+1}}\tan\left(\frac{\pi a}{2}\right),
\end{align}}
where in the penultimate step, we used \eqref{dupl} as well as the well-known variant of the reflection formula in \eqref{refl}, namely, $\Gamma\left(\frac{1+a}{2}\right)\Gamma\left(\frac{1-a}{2}\right)=\frac{\pi}{\cos\left(\frac{\pi a}{2}\right)}$.

Finally, let $u\to0$ in \eqref{beforelimit}, invoke \eqref{expr} and use dominated convergence theorem of Lebesgue to arrive at
{\allowdisplaybreaks\begin{align*}
&\int_{0}^{\infty}x^{\frac{ a}{2}} \Bigg\{\frac{x^{-\frac{a}{2}}(2\pi)^{-a}}{\Gamma(1-a)} {}_1F_2\left(1;\frac{1-a}{2},1-\frac{a}{2};\frac{4\pi^4x^2}{y^2} \right)-\left(\frac{4\pi^2x}{y^2}\right)^\frac{a}{2}\cosh\left(\frac{4\pi^2x}{y}\right)\Bigg\}\ dx\nonumber\\
&=\frac{y2^{-a-3}}{\Gamma(-a)\pi^{a+1}}\tan\left(\frac{\pi a}{2}\right).
\end{align*}}
Now multiply both sides of the above equation by $\frac{2\pi}{y\sin\left(\frac{\pi a}{2}\right)}$ to obtain \eqref{integralevaluation1}.
\end{proof}

\begin{lemma}\label{integralevaluationmain}
Let $y>0$. Then, for $-1<\mathrm{Re}(a)<1$, we have
\begin{align}\label{integralevaluationmaineqn}
\int_{0}^{\infty}x^{-\frac{ a}{2}} \Bigg\{\frac{x^{-\frac{a}{2}}(2\pi)^{-a}}{\Gamma(1-a)} {}_1F_2\left(1;\frac{1-a}{2},1-\frac{a}{2};\frac{4\pi^4x^2}{y^2} \right)-\left(\frac{4\pi^2x}{y^2}\right)^\frac{a}{2}\cosh\left(\frac{4\pi^2x}{y}\right)\Bigg\}\, dx=0.
\end{align}	
\end{lemma}

\begin{proof}
Let $p=1,\ q=2$, $a_1=1, b_1=\frac{1}{2}(1-a)$, $b_2=1-\frac{1}{2}a$ in \eqref{intofpfq}, replace $x$ by $x^2u$, and then in the resulting equation let $s=\frac{1-a}{2},\ \d=\frac{4\pi^4}{uy^2}$ so that for Re$(a)<1$,
\begin{align}\label{xexp1f1}
\int_0^\infty e^{-ux^2}x^{-a}{}_1F_2\left(1;\frac{1-a}{2},1-\frac{a}{2};\frac{4\pi^4x^2}{y^2} \right)\ dx=\frac{\Gamma\left(\frac{1-a}{2}\right)}{2u^{\frac{1-a}{2}}}{}_1F_1\left(1;1-\frac{a}{2};\frac{4\pi^4}{uy^2}\right).
\end{align}
Let $s=1,\ c=u$ and $b=\frac{4\pi^2i}{y}$ in \eqref{mtofcosexp} to get
\begin{align}\label{expcosh}
\int_0^\infty e^{-ux^2}\cosh\left(\frac{4\pi^2x}{y}\right)\ dx=\frac{1}{2}\sqrt{\frac{\pi}{u}}e^{\frac{4\pi^4}{uy^2}}.
\end{align}
From \eqref{xexp1f1} and \eqref{expcosh}, for Re$(a)<1$, we find that
\begin{align}\label{beforelimitzero}
&\int_{0}^{\infty}x^{-\frac{ a}{2}}e^{-ux^2} \Bigg\{\frac{x^{-\frac{a}{2}}(2\pi)^{-a}}{\Gamma(1-a)} {}_1F_2\left(1;\frac{1-a}{2},1-\frac{a}{2};\frac{4\pi^4x^2}{y^2} \right)-\left(\frac{4\pi^2x}{y^2}\right)^\frac{a}{2}\cosh\left(\frac{4\pi^2x}{y}\right)\Bigg\}\ dx\nonumber\\
&=\frac{(2\pi)^{-a}}{2\Gamma(1-a)u^{\frac{1-a}{2}}}\Gamma\left(\frac{1-a}{2}\right){}_1F_1\left(1;1-\frac{a}{2};\frac{4\pi^4}{uy^2}\right)-\frac{1}{2}\sqrt{\frac{\pi}{u}}\left(\frac{2\pi}{y}\right)^a e^{\frac{4\pi^4}{uy^2}}.
\end{align}
Let $b=1,\ c=1-\frac{a}{2}$ and $z=\frac{4\pi^4}{uy^2}$ in \eqref{1f1asym}. Observe that only the constant term of the first infinite series on the right-hand of \eqref{1f1asym} survives, so that as $u\to0$,
\begin{align}\label{1f1b}
{}_1F_1\left(1;1-\frac{a}{2};\frac{4\pi^4}{uy^2}\right)= \left(\frac{4\pi^4}{uy^2}\right)^{\frac{a}{2}}\Gamma\left(1-\frac{a}{2}\right) e^{\frac{4\pi^4}{uy^2}}+\frac{uy^2}{4\pi^4}e^{-\pi i}\frac{\Gamma\left(1-\frac{a}{2}\right)}{\Gamma\left(-\frac{a}{2}\right)}+O_{a,y}\left(u^2\right).
\end{align}
Using \eqref{1f1b}, we see that for $-1<\textup{Re}(a)<1$,
\begin{align}\label{limitzerozero}
&\lim_{u\to0}\left(\frac{(2\pi)^{-a}}{2\Gamma(1-a)u^{\frac{1-a}{2}}}\Gamma\left(\frac{1-a}{2}\right){}_1F_1\left(1;1-\frac{a}{2};\frac{4\pi^4}{uy^2}\right)-\frac{1}{2}\sqrt{\frac{\pi}{u}}\left(\frac{2\pi}{y}\right)^a e^{\frac{4\pi^4}{uy^2}}\right)\nonumber\\
&=\lim_{u\to0}\left(\frac{ay^2}{(2\pi)^{a+4}\G(1-a)}u^{\frac{a+1}{2}}+O_{a,y}\left((u^{\frac{1}{2}(a+3)}\right)\right)\nonumber\\
&=0.
\end{align}
Now let $u\to0$ in \eqref{beforelimitzero}, invoke Lebesgue's dominated convergence theorem and then employ \eqref{limitzerozero} to complete the proof of \eqref{integralevaluationmaineqn}. 
\end{proof}
We are now ready to prove our new transformation involving infinite series of the generalized divisor function $\sigma_a(n)$. We do this by employing the following analogue of Vorono\"{\dotlessi} summation formula for $\sigma_a(n)$ due to Guinand \cite[Equation (1)]{guinand}.
\begin{theorem}\label{guinandSummationFormula}
Let $-\frac{1}{2}<\textup{Re}(a)<\frac{1}{2}$. If $f(x)$ and $f'(x)$ are integrals, $f$ tends to zero as $x\rightarrow\infty,f(x),xf'(x)$ and $x^2f''(x)$ belong to $L^2(0,\infty)$, and
\begin{equation}\label{fgkt}
		g(x)=2\pi\int_0^\infty f(t)\left(\cos\left(\frac{\pi a}{2}\right)M_a(4\pi\sqrt{xt})-\sin\left(\frac{\pi a}{2}\right)J_a(4\pi\sqrt{xt})\right)dt,
		\end{equation}
		then the following transformation holds:
		\begin{align}\label{guinandSummationFormulaeqn}
		&\sum_{n=1}^\infty\sigma_{-a}(n)n^{\frac{a}{2}}f(n)-\zeta(1+a)\int_0^\infty x^{\frac{a}{2}}f(x)\, dx-\zeta(1-a)\int_0^\infty x^{-\frac{a}{2}}f(x)\, dx\nonumber\\
		&=\sum_{n=1}^\infty\sigma_{-a}(n)n^{\frac{a}{2}}g(n)-\zeta(1+a)\int_0^\infty x^{\frac{a}{2}}g(x)\, dx-\zeta(1-a)\int_0^\infty x^{-\frac{a}{2}}g(x)\, dx.
		\end{align}
\end{theorem}

\begin{proof}[Theorem \textup{\ref{main}}][]
We first prove the result for $0<a<\frac{1}{2}$ and $y>0$, and then extend it to $\textup{Re}(a)>-1$ and Re$(y)>0$ by analytic continuation.

Let $w=0, \mu=\frac{1}{2},\nu=\frac{a}{2}$ and replace $x$ by $4\pi^2 x$ in Theorem \ref{koshlyakovg-wthm} (or, equivalently, $\mu=\frac{1}{2},\nu=\frac{a}{2}$ and replace $x$ by $4\pi^2 x$ in \eqref{koshlyakovg-1}) and then use \eqref{khalfz} and \eqref{1st} in the resulting equation to obtain
\begin{align*}
&\int_{0}^{\infty}e^{-t}t^\frac{a}{2}\left\{\cos\left(\frac{\pi a }{2}\right)M_{a}(4\pi\sqrt{xt})-\sin\left(\frac{\pi a }{2}\right)J_{a}(4\pi\sqrt{xt})\right\}\ dt\nonumber\\
&=\frac{1}{\sin\left(\frac{\pi a}{2}\right)}\Bigg\{\frac{x^{-\frac{a}{2}}\pi^{\frac{1}{2}-a}}{\Gamma\left(1-\frac{a}{2}\right)\Gamma\left(\frac{1-a}{2}\right)} {}_1F_2\left(1;\frac{1-a}{2},1-\frac{a}{2};4\pi^4x^2 \right)-(4\pi^2x)^\frac{a}{2}\cosh(4\pi^2x)\Bigg\},
\end{align*}
where we used the fact that ${}_0F_1(-;1/2;x^2)=\cosh(2x).$ 

Replace $t$ by $yt$ and replace $x$ by $\frac{x}{y}$ in the above equation to get 
\begin{align}\label{kosh1/2}
&\int_{0}^{\infty}e^{-ty}t^\frac{a}{2}\left(\cos\left(\frac{\pi a }{2}\right)M_{a}(4\pi\sqrt{xt})-\sin\left(\frac{\pi a }{2}\right)J_{a}(4\pi\sqrt{xt})\right)\ dt\nonumber\\
&=\frac{1}{y\sin\left(\frac{\pi a}{2}\right)}\Bigg\{\frac{x^{-\frac{a}{2}}{(2\pi)}^{-a}}{\Gamma(1-a)} {}_1F_2\left(1;\frac{1-a}{2},1-\frac{a}{2};\frac{4\pi^4x^2}{y^2} \right)-\left(\frac{2\pi \sqrt{x}}{y}\right)^a\cosh\left(\frac{4\pi^2x}{y}\right)\Bigg\}.
\end{align}
We would like to now put $f(x):=e^{-xy}x^{\frac{a}{2}}$ in Theorem \ref{guinandSummationFormula}. To that end, observe that $f(x)$ and $f'(x)$ are integrals. Also, it is easy to see that $f(x)\to0$ as $x\to\infty$, and that $f(x),\ xf'(x)$ and $x^2f''(x)$ belong to $L^2(0,\infty)$. 

From \eqref{fgkt} and \eqref{kosh1/2}, we have
\begin{align}\label{g(x)2}
g(x)=\frac{2\pi}{y\sin\left(\frac{\pi a}{2}\right)}\Bigg\{\frac{x^{-\frac{a}{2}}(2\pi)^{-a}}{\Gamma(1-a)} {}_1F_2\left(1;\frac{1-a}{2},1-\frac{a}{2};\frac{4\pi^4x^2}{y^2} \right)-\left(\frac{2\pi \sqrt{x}}{y}\right)^a\cosh\left(\frac{4\pi^2x}{y}\right)\Bigg\}.
\end{align}
Our next task is to evaluate the four integrals in \eqref{guinandSummationFormulaeqn} for $0<a<1/2$. To that end, note that
{\allowdisplaybreaks\begin{align}\label{f(x)12}
\int_0^\infty x^{\frac{a}{2}}f(x)\, dx=y^{-a-1}\Gamma(a+1),\hspace{5mm} \int_0^\infty x^{-\frac{a}{2}}f(x)\, dx=\frac{1}{y}.
\end{align}}
Further, from \eqref{g(x)2}, Lemmas \ref{integralevaluation} and \ref{integralevaluationmain},
\begin{align}\label{g(x)12}
\int_{0}^{\infty}x^{\frac{a}{2}}g(x)\ dx=\frac{2^{-2-a}\pi^{-a}\sec\left(\frac{\pi a}{2}\right)}{\Gamma(-a)}, \hspace{5mm}
\int_{0}^{\infty}x^{-\frac{a}{2}}g(x)\ dx=0.
\end{align}
Thus, from \eqref{guinandSummationFormulaeqn}, \eqref{f(x)12} and \eqref{g(x)12} and the elementary fact $\sigma_{-a}(n)n^{a}=\sigma_{a}(n)$, for $0<a<\frac{1}{2}$ and $y>0$,
\begin{align*}
&\sum_{n=1}^\infty  \sigma_a(n)e^{-ny}-\frac{\zeta(1+a)\Gamma(1+a)}{y^{1+a}}-\frac{\zeta(1-a)}{y}+\zeta(1+a)\frac{2^{-2-a}\pi^{-a}\sec\left(\frac{\pi a}{2}\right)}{\Gamma(-a)}\nonumber\\
&=\frac{2\pi}{y\sin\left(\frac{\pi a}{2}\right)}\sum_{n=1}^\infty \sigma_{a}(n)\Bigg(\frac{(2\pi n)^{-a}}{\Gamma(1-a)} {}_1F_2\left(1;\frac{1-a}{2},1-\frac{a}{2};\frac{4\pi^4n^2}{y^2} \right) -\left(\frac{2\pi}{y}\right)^{a}\cosh\left(\frac{4\pi^2n}{y}\right)\Bigg).
\end{align*}
Now use the functional equation \eqref{zetafealt} with $s=-a$ twice to write $\zeta(1+a)$ in the above equation in terms of $\zeta(-a)$ thereby obtaining \eqref{maineqn} after simplification. This proves Theorem \ref{main} for $0<a<1/2$ and $y>0$.

Next, we show the validity of the transformation for $\textup{Re}(a)>-1$. From Remark \ref{sigmatrans}, the series on the right-hand side of \eqref{maineqn} is $\displaystyle\sum_{n=1}^{\infty}\sigma_a(n)n^{-\frac{a}{2}}{}_{\frac{1}{2}}K_{\frac{a}{2}}\left(\frac{4\pi^2n}{y},0\right)$. Now Lemma \ref{asymukhalfo} gives ${}_{\frac{1}{2}}K_{\frac{a}{2}}\left(\frac{4\pi^2n}{y},0\right)=O_{a,y}\left(n^{-\frac{1}{2}\textup{Re}(a)-2}\right)$ as $n\to\infty$. Along with the elementary bound $\sigma_a(n)n^{-\frac{a}{2}}=O\left(n^{\frac{1}{2}|\textup{Re}(a)|+\epsilon}\right)$ which is valid for every $\epsilon>0$, we see that $\displaystyle\sum_{n=1}^{\infty}\sigma_a(n)n^{-\frac{a}{2}}{}_{\frac{1}{2}}K_{\frac{a}{2}}\left(\frac{4\pi^2n}{y},0\right)$ converges uniformly as long as Re$(a)>-1$. Since the summand of the series is also analytic for Re$(a)>-1$, by Weierstrass' theorem on analytic functions, we see that it represents an analytic function of $a$ when Re$(a)>-1$. 

Since the left-hand side of \eqref{maineqn} is analytic for $\textup{Re}(a)>-1$, by the principle of analytic continuation, we see that \eqref{maineqn} holds for $\textup{Re}(a)>-1$ and $y>0$. Using similar method as above, both sides of \eqref{maineqn} are seen to be analytic, as a function of $y$, in Re$(y)>0$. Therefore by the principle of analytic continuation, \eqref{maineqn} holds for $\textup{Re}(a)>-1$ and $\textup{Re}(y)>0$.

%
%
\end{proof}
Theorem \ref{main}, which we just proved, can be analytically continued to $\mathrm{Re}(a)>-2m-3$, where $m$ is any non-negative integer. This is done in Theorem \ref{extendedid} and is crucial in deriving Corollary \ref{maineqn2cor} and Theorem \ref{trans-2m}.
\begin{proof}[Theorem \textup{\ref{extendedid}}][]
From Theorem \ref{main} and Remark \ref{sigmatrans}, we have $\textup{Re}(a)>-1$,
{\allowdisplaybreaks\begin{align*}
\sum_{n=1}^\infty  \sigma_a(n)e^{-ny}+\frac{1}{2}\left(\left(\tfrac{2\pi}{y}\right)^{1+a}\mathrm{cosec}\left(\tfrac{\pi a}{2}\right)+1\right)\zeta(-a)-\frac{1}{y}\zeta(1-a)=\tfrac{2\sqrt{2\pi}}{y^{1+\frac{a}{2}}}\sum_{n=1}^{\infty}\tfrac{\sigma_a(n)}{n^{\frac{a}{2}}}{}_{\frac{1}{2}}K_{\frac{a}{2}}\left(\frac{4\pi^2n}{y},0\right).
\end{align*}}
We rewrite the above identity as
\begin{align}\label{extendedideqn00}
&\sum_{n=1}^\infty  \sigma_a(n)e^{-ny}+\frac{1}{2}\left(\left(\frac{2\pi}{y}\right)^{1+a}\mathrm{cosec}\left(\frac{\pi a}{2}\right)+1\right)\zeta(-a)-\frac{\zeta(1-a)}{y}\nonumber\\
&=\frac{2\sqrt{2\pi}}{y^{1+\frac{a}{2}}}\sum_{n=1}^\infty\sigma_a(n)n^{-\frac{a}{2}}\left\{{}_{\frac{1}{2}}{K}_{\frac{a}{2}}\left(\frac{4\pi^2n}{y},0\right)-\frac{\pi2^{\frac{3}{2}+a}}{\sin\left(\frac{\pi a}{2}\right)}\left(\frac{4\pi^2n}{y}\right)^{-\frac{a}{2}-2}A_m\left(\frac{1}{2},\frac{a}{2},0;\frac{4\pi^2n}{y}\right)\right\}\nonumber\\
&\qquad+\frac{y\pi^{-a-5/2}}{2\sin\left(\frac{\pi a}{2}\right)}\sum_{n=1}^{\infty}\frac{\sigma_a(n)}{n^{a+2}}A_m\left(\frac{1}{2},\frac{a}{2},0;\frac{4\pi^2n}{y}\right).
\end{align} 
Using the definition of $A_m$ in \eqref{am} and employing \eqref{dupl}, we see that
\begin{align}\label{extendedideqn01}
&\frac{y\pi^{-a-5/2}}{2\sin\left(\frac{\pi a}{2}\right)}\sum_{n=1}^{\infty}\frac{\sigma_a(n)}{n^{a+2}}A_m\left(\frac{1}{2},\frac{a}{2},0;\frac{4\pi^2n}{y}\right)\nonumber\\
&=-\frac{y(2\pi)^{-a-3}}{\sin\left(\frac{\pi a}{2}\right)}\sum_{k=0}^m\frac{1}{\Gamma(-a-1-2k)}\left(\frac{4\pi^2}{y}\right)^{-2k}\sum_{n=1}^\infty \frac{\sigma_{a}(n)}{n^{a+2k+2}}\nonumber\\
&=-\frac{y(2\pi)^{-a-3}}{\sin\left(\frac{\pi a}{2}\right)}\sum_{k=0}^m\frac{\zeta(a+2k+2)\zeta(2k+2)}{\Gamma(-a-1-2k)}\left(\frac{4\pi^2}{y}\right)^{-2k},
\end{align}
where in the last step we used the formula \cite[p.~8, Equation (1.3.1)]{titch} $\sum_{n=1}^\infty\frac{\sigma_z(n)}{n^s}=\zeta(s)\zeta(s-z)$, valid for Re$(s)>\mathrm{max}\{1,\ 1+\mathrm{Re}(z)\}$. Hence from \eqref{extendedideqn00} and \eqref{extendedideqn01}, for Re$(a)>-1$,
\begin{align}\label{extendedideqn02}
&\sum_{n=1}^\infty  \sigma_a(n)e^{-ny}+\frac{1}{2}\left(\left(\frac{2\pi}{y}\right)^{1+a}\mathrm{cosec}\left(\frac{\pi a}{2}\right)+1\right)\zeta(-a)-\frac{\zeta(1-a)}{y}\nonumber\\
&=\frac{2\sqrt{2\pi}}{y^{1+\frac{a}{2}}}\sum_{n=1}^\infty\sigma_a(n)n^{-\frac{a}{2}}\left\{{}_{\frac{1}{2}}{K}_{\frac{a}{2}}\left(\frac{4\pi^2n}{y},0\right)-\frac{\pi2^{\frac{3}{2}+a}}{\sin\left(\frac{\pi a}{2}\right)}\left(\frac{4\pi^2n}{y}\right)^{-\frac{a}{2}-2}A_m\left(\frac{1}{2},\frac{a}{2},0;\frac{4\pi^2n}{y}\right)\right\}\nonumber\\
&\quad-\frac{y(2\pi)^{-a-3}}{\sin\left(\frac{\pi a}{2}\right)}\sum_{k=0}^m\frac{\zeta(a+2k+2)\zeta(2k+2)}{\Gamma(-a-1-2k)}\left(\frac{4\pi^2}{y}\right)^{-2k}.
\end{align}
Invoking Lemma \ref{asymukhalfo} and the fact $B_m\left(\frac{1}{2},\frac{a}{2},0;\frac{4\pi^2n}{y}\right)=0$, we see that for any $\epsilon>0$,
{\allowdisplaybreaks\begin{align*}
&\sigma_a(n)n^{-\frac{a}{2}}\left\{{}_{\frac{1}{2}}{K}_{\frac{a}{2}}\left(\frac{4\pi^2n}{y},0\right)-\frac{\pi2^{\frac{3}{2}+a}}{\sin\left(\frac{\pi a}{2}\right)}\left(\frac{4\pi^2n}{y}\right)^{-\frac{a}{2}-2}A_m\left(\frac{1}{2},\frac{a}{2},0;\frac{4\pi^2n}{y}\right)\right\}\nonumber\\
&=O_{a, y}\left(n^{-2m-4-\frac{1}{2}\textup{Re}(a)+\frac{1}{2}|\textup{Re}(a)|+\epsilon}\right).
\end{align*}}
This implies that the series
\begin{align*}
\sum_{n=1}^{\infty}\sigma_a(n)n^{-\frac{a}{2}}\left\{{}_{\frac{1}{2}}{K}_{\frac{a}{2}}\left(\frac{4\pi^2n}{y},0\right)-\frac{\pi2^{\frac{3}{2}+a}}{\sin\left(\frac{\pi a}{2}\right)}\left(\frac{4\pi^2n}{y}\right)^{-\frac{a}{2}-2}A_m\left(\frac{1}{2},\frac{a}{2},0;\frac{4\pi^2n}{y}\right)\right\}
\end{align*}
is uniformly convergent in Re$(a)\geq-2m-3+\epsilon$ for any $\epsilon>0$. Since the summand of the above series is analytic in Re$(a)>-2m-3$, we see by Weierstrass' theorem that this series represents an analytic function of $a$ for Re$(a)>-2m-3$.

Since the left-hand side of \eqref{extendedideqn02} as well as the finite sum on its right-hand side are also analytic in Re$(a)>-2m-3$, by the principle of analytic continuation, we see that \eqref{extendedideqn02} holds for Re$(a)>-2m-3$ for any $m\geq 0$. This completes the proof of Theorem \ref{extendedid}.
\end{proof}

\section{Transformation formulas for $\sum_{n=1}^{\infty}\sigma_a(n)e^{-ny}$ for $a$ even}\label{aeven}

Here, we derive some lemmas crucial in obtaining the transformations for $\sum_{n=1}^{\infty}\sigma_a(n)e^{-ny}$ for $a=2m$ as well as $-2m$.

\begin{lemma}\label{minusonelemma}
Let $\mathrm{Re}(w)>0$. Let $\textup{Shi}(z)$ and $\textup{Chi}(z)$ be defined in \eqref{shichi}. Then
\begin{equation}\label{zero}
\int_{0}^{\infty}\frac{t\cos t\, dt}{t^2+w^2}=\sinh(w)\mathrm{Shi}(w)-\cosh(w)\mathrm{Chi}(w).
\end{equation}
\end{lemma}
\begin{proof}
First, let $w>0$. From \cite[p.~395, Formula \textbf{2.5.9.12}]{Prudnikov},
\begin{equation}\label{one}
\int_{0}^{\infty}\frac{t\cos t\, dt}{t^2+w^2}=-\frac{1}{2}\left(e^{w}\textup{Ei}(-w)+e^{-w}\textup{Ei}(w)\right),
\end{equation}
where $\textup{Ei}(x)$ is the exponential integral given \cite[p.~788]{Prudnikov} for $x>0$ by $\textup{Ei}(x):=\int_{-\infty}^{x}e^{t}/t\, dt,$ or, as can be equivalently seen from \cite[p.~1]{je}, by $\textup{Ei}(-x):=-\int_{x}^{\infty}e^{-t}/t\, dt$.
Also, from \cite[p.~884, \textbf{8.214.1}, \textbf{8.214.2}]{grad},
\begin{align}
\textup{Ei}(x)=\gamma+\log(-x)+\sum_{k=1}^{\infty}\frac{x^k}{kk!}\hspace{5mm}(x<0),\hspace{4mm}
\textup{Ei}(x)=\gamma+\log x+\sum_{k=1}^{\infty}\frac{x^k}{kk!}\hspace{5mm}(x>0).\label{four}
\end{align}
Hence from \eqref{one} and \eqref{four}, for $w>0$,
{\allowdisplaybreaks\begin{align*}
&\int_{0}^{\infty}\frac{t\cos t\, dt}{t^2+w^2}\nonumber\\
&=-\frac{1}{2}\left\{e^w\left(\g+\log w+\sum_{k=1}^{\infty}\frac{(-w)^k}{kk!}\right)+e^{-w}\left(\g+\log w+\sum_{k=1}^{\infty}\frac{w^k}{kk!}\right)\right\}\nonumber\\
&=\frac{e^w}{2}\left\{\sum_{k=1}^{\infty}\frac{w^{2k-1}}{(2k-1)(2k-1)!}-\sum_{k=1}^{\infty}\frac{w^{2k}}{(2k)(2k)!}-\g-\log w\right\}\nonumber\\
&\quad+\frac{e^{-w}}{2}\left\{-\sum_{k=1}^{\infty}\frac{w^{2k-1}}{(2k-1)(2k-1)!}-\sum_{k=1}^{\infty}\frac{w^{2k}}{(2k)(2k)!}-\g-\log w\right\}\nonumber\\
&=\sinh(w)\mathrm{Shi}(w)-\cosh(w)\mathrm{Chi}(w),
\end{align*}}
where in the last step, we used the series representations of $\mathrm{Shi}(x)$ and $\mathrm{Chi}(x)$, namely,
\begin{align}\label{shichiseries}
\mathrm{Shi}(x)=\sum_{k=1}^{\infty}\frac{x^{2k-1}}{(2k-1)(2k-1)!},\hspace{5mm}
\mathrm{Chi}(x)=\g+\log x+\sum_{k=1}^{\infty}\frac{x^{2k}}{(2k)(2k)!},
\end{align}
which can be easily derived from their definitions in \eqref{shichi}.
Thus \eqref{zero} is established for $w>0$. Using \eqref{shichiseries}, we see that both sides of \eqref{zero} are analytic for Re$(w)>0$  Hence by analytic continuation, \eqref{zero} holds for Re$(w)>0$. 

\end{proof}

\begin{lemma}\label{closedformpsi}
Let $m\in\mathbb{N}\cup\{0\}$.  Then for\footnote{This result is actually true for all $z\in\mathbb{C}$. Note that at any non-positive real number $z$, the right-hand side has a removable singularity.} $\textup{Re}(z)>0$,
\begin{align}\label{closedformpsieqn}
\sum_{j=0}^\infty\frac{\psi(2j+2m+1)}{(m+1)_j\left(m+\frac{1}{2}\right)_j}z^{2j}&=\frac{\Gamma(2m+1)}{(2z)^{2m}}\Bigg\{\sinh(2z)\mathrm{Shi}(2z)-\cosh(2z)\mathrm{Chi}(2z)\nonumber\\
&\quad+\log(2z)\cosh(2z)\Bigg\}-\frac{\Gamma(2m+1)}{(2z)^{2}}\sum_{j=0}^{m-1}\frac{\psi(2m-2j-1)}{\Gamma(2m-2j-1)}(2z)^{-2j}.
\end{align}
\end{lemma}

\begin{proof}
Note that
\begin{align*}
\sum_{j=0}^\infty\frac{\psi(2j+2m+1)}{(m+1)_j\left(m+\frac{1}{2}\right)_j}z^{2j}&=\Gamma(2m+1)\sum_{j=0}^\infty\frac{\psi(2j+2m+1)}{\Gamma\left(2j+2m+1\right)}(2z)^{2j}.
\end{align*}
Changing the index of summation from $j$ to $k$ with $k=j+m$ on the right-hand side of the above equation, we see that
{\allowdisplaybreaks\begin{align}\label{suminsum}
&\sum_{j=0}^\infty\frac{\psi(2j+2m+1)}{(m+1)_j\left(m+\frac{1}{2}\right)_j}z^{2j}\nonumber\\
&=\frac{\Gamma(2m+1)}{(2z)^{2m}}\sum_{k=m}^\infty\frac{\psi(2k+1)}{\Gamma\left(2k+1\right)}(2z)^{2k}\nonumber\\
&=\frac{\Gamma(2m+1)}{(2z)^{2m}}\sum_{k=0}^\infty\frac{\psi(2k+1)}{\Gamma\left(2k+1\right)}(2z)^{2k}-\frac{\Gamma(2m+1)}{(2z)^{2}}\sum_{j=0}^{m-1}\frac{\psi(2m-2j-1)}{\Gamma\left(2m-2j-1\right)}(2z)^{-2j}.
\end{align}}
However, from \cite[Lemma 3.2]{dgkm} and Lemma \ref{minusonelemma}, for Re$(z)>0$,
\begin{align}\label{suminsinh}
\sum_{k=0}^\infty\frac{\psi(2k+1)}{\Gamma\left(2k+1\right)}(2z)^{2k}&=\sinh(2z)\mathrm{Shi}(2z)-\cosh(2z)\mathrm{Chi}(2z)+\log(2z)\cosh(2z).
\end{align}
Now substitute \eqref{suminsinh} in \eqref{suminsum} to arrive at \eqref{closedformpsieqn}.
\end{proof}

\subsection{A transformation for $\sum_{n=1}^{\infty}\sigma_{2m}(n)e^{-ny}, m\in\mathbb{N}$}\label{aeven2m}

\begin{lemma}\label{der1f2at2m}
Let $n\in\mathbb{N}$, $m\in\mathbb{N}\cup\{0\}$ and $y\in\mathbb{C}$. Then
\begin{align*}
&\frac{d}{da}\left(\frac{1}{\Gamma(1-a)}{}_1F_2\left(1;1-\frac{a}{2},\frac{1-a}{2};\frac{4\pi^4n^2}{y^2}\right)\right)\Bigg|_{a=2m}\nonumber\\
&=\left(\frac{4\pi^2n}{y}\right)^{2m}\Bigg\{\sinh\left(\frac{4\pi^2n}{y}\right)\mathrm{Shi}\left(\frac{4\pi^2n}{y}\right)-\cosh\left(\frac{4\pi^2n}{y}\right)\mathrm{Chi}\left(\frac{4\pi^2n}{y}\right)\nonumber\\
&\quad+\log\left(\frac{4\pi^2n}{y}\right)\cosh\left(\frac{4\pi^2n}{y}\right)+\sum_{j=1}^m(2j-1)!\left(\frac{4\pi^2n}{y}\right)^{-2j}\Bigg\}.
\end{align*}
\end{lemma}

\begin{proof}
Using the series definition of ${}_1F_2$ and \eqref{dupl} in the first step, we see that
\begin{align*}
\frac{d}{da}\left(\frac{1}{\Gamma(1-a)}{}_1F_2\left(1;1-\frac{a}{2},\frac{1-a}{2};\frac{4\pi^4n^2}{y^2}\right)\right)&=\frac{d}{da}\left(\sum_{k=0}^\infty\frac{2^{2k}}{\Gamma(1-a+2k)}\left(\frac{4\pi^4n^2}{y^2}\right)^{k}\right)\nonumber\\
&=\sum_{k=0}^\infty\frac{\psi(1-a+2k)}{\Gamma(1-a+2k)}\left(\frac{16\pi^4n^2}{y^2}\right)^{k}.
\end{align*}
We evaluate the derivative in the above equation at $a=2m$ to get
{\allowdisplaybreaks\begin{align*}
&\frac{d}{da}\left(\frac{1}{\Gamma(1-a)}{}_1F_2\left(1;1-\frac{a}{2},\frac{1-a}{2};\frac{4\pi^4n^2}{y^2}\right)\right)\Bigg|_{a=2m}\nonumber\\
&=\lim_{a\to 2m}\sum_{k=0}^{m-1}\frac{\psi(1+2k-a)}{\Gamma(1+2k-a)}\left(\frac{16\pi^4n^2}{y^2}\right)^{k}+\sum_{k=m}^{\infty}\frac{\psi(1+2k-2m)}{\Gamma(1+2k-2m)}\left(\frac{16\pi^4n^2}{y^2}\right)^{k}\nonumber\\
&=\left(\frac{16\pi^4n^2}{y^2}\right)^{m}\left(\sum_{j=1}^{m}\lim_{a\to 2m}\frac{\psi(1+2m-2j-a)}{\Gamma(1+2m-2j-a)}\left(\frac{16\pi^4n^2}{y^2}\right)^{-j}+\sum_{j=0}^{\infty}\frac{\psi(1+2j)}{\Gamma(1+2j)}\left(\frac{16\pi^4n^2}{y^2}\right)^{j}\right)\nonumber\\
&=\left(\frac{4\pi^2n}{y}\right)^{2m}\Bigg\{\sum_{j=1}^m(2j-1)!\left(\frac{4\pi^2n}{y}\right)^{-2j}+\sinh\left(\frac{4\pi^2n}{y}\right)\mathrm{Shi}\left(\frac{4\pi^2n}{y}\right)\nonumber\\
&\quad-\cosh\left(\frac{4\pi^2n}{y}\right)\mathrm{Chi}\left(\frac{4\pi^2n}{y}\right)+\log\left(\frac{4\pi^2n}{y}\right)\cosh\left(\frac{4\pi^2n}{y}\right)\Bigg\},
\end{align*}}
where in the last step, we used the fact that $\lim_{a\to 2m}\displaystyle\frac{\psi(1+2m-2j-a)}{\Gamma(1+2m-2j-a)}=(2j-1)!$ as well as Lemma \ref{closedformpsi} with $m=0$ and $z=2\pi^2n/y$. 
\end{proof}

We are now ready to prove Theorem \ref{a=2mcase}.

\begin{proof}[Theorem \textup{\ref{a=2mcase}}][]
First, using the series definition of ${}_1F_{2}$, we have
\begin{align}\label{9.12}
\lim_{a\to 2m}\frac{1}{\Gamma(1-a)}{}_1F_2\left(1;1-\frac{a}{2},\frac{1-a}{2};\frac{4\pi^4n^2}{y^2}\right)=\left(\frac{4\pi^2n}{y}\right)^{2m}\cosh\left(\frac{4\pi^2n}{y}\right).
\end{align}
Using the above identity, we see that right-hand side of\eqref{maineqn} has $\frac{0}{0}$ form. Also due to the term $\displaystyle\frac{\zeta(-a)}{\sin\left(\frac{\pi a}{2}\right)}$, the left-hand side of \eqref{maineqn} has $\frac{0}{0}$ form. Therefore in order to let $a\to2m$ in Theorem \ref{main}, we use L'Hopital's rule. First note that 
\begin{align}\label{9.13}
&\frac{d}{da}\left(\frac{(2\pi n)^{-a}}{\Gamma(1-a)} {}_1F_2\left(1;\frac{1-a}{2},1-\frac{a}{2};\frac{4\pi^4n^2}{y^2} \right) -\left(\frac{2\pi}{y}\right)^{a}\cosh\left(\frac{4\pi^2n}{y}\right)\right)\nonumber\\
&=-\frac{(2\pi n)^{-a}\log(2\pi n)}{\Gamma(1-a)}{}_1F_2\left(1;\frac{1-a}{2},1-\frac{a}{2};\frac{4\pi^4n^2}{y^2} \right)\nonumber\\
&\quad+(2\pi n)^{-a}\frac{d}{da}\left(\frac{1}{\Gamma(1-a)}{}_1F_2\left(1;\frac{1-a}{2},1-\frac{a}{2};\frac{4\pi^4n^2}{y^2} \right)\right)\nonumber\\
&\quad-\left(\frac{2\pi}{y}\right)^a\log\left(\frac{2\pi}{y}\right)\cosh\left(\frac{4\pi^2n}{y}\right).
\end{align}
Let $a\to 2m$ in \eqref{9.13}, then invoke Lemma \ref{der1f2at2m} with $z=\frac{2\pi^2n}{y}$ and then use \eqref{9.12} so as to arrive at
\begin{align}\label{additional}
&\frac{d}{da}\left(\frac{(2\pi n)^{-a}}{\Gamma(1-a)} {}_1F_2\left(1;\frac{1-a}{2},1-\frac{a}{2};\frac{4\pi^4n^2}{y^2} \right) -\left(\frac{2\pi}{y}\right)^{a}\cosh\left(\frac{4\pi^2n}{y}\right)\right)\Bigg|_{a=2m}\nonumber\\
&=\left(\frac{2\pi}{y}\right)^{2m}\Bigg\{\sinh\left(\frac{4\pi^2n}{y}\right)\mathrm{Shi}\left(\frac{4\pi^2n}{y}\right)-\cosh\left(\frac{4\pi^2n}{y}\right)\mathrm{Chi}\left(\frac{4\pi^2n}{y}\right)+\sum_{j=1}^m(2j-1)!\left(\frac{4\pi^2n}{y}\right)^{-2j}\Bigg\}.
\end{align}
We wish to take limit $a\to 2m$ on both sides of \eqref{maineqn}. First, the right-hand side is evaluated using L'Hopital's rule and \eqref{additional}.
{\allowdisplaybreaks\begin{align}\label{rhseqn}
&\lim_{a\to2m}\Bigg\{\frac{2\pi}{y\sin\left(\frac{\pi a}{2}\right)}\sum_{n=1}^\infty \sigma_{a}(n)\Bigg(\frac{(2\pi n)^{-a}}{\Gamma(1-a)} {}_1F_2\left(1;\frac{1-a}{2},1-\frac{a}{2};\frac{4\pi^4n^2}{y^2} \right) -\left(\frac{2\pi}{y}\right)^{a}\cosh\left(\frac{4\pi^2n}{y}\right)\Bigg)\Bigg\}\nonumber\\
&=(-1)^m\frac{2}{\pi}\left(\frac{2\pi}{y}\right)^{2m+1}\sum_{n=1}^\infty\sigma_{2m}(n)\Bigg\{\sinh\left(\frac{4\pi^2n}{y}\right)\mathrm{Shi}\left(\frac{4\pi^2n}{y}\right)\nonumber\\
&\quad-\cosh\left(\frac{4\pi^2n}{y}\right)\mathrm{Chi}\left(\frac{4\pi^2n}{y}\right)+\sum_{j=1}^m(2j-1)!\left(\frac{4\pi^2n}{y}\right)^{-2j}\Bigg\}.
\end{align}}
We now evaluate the left-hand side of \eqref{maineqn} as $a\to2m$. Using $\zeta(-2m)=0$, we have
\begin{align}\label{lhseqn}
&\lim_{a\to2m}\left\{\sum_{n=1}^\infty  \sigma_a(n)e^{-ny}+\frac{1}{2}\left(\left(\frac{2\pi}{y}\right)^{1+a}\mathrm{cosec}\left(\frac{\pi a}{2}\right)+1\right)\zeta(-a)-\frac{\zeta(1-a)}{y}\right\}\nonumber\\
&=\sum_{n=1}^\infty \sigma_{2m}(n)e^{-ny}-\frac{(-1)^m}{\pi}\left(\frac{2\pi}{y}\right)^{2m+1}\zeta'(-2m)-\frac{\zeta(1-2m)}{y}\nonumber\\
&=\sum_{n=1}^\infty \sigma_{2m}(n)e^{-ny}-\frac{(2m)!}{y^{2m+1}}\zeta(2m+1)+\frac{B_{2m}}{2my},
\end{align}
where in the last step we used \cite[Equation (1)]{ktyham}
\begin{align}\label{zetaprime}
\zeta'(-2m)=(-1)^{m}\frac{(2m)!}{2(2\pi)^{2m}}\zeta(2m+1)
\end{align}
as well as \cite[p.~264, Equation (17)]{apostol}
\begin{align}\label{zeta1-2m}
\zeta(1-2m)=-\frac{B_{2m}}{2m}.
\end{align}
Hence from \eqref{rhseqn} and \eqref{lhseqn}, we obtain \eqref{a=2midentity}.
\end{proof}

\subsection{A transformation for $\sum_{n=1}^{\infty}\sigma_{-2m}(n)e^{-ny}, m\in\mathbb{N}$}\label{aeven-2m}

\begin{lemma}\label{derivatveof1f2a=-2m}
Let $m\in\mathbb{N}\cup\{0\}$. Let $\textup{Shi}(z)$ and $\textup{Chi}(z)$ be defined in \eqref{shichi}. Then for $\textup{Re}(z)>0$,
\begin{align}\label{derivatveof1f2a=-2meqn}
&\frac{d}{da}{}_1F_2\left(1;\frac{1-a}{2},1-\frac{a}{2};z^2 \right)\Bigg|_{a=-2m}\nonumber\\
&=\frac{\Gamma(2m+1)}{(2z)^{2m}}\left\{\sinh(2z)\mathrm{Shi}(2z)-\cosh(2z)\mathrm{Chi}(2z)+\log(2z)\cosh(2z)\right\}\nonumber\\
&\quad-\frac{\Gamma(2m+1)}{(2z)^{2}}\sum_{j=0}^{m-1}\frac{\psi(2m-2j-1)}{\Gamma(2m-2j-1)}(2z)^{-2j}-\psi(2m+1){}_1F_2\left(1;m+1,m+\frac{1}{2};z^2\right).
\end{align}
\end{lemma}

\begin{proof}
Using the series definition of ${}_1F_2$, we have
\begin{align*}
	\frac{d}{da}{}_1F_2\left(1;\frac{1-a}{2},1-\frac{a}{2};z^2 \right)
    &=-\sum_{j=0}^{\infty}\frac{z^{2j}}{\left(\frac{1-a}{2}\right)_j\left(1-\frac{a}{2}\right)_j}\frac{d}{da}\left(\log\left(\frac{1-a}{2}\right)_j+\log\left(1-\frac{a}{2}\right)_j\right)	\nonumber\\	
     &=\frac{1}{2}\sum_{j=0}^{\infty}\frac{z^{2j}}{\left(\frac{1-a}{2}\right)_j\left(1-\frac{a}{2}\right)_j}\sum_{k=0}^{j-1}\left(\frac{1}{\frac{1-a}{2}+k}+\frac{1}{1-\frac{a}{2}+k}\right)\nonumber\\	
      &=\sum_{j=0}^{\infty}z^{2j}\frac{\psi(2j+1-a)-\psi(1-a)}{(1-\frac{a}{2})_j(\frac{1-a}{2})_j}.
      \end{align*}
Thus,
\begin{align*}
\frac{d}{da}{}_1F_2\left(1;\frac{1-a}{2},1-\frac{a}{2};z^2 \right)\Bigg|_{a=-2m}
=\sum_{j=0}^{\infty}z^{2j}\frac{\psi(2j+2m+1)}{(m+1)_j(m+\frac{1}{2})_j}-\psi(2m+1){}_1F_2\left(1;m+1,m+\frac{1}{2};z^2\right).
\end{align*}
Using Lemma \ref{closedformpsi} in the above equation, we arrive at \eqref{derivatveof1f2a=-2meqn}.
\end{proof}

\begin{lemma}\label{1f2at-2m}
Let $m\in\mathbb{N}\cup\{0\}$, $n\in\mathbb{C}$ and $y\in\mathbb{C}\backslash\{0\}$,
\begin{align}\label{1f2at-2meqn}
&\frac{1}{\Gamma(2m+1)}{}_1F_2\left(1;m+1,m+\frac{1}{2};\frac{4\pi^4n^2}{y^2}\right)\nonumber\\
&=\left(\frac{4\pi^2n}{y}\right)^{-2m}\cosh\left(\frac{4\pi^2n}{y}\right)-\frac{y^2}{16\pi^4n^2}\sum_{k=0}^{m-1}\frac{1}{\Gamma(2m-2k-1)}\left(\frac{4\pi^2n}{y}\right)^{-2k}.
\end{align}
\end{lemma}

\begin{proof}
Using the series definition of ${}_1F_{2}$, it is easy to see that
\begin{align}\label{1f2transrec}
{}_1F_2\left(a+1;b+1,c+1;x\right)=-\frac{bc}{x}\left({}_1F_2\left(a;b,c;x\right)-{}_1F_2\left(a+1;b,c;x\right)\right).
\end{align}
Use \eqref{1f2transrec} with $a=0,\ b=m,\ c=m-1/2$ and $x=4\pi^4n^2/y^2$ to see that
\begin{align*}
{}_1F_2\left(1;m+1,m+\frac{1}{2};\frac{4\pi^4n^2}{y^2}\right)&=-m\left(m-\frac{1}{2}\right)\left(\frac{4\pi^4n^2}{y^2}\right)^{-1}\left(1-{}_1F_2\left(1;m,m-\frac{1}{2};\frac{4\pi^4n^2}{y^2}\right)\right)
\end{align*}
Again invoke \eqref{1f2transrec} with $a=0,\ b=m-1,\ c=m-3/2$ and $x=4\pi^4n^2/y^2$ on the right-hand side of the above equation to get
{\allowdisplaybreaks\begin{align*}
&{}_1F_2\left(1;m+1,m+\frac{1}{2};\frac{4\pi^4n^2}{y^2}\right)\nonumber\\
&=-m\left(m-\frac{1}{2}\right)\left(\frac{4\pi^4n^2}{y^2}\right)^{-1}-m(m-1)\left(m-\frac{1}{2}\right)\left(m-\frac{3}{2}\right)\left(\frac{4\pi^4n^2}{y^2}\right)^{-2}\nonumber\\
&\qquad\times\left(1-{}_1F_2\left(1;m-1,m-\frac{3}{2};\frac{4\pi^4n^2}{y^2}\right)\right)\nonumber\\
&=\Gamma(m+1)\Gamma\left(m+\frac{1}{2}\right)\Bigg\{-\frac{1}{\Gamma(m)\Gamma\left(m-\frac{1}{2}\right)}\left(\frac{4\pi^4n^2}{y^2}\right)^{-1}-\frac{1}{\Gamma(m-1)\Gamma\left(m-\frac{3}{2}\right)}\left(\frac{4\pi^4n^2}{y^2}\right)^{-2}\nonumber\\
&\qquad+\frac{1}{\Gamma(m-1)\Gamma\left(m-\frac{3}{2}\right)}\left(\frac{4\pi^4n^2}{y^2}\right)^{-2}{}_1F_2\left(1;m-1,m-\frac{3}{2};\frac{4\pi^4n^2}{y^2}\right)\Bigg\}.
\end{align*}}
Iterating this process using \eqref{1f2transrec} and employing the elementary fact ${}_1F_2\left(1;1,\frac{1}{2};\frac{4\pi^4n^2}{y^2}\right)=\cosh\left(\frac{4\pi^2n}{y}\right)$, we deduce that
{\allowdisplaybreaks\begin{align*}
&{}_1F_2\left(1;m+1,m+\frac{1}{2};\frac{4\pi^4n^2}{y^2}\right)\nonumber\\
&=-\Gamma(m+1)\Gamma\left(m+\frac{1}{2}\right)\left(\sum_{k=0}^{m-1}\frac{(4\pi^4n^2/y^2)^{-(k+1)}}{\Gamma(m-k)\Gamma\left(m-k-\frac{1}{2}\right)}-\frac{1}{\sqrt{\pi}}\left(\frac{4\pi^4n^2}{y^2}\right)^{-m}\cosh\left(\frac{4\pi^2n}{y}\right)\right).
\end{align*}}
Using \eqref{dupl} twice and then dividing both sides by $\Gamma(2m+1)$, we arrive at \eqref{1f2at-2meqn}.
\end{proof}

We are now ready to prove Theorem \ref{trans-2m}.

\begin{proof}[Theorem \textup{\ref{trans-2m}}][]
Using \eqref{def2varbessel} and \eqref{am}, we rewrite \eqref{extendedideqn} as
\begin{align}\label{a=-2mbeforelimit}
&\sum_{n=1}^\infty  \sigma_a(n)e^{-ny}+\frac{1}{2}\zeta(-a)-\frac{\zeta(1-a)}{y}+\frac{1}{\sin\left(\frac{\pi a}{2}\right)}\Bigg\{\frac{1}{2}\left(\frac{2\pi}{y}\right)^{1+a}\zeta(-a)\nonumber\\
&\quad+y(2\pi)^{-a-3}\sum_{k=0}^{m-1}\frac{\zeta(a+2k+2)\zeta(2k+2)}{\Gamma(-a-1-2k)}\left(\frac{4\pi^2}{y}\right)^{-2k}\Bigg\}\nonumber\\
&=\frac{1}{\sin\left(\frac{\pi a}{2}\right)}\sum_{n=1}^\infty\sigma_a(n)\Bigg\{\frac{2\pi}{y}\left(\frac{(2\pi n)^{-a}}{\Gamma(1-a)} {}_1F_2\left(1;\frac{1-a}{2},1-\frac{a}{2};\frac{4\pi^4n^2}{y^2} \right) -\left(\frac{2\pi}{y}\right)^{a}\cosh\left(\frac{4\pi^2n}{y}\right)\right)\nonumber\\
&\quad+\frac{yn^{-a-2}}{2\pi^{a+\frac{5}{2}}}\sum_{k=0}^{m-1}\frac{(2\pi^2n/y)^{-2k}}{\Gamma\left(-\frac{a}{2}-\frac{1}{2}-k\right)\Gamma\left(-\frac{a}{2}-k\right)}\Bigg\}.
\end{align}
We wish to take $a\to-2m$ in  \eqref{a=-2mbeforelimit}.  We show that we have $\frac{0}{0}$ form on both sides of \eqref{a=-2mbeforelimit} at $a=-2m$. By invoking Lemma \ref{1f2at-2m}, we observe that
{\allowdisplaybreaks\begin{align}\label{rhszero}
&\Bigg[\frac{2\pi}{y}\left(\frac{(2\pi n)^{-a}}{\Gamma(1-a)} {}_1F_2\left(1;\frac{1-a}{2},1-\frac{a}{2};\frac{4\pi^4n^2}{y^2} \right) -\left(\frac{2\pi}{y}\right)^{a}\cosh\left(\frac{4\pi^2n}{y}\right)\right)\nonumber\\
&\quad+\frac{yn^{-a-2}}{2\pi^{a+\frac{5}{2}}}\sum_{k=0}^{m-1}\frac{(2\pi^2n/y)^{-2k}}{\Gamma\left(-\frac{a}{2}-\frac{1}{2}-k\right)\Gamma\left(-\frac{a}{2}-k\right)}\Bigg]_{a=-2m}=0.
\end{align}}
Also,
{\allowdisplaybreaks\begin{align}\label{lhszero}
&\left[\frac{1}{2}\left(\frac{2\pi}{y}\right)^{1+a}\zeta(-a)+y(2\pi)^{-a-3}\sum_{k=0}^{m-1}\frac{\zeta(a+2k+2)\zeta(2k+2)}{\Gamma(-a-1-2k)}\left(\frac{4\pi^2}{y}\right)^{-2k}\right]_{a=-2m}\nonumber\\
&=\frac{1}{2}\left(\frac{2\pi}{y}\right)^{1-2m}\zeta(2m)+\left(\frac{2\pi}{y}\right)^{1-2m}\zeta(0)\zeta(2m)\nonumber\\
&\qquad+y(2\pi)^{2m-3}\sum_{k=0}^{m-2}\frac{\zeta(-2m+2k+2)\zeta(2k+2)}{\Gamma(2m-1-2k)}\left(\frac{4\pi^2}{y}\right)^{-2k}\nonumber\\
&=y(2\pi)^{2m-3}\sum_{j=0}^{m-2}\frac{\zeta(-2j-2)\zeta(2m-2j-4)}{\Gamma(2j+3)}\left(\frac{4\pi^2}{y}\right)^{2j-2m+4}\nonumber\\
&=0,
\end{align}}%
where, in the first step we separated out the term $k=m-1$ from the summation to cancel out the first term of the right-hand side, in the penultimate step we let $k=m-2-j$, and in the last step we used the fact that $\zeta(s)$ has zeros at negative even integers.

By using \eqref{rhszero}, \eqref{lhszero} and the fact that $\sin(m\pi)=0$, we see that we have $\frac{0}{0}$ form on both sides of \eqref{a=-2mbeforelimit}. Therefore we use L'Hopital's rule after letting $a\to-2m$ on both sides of \eqref{a=-2mbeforelimit}. To that end,
\begin{align*}
&\frac{d}{da}\left(\frac{1}{2}\left(\frac{2\pi}{y}\right)^{1+a}\zeta(-a)+y(2\pi)^{-a-3}\sum_{k=0}^{m-1}\frac{\zeta(a+2k+2)\zeta(2k+2)}{\Gamma(-a-1-2k)}\left(\frac{4\pi^2}{y}\right)^{-2k}\right)\nonumber\\
&=\frac{1}{2}\left(\frac{2\pi}{y}\right)^{1+a}\log\left(\frac{2\pi}{y}\right)\zeta(-a)-\frac{1}{2}\left(\frac{2\pi}{y}\right)^{1+a}\zeta'(-a)-y(2\pi)^{-a-3}\log(2\pi)\\
&\quad\times\left\{\frac{\zeta(a+2m)\zeta(2m)}{\G(-a+1-2m)}\left(\frac{4\pi^2}{y}\right)^{2-2m}+\sum_{k=0}^{m-2}\frac{\zeta(a+2k+2)\zeta(2k+2)}{\Gamma(-a-1-2k)}\left(\frac{4\pi^2}{y}\right)^{-2k}\right\}\\
&\quad+y(2\pi)^{-a-3}\left\{\zeta(2m)\left(\frac{4\pi^2}{y}\right)^{2-2m}\left(\frac{\zeta'(a+2m)}{\G(-a+1-2m)}+\frac{\zeta(a+2m)\psi(-a+1-2m)}{\G(-a+1-2m)}\right)\right.\\
&\quad\left.+\sum_{k=0}^{m-2}\zeta(2k+2)\left(\frac{4\pi^2}{y}\right)^{-2k}\left(\frac{\zeta'(a+2k+2)}{\Gamma(-a-1-2k)}+\frac{\zeta(a+2k+2)\psi(-a-1-2k)}{\Gamma(-a-1-2k)}\right)\right\},
\end{align*}
where we separate the $k=m-1$ term in each of the two sums. Thus,
{\allowdisplaybreaks\begin{align}\label{dofrhs}
&\frac{d}{da}\left(\frac{1}{2}\left(\frac{2\pi}{y}\right)^{1+a}\zeta(-a)+y(2\pi)^{-a-3}\sum_{k=0}^{m-1}\frac{\zeta(a+2k+2)\zeta(2k+2)}{\Gamma(-a-1-2k)}\left(\frac{4\pi^2}{y}\right)^{-2k}\right)\Bigg|_{a=-2m}\nonumber\\
&=\frac{1}{2}\left(\frac{2\pi}{y}\right)^{1-2m}\log\left(\frac{2\pi}{y}\right)\zeta(2m)-\frac{1}{2}\left(\frac{2\pi}{y}\right)^{1-2m}\zeta'(2m)+\frac{\gamma}{2}\left(\frac{2\pi}{y}\right)^{1-2m}\zeta(2m)\nonumber\\
&\quad+y(2\pi)^{2m-3}\sum_{k=0}^{m-2}\frac{\zeta'(-2m+2k+2)\zeta(2k+2)}{\Gamma(2m-1-2k)}\left(\frac{4\pi^2}{y}\right)^{-2k}\nonumber\\
&=\left(\frac{2\pi}{y}\right)^{1-2m}\left(\frac{1}{2}\log\left(\frac{2\pi}{y}\right)\zeta(2m)-\frac{1}{2}\zeta'(2m)+\frac{\gamma}{2}\zeta(2m)\right)\nonumber\\
&\quad+\frac{1}{2}\left(\frac{y}{2\pi}\right)^{2m-3}\sum_{k=0}^{m-2}(-1)^{k+1}\zeta(2k+3)\zeta(2m-2k-2)\left(\frac{2\pi}{y}\right)^{2k},
\end{align}}%
where we employed \eqref{zetaprime} in the last step.
Next, invoking Lemma \ref{derivatveof1f2a=-2m} with $z=\frac{2\pi^2n}{y}$ and Lemma \ref{1f2at-2m}, we see upon simplification that
{\allowdisplaybreaks\begin{align}\label{rhsderivative}
&\frac{d}{da}\Bigg[\frac{2\pi}{y}\left(\frac{(2\pi n)^{-a}}{\Gamma(1-a)} {}_1F_2\left(1;\frac{1-a}{2},1-\frac{a}{2};\frac{4\pi^4n^2}{y^2} \right) -\left(\frac{2\pi}{y}\right)^{a}\cosh\left(\frac{4\pi^2n}{y}\right)\right)\nonumber\\
&\quad+\frac{y(2\pi)^{-3}(2\pi n)^{-a}}{n^2}\sum_{k=0}^{m-1}\frac{1}{\Gamma\left(-a-2k-1\right)}\left(\frac{4\pi^2n}{y}\right)^{-2k}\Bigg]\Bigg|_{a=-2m}\nonumber\\
&=\left(\frac{y}{2\pi}\right)^{2m-1}\left\{\sinh\left(\frac{4\pi^2n}{y}\right)\mathrm{Shi}\left(\frac{4\pi^2n}{y}\right)-\cosh\left(\frac{4\pi^2n}{y}\right)\mathrm{Chi}\left(\frac{4\pi^2n}{y}\right)\right\}.
\end{align}}
Now let $a\to-2m$ in \eqref{a=-2mbeforelimit}, interchange the order of limit and summation on both sides, which is permissible because of uniform convergence, and then use \eqref{dofrhs} and \eqref{rhsderivative} so as to obtain 
\begin{align*}
&\sum_{n=1}^\infty \sigma_{-2m}(n) e^{-ny}+\frac{1}{2}\zeta(2m)-\frac{1}{y}\zeta(2m+1)+\frac{2(-1)^{m}}{\pi}\Bigg\{\left(\frac{2\pi}{y}\right)^{1-2m}\Bigg(\frac{1}{2}\log\left(\frac{2\pi}{y}\right)\zeta(2m)\nonumber\\
&-\frac{1}{2}\zeta'(2m)+\frac{\gamma}{2}\zeta(2m)\Bigg)+\frac{1}{2}\left(\frac{y}{2\pi}\right)^{2m-3}\sum_{k=0}^{m-2}(-1)^{k+1}\zeta(2k+3)\zeta(2m-2k-2)\left(\frac{2\pi}{y}\right)^{2k}\Bigg\}\nonumber\\
&=\frac{2(-1)^{m}}{\pi}\left(\frac{y}{2\pi}\right)^{2m-1}\sum_{n=1}^\infty\sigma_{-2m}(n)\left\{\sinh\left(\frac{4\pi^2n}{y}\right)\mathrm{Shi}\left(\frac{4\pi^2n}{y}\right)-\cosh\left(\frac{4\pi^2n}{y}\right)\mathrm{Chi}\left(\frac{4\pi^2n}{y}\right)\right\}.
\end{align*}
Finally, to get to \eqref{resultinsincosh} from the above identity, note that $-\frac{1}{y}\zeta(2m+1)$ can be inducted into the finite sum thereby resulting in its upper index of summation being $m-1$.

We now simplify \eqref{resultinsincosh} further and represent it in the equivalent form given in \eqref{equiforma=-2mano}. Note that the right-hand side of \eqref{resultinsincosh} can be written using \eqref{zero} in the form
{\allowdisplaybreaks\begin{align}\label{almost}
&\frac{2(-1)^{m}}{\pi}\left(\frac{y}{2\pi}\right)^{2m-1}\sum_{n=1}^\infty\sigma_{-2m}(n)\left\{\sinh\left(\frac{4\pi^2n}{y}\right)\mathrm{Shi}\left(\frac{4\pi^2n}{y}\right)-\cosh\left(\frac{4\pi^2n}{y}\right)\mathrm{Chi}\left(\frac{4\pi^2n}{y}\right)\right\}\nonumber\\
&=\frac{2(-1)^{m}}{\pi}\left(\frac{y}{2\pi}\right)^{2m-1}\sum_{n=1}^\infty\sigma_{-2m}(n)\int_0^\infty \frac{t\cos(t)}{t^2+\left(\frac{4\pi^2n}{y}\right)^2}\ dt\nonumber\\
&=\frac{2(-1)^{m}}{\pi}\left(\frac{y}{2\pi}\right)^{2m-1}\sum_{d=1}^\infty d^{-2m}\sum_{k=1}^\infty\int_0^\infty \frac{t\cos(t)}{t^2+\left(\frac{4\pi^2d}{y}\right)^2k^2}\ dt\nonumber\\
&=\frac{(-1)^{m}}{\pi}\left(\frac{y}{2\pi}\right)^{2m-1}\sum_{d=1}^\infty d^{-2m}\left\{\log\left(\frac{2\pi d}{y}\right)-\frac{1}{2}\left(\psi\left(\frac{2\pi id}{y}\right)+\psi\left(-\frac{2\pi id}{y}\right)\right)\right\},
\end{align}}
where we used \eqref{dgkmresult} with $u=\frac{4\pi^2d}{y}$ in the last step.
Using the fact that $\zeta'(s)=-\sum_{n=1}^{\infty}\log n/n^s$ whenever Re$(s)>1$, we see that
\begin{align}\label{sumlog}
\sum_{n=1}^\infty  \frac{1}{n^{2m}}\log\left(\frac{2\pi n}{y}\right)=\zeta(2m)\log\left(\frac{2\pi}{y}\right)-\zeta'(2m).
\end{align}
Finally, substitute \eqref{almost} and \eqref{sumlog} in \eqref{resultinsincosh}, replace $k$ by $m-1-k$ in the finite sum of the resulting identity and then use Euler's formula in \eqref{zeta(2m)} to arrive at \eqref{equiforma=-2mano} upon simplification.
\end{proof}
We now use \eqref{equiforma=-2mano} to obtain a nice companion to Ramanujan's formula \eqref{rameqn}. See also \eqref{berndtanalogue}.
\begin{proof}[Corollary \textup{\ref{companion}}][]
Let $y=2\a, \a\b=\pi^2$ in \eqref{equiforma=-2mano} and multiply both sides by $\a^{-\left(m-\frac{1}{2}\right)}$ to arrive at \eqref{companioneqn} after simplification.
\end{proof}
As special cases of \eqref{equiforma=-2mano}, we obtain two following corollaries involving $\zeta(3)$ and $\zeta(5)$.
\begin{corollary}\label{zeta3}
For $\mathrm{Re}(y)>0$, we have
\begin{align}\label{zeta3formula}
&\sum_{n=1}^\infty \frac{n^{-2}}{e^{ny}-1}+\frac{\pi^2-\gamma y}{12}-\frac{1}{y}\zeta(3)=\frac{y}{4\pi^2}\sum_{n=1}^\infty\frac{1}{n^{2}}\left(\psi\left(\frac{2\pi in}{y}\right)+\psi\left(-\frac{2\pi in}{y}\right)\right).
\end{align}
\end{corollary}
\begin{proof}
Let $m=1$ in \eqref{equiforma=-2mano}.
\end{proof}
\begin{corollary}\label{zeta35}
Let $\mathrm{Re}(y)>0$. Then 
\begin{align*}
\sum_{n=1}^\infty\frac{n^{-4}}{e^{ny}-1}+\left(1+\frac{\gamma y^3}{4\pi^4}\right)\frac{\pi^4}{180}-\frac{1}{y}\zeta(5)-\frac{y}{12}\zeta(3)=-\frac{y^3}{16\pi^4}\sum_{n=1}^\infty\frac{1}{n^{4}}\left(\psi\left(\frac{2\pi in}{y}\right)+\psi\left(-\frac{2\pi in}{y}\right)\right).
\end{align*}
\end{corollary}
\begin{proof}
Let $m=2$ in \eqref{equiforma=-2mano}.
\end{proof}
\begin{remark}
Corollaries \textup{\ref{zeta3}} and \textup{\ref{zeta35}}  together give the following representation for $\zeta(5)$:
\begin{align*}
\zeta(5)&=\left(1+\frac{\gamma y^3}{4\pi^4}\right)\frac{\pi^4y}{180}-\frac{y^3(\pi^2-\gamma y)}{144}+\sum_{n=1}^\infty\left(\frac{y}{n^4}-\frac{y^3}{12n^2}\right)\frac{1}{e^{ny}-1}\\
&\qquad+\frac{y^4}{16\pi^2}\sum_{n=1}^\infty\left(\frac{1}{3n^2}+\frac{1}{\pi^2n^4}\right)\left(\psi\left(\frac{2\pi in}{y}\right)+\psi\left(-\frac{2\pi in}{y}\right)\right),
\end{align*}
which is valid for $\mathrm{Re}(y)>0$.
\end{remark}
\subsection{A transformation for $\sum_{n=1}^{\infty}d(n)e^{-ny}$}\label{d(n)}

\begin{proof}[Corollary \textup{\ref{kan}}][]
Let $a\to 0$ in Theorem \ref{main}. This gives
\begin{align}\label{maineqn1}
&\sum_{n=1}^\infty  d(n)e^{-ny}+\lim_{a\to 0}\left(\frac{1}{2}\left(\frac{2\pi}{y}\right)^{1+a}\mathrm{cosec}\left(\frac{\pi a}{2}\right)\zeta(-a)-\frac{1}{y}\zeta(1-a)\right)+\frac{1}{2}\zeta(0)\nonumber\\
&=\frac{2\pi}{y}\sum_{n=1}^\infty d(n)\lim_{a\to 0}\frac{1}{\sin\left(\frac{\pi a}{2}\right)}\Bigg(\frac{(2\pi n)^{-a}}{\Gamma(1-a)} {}_1F_2\left(1;\frac{1-a}{2},1-\frac{a}{2};\frac{4\pi^4n^2}{y^2} \right) -\left(\frac{2\pi}{y}\right)^{a}\cosh\left(\frac{4\pi^2n}{y}\right)\Bigg).
\end{align}
We first evaluate the limit on the right-hand side. Note that since ${}_1F_2\left(1;\frac{1}{2},1;\frac{4\pi^4n^2}{y^2}\right)=\cosh\left(\frac{4\pi^2n}{y}\right)$, we have $\frac{0}{0}$ form. Hence we need to use L'Hopital's rule. To that end, use Lemma \ref{der1f2at2m} with $m=0$ so as to obtain
{\allowdisplaybreaks\begin{align}\label{serieseval}
&\lim_{a\to 0}\frac{1}{\sin\left(\frac{\pi a}{2}\right)}\Bigg(\frac{(2\pi n)^{-a}}{\Gamma(1-a)} {}_1F_2\left(1;\frac{1-a}{2},1-\frac{a}{2};\frac{4\pi^4n^2}{y^2} \right) -\left(\frac{2\pi}{y}\right)^{a}\cosh\left(\frac{4\pi^2n}{y}\right)\Bigg)\nonumber\\
&=\frac{2}{\pi}\Bigg\{\sinh\left(\frac{4\pi^2n}{y}\right)\mathrm{Shi}\left(\frac{4\pi^2n}{y}\right)-\cosh\left(\frac{4\pi^2n}{y}\right)\mathrm{Chi}\left(\frac{4\pi^2n}{y}\right)+\log\left(\frac{4\pi^2n}{y}\right)\cosh\left(\frac{4\pi^2n}{y}\right)\nonumber\\
&\quad-\log(2\pi n)\cosh\left(\frac{4\pi^2n}{y}\right)-\log\left(\frac{2\pi}{y}\right)\cosh\left(\frac{4\pi^2n}{y}\right)\Bigg\}\nonumber\\
&=\frac{2}{\pi}\left(\sinh\left(\frac{4\pi^2n}{y}\right)\mathrm{Shi}\left(\frac{4\pi^2n}{y}\right)-\cosh\left(\frac{4\pi^2n}{y}\right)\mathrm{Chi}\left(\frac{4\pi^2n}{y}\right)\right).
\end{align}}
We next evaluate the limit on the left-hand side of \eqref{maineqn1}. Using the well-known power-series expansions of $\zeta(-a), \zeta(1-a), (2\pi)^{1+a}, y^{-(1+a)}$ and $\csc\left(\frac{\pi a}{2}\right)$, as $a\to0$, we find that
	\begin{align}\label{limiteval}
\lim_{a\to 0}\left(\frac{1}{2}\left(\frac{2\pi}{y}\right)^{1+a}\mathrm{cosec}\left(\frac{\pi a}{2}\right)\zeta(-a)-\frac{1}{y}\zeta(1-a)\right)
=\frac{-\gamma+\log y}{y}.
\end{align}
Finally, from \eqref{maineqn1}, \eqref{serieseval} and \eqref{limiteval} and the fact that $\zeta(0)=-1/2$, we arrive at the first identity in \eqref{kaneqn}.

To obtain \eqref{kanot}, we invoke Lemma \ref{minusonelemma} with $w=4\pi^2n/y$ in the first step below and then use \eqref{dgkmresult} with $u=4\pi^2n/y$ in the second step so that
\begin{align*}
&\sum_{n=1}^{\infty}d(n)\left\{\sinh\left(\frac{4\pi^2n}{y}\right)\mathrm{Shi}\left(\frac{4\pi^2n}{y}\right)-\cosh\left(\frac{4\pi^2n}{y}\right)\mathrm{Chi}\left(\frac{4\pi^2n}{y}\right)\right\}\nonumber\\
&=\sum_{k=1}^{\infty}\sum_{n=1}^{\infty}\int_{0}^{\infty}\frac{t\cos(t)}{t^2+\left(\frac{4\pi^2nk}{y}\right)^2}\, dt\nonumber\\
&=\frac{1}{2}\sum_{n=1}^{\infty}\left\{\log\left(\frac{2\pi n}{y}\right)-\frac{1}{2}\left(\psi\left(\frac{2\pi in}{y}\right)+\psi\left(-\frac{2\pi in}{y}\right)\right)\right\}.
\end{align*}
This completes the proof of \eqref{kanot}.
\end{proof}

As an equivalent form of Corollary \ref{kan}, we obtain a result of Wigert and Bellman given in \cite[Theorem 2]{kanemitsu}\footnote{In the statement of their theorem, the $1/\pi^2$ appearing in front of the summation on the right-hand side should be $1/\pi$.}.
\begin{corollary}\label{kan1}
Let $U(a;c;z)$ be defined in \eqref{tricomiu}. For $\textup{Re}(y)>0$, 
\begin{align*}
\sum_{n=1}^{\infty}d(n)e^{-ny}-\frac{1}{4}-\frac{\left(\gamma-\log(y)\right)}{y}=\frac{2}{y}\sum_{n=1}^{\infty}d(n)\left\{U\left(1;1;\frac{4\pi^2 n}{y}\right)+U\left(1;1;-\frac{4\pi^2 n}{y}\right)\right\},
\end{align*}
\end{corollary}
\begin{proof}
In view of \eqref{kaneqn}, it suffices to show that for Re$(x)>0$,
\begin{equation}\label{ushichi}
U\left(1;1;x\right)+U\left(1;1;-x\right)=2\left(\sinh(x)\mathrm{Shi}(x)-\cosh(x)\mathrm{Chi}(x)\right).
\end{equation}
First let $x>0$. From \cite[p.~612]{kanemitsu}, $U(1;1;x)=e^x\G(0,x)$. Also, from \cite[p.~902, Equation \textbf{8.359.1}]{grad}, $\G(0,x)=-\textup{Ei}(-x)$. Hence
\begin{align}\label{ushichi1}
U\left(1;1;x\right)+U\left(1;1;-x\right)&=e^{x}\G(0,x)+e^{-x}\G(0,-x)\nonumber\\
&=-\left(e^x\textup{Ei}(-x)+e^{-x}\textup{Ei}(x)\right).
\end{align}
The equality in \eqref{ushichi} now follows for $x>0$ from \eqref{ushichi1} and by equating the right-hand sides of \eqref{zero} and \eqref{one}. By analytic continuation, it holds for Re$(x)>0$.
\end{proof}
\section{Transformation formulas for $\sum_{n=1}^{\infty}\sigma_a(n)e^{-ny}$ for $a$ odd}\label{aodd}
As discussed in the Section \ref{mr}, we get some results in modular forms as corollaries of our master identity, that is, Theorem \ref{extendedid}. These results are derived in this section.

\subsection{Modular transformation for Eisenstein series on $\textup{SL}_{2}(\mathbb{Z})$}\label{aodd2m+1}

\begin{proof}[Corollary \textup{\ref{a=2m-1}}][]
Let $a=2m-1, m>1,$ in Theorem \ref{main}. Using \eqref{zeta1-2m}, the fact that $\zeta(-2m)=0,\ m\in\mathbb{N}$, and noting that
\begin{align}\label{derreg1f22m-1}
\lim_{a\to 2m-1}\frac{1}{\Gamma(1-a)}{}_1F_2\left(1;1-\frac{a}{2},\frac{1-a}{2},\frac{4\pi^4n^2}{y^2}\right)=\left(\frac{4\pi^2n}{y}\right)^{2m-1}\sinh\left(\frac{4\pi^2n}{y}\right),
\end{align}
we see that
\begin{align*}
&\sum_{n=1}^\infty\sigma_{2m-1}(n)e^{-ny}-\frac{1}{2}\left((-1)^{m+1}\left(\frac{2\pi}{y}\right)^{2m}+1\right)\frac{B_{2m}}{2m}\nonumber\\
&=\frac{2\pi(-1)^{m}}{y}\sum_{n=1}^\infty\sigma_{2m-1}(n)\left\{\left(\frac{2\pi}{y}\right)^{2m-1}\cosh\left(\frac{4\pi^2n}{y}\right)-\left(\frac{2\pi}{y}\right)^{2m-1}\sinh\left(\frac{4\pi^2n}{y}\right)\right\}\nonumber\\
&=(-1)^m\left(\frac{2\pi}{y}\right)^{2m}\sum_{n=1}^\infty\sigma_{2m-1}(n)e^{-\frac{4\pi^2n}{y}}.
\end{align*}
This gives
\begin{align}\label{a=2m-1eqn}
&\sum_{n=1}^\infty\sigma_{2m-1}(n)e^{-ny}-\frac{1}{2}\left((-1)^{m+1}\left(\frac{2\pi}{y}\right)^{2m}+1\right)\frac{B_{2m}}{2m}=(-1)^m\left(\frac{2\pi}{y}\right)^{2m}\sum_{n=1}^\infty\sigma_{2m-1}(n)e^{-\frac{4\pi^2n}{y}}.
\end{align}
To prove \eqref{a=2m-1equv}, let $y=2\alpha$ with $\alpha\beta=\pi^2$ in \eqref{a=2m-1eqn} and simplify.
\end{proof}

\subsection{Transformation formula for weight-$2$ Eisenstein series on $\textup{SL}_{2}(\mathbb{Z})$}\label{a1}
\begin{proof}[Corollary \textup{\ref{e2transcor}}][]
Let $a=1$ in Theorem \ref{main}. Using the well-known special values $\zeta(-1)=-1/12,\ \zeta(0)=-1/2$ and invoking \eqref{derreg1f22m-1}, we see that
 \begin{align*}
 \sum_{n=1}^\infty\sigma(n)e^{-ny}-\frac{1}{24}\left(1+\frac{4\pi^2}{y^2}\right)+\frac{1}{2y}&=-\frac{4\pi^2}{y^2}\sum_{n=1}^{\infty}\sigma(n)\left(\cosh\left(\frac{4\pi^2n}{y}\right)-\sinh\left(\frac{4\pi^2n}{y}\right)\right)\nonumber\\
&=-\frac{4\pi^2}{y^2}\sum_{n=1}^\infty\sigma(n) e^{-\frac{4\pi^2n}{y}}.
\end{align*}
Letting $y=2\a$ with $\a\b=\pi^2$ then leads to \eqref{e2trans}.
\end{proof}

\subsection{Ramanujan's famous formula for $\zeta(2m+1)$}\label{aodd-2m-1}

We begin with a lemma which evaluates ${}_{\frac{1}{2}}K_{\frac{-2m-1}{2}}(z,0)$ in terms of elementary functions.
\begin{lemma}\label{khalfwithm}
Let $z\in\mathbb{C}$ and $m\in\mathbb{N}\cup\{0\}$. Then
\begin{align}\label{khalfwithmeqn}
{}_{\frac{1}{2}}K_{\frac{-2m-1}{2}}(z,0)&=\frac{(-1)^{m}\sqrt{\pi}}{\sqrt{2}}\left\{z^{-\frac{2m+1}{2}}e^{-z}+z^{\frac{2m-3}{2}}\sum_{k=0}^m\frac{z^{-2k}}{\Gamma\left(2m-2k\right)}\right\}.
\end{align}
\end{lemma}
\begin{proof}
Using the definition of ${}_{\mu}K_{\nu}(z, w)$ from \eqref{def2varbessel}, we see that
\begin{align}\label{khalfwithz}
{}_{\frac{1}{2}}K_{\frac{-2m-1}{2}}(z,0)=\frac{(-1)^{m+1}\sqrt{\pi}}{\sqrt{2}}\left\{\frac{z^{\frac{2m+1}{2}}}{\Gamma(2m+2)}{}_1F_2\left(1;m+1,m+\frac{3}{2};\frac{z^2}{4}\right)-\frac{\cosh(z)}{z^{\frac{2m+1}{2}}}\right\}.
\end{align}
Now it can be proved along the similar lines as in the proof of Lemma \ref{1f2at-2m} that
\begin{align}\label{z21f2}
&\frac{1}{\Gamma(2m+2)}{}_1F_2\left(1;m+1,m+\frac{3}{2};\frac{z^2}{4}\right)=-\frac{1}{z^2}\sum_{k=0}^{m-1}\frac{z^{-2k}}{\Gamma(2m-2k)}+\frac{1}{z^{2m+1}}\sinh(z).
\end{align}
Substitute \eqref{z21f2} in \eqref{khalfwithz} to arrive at \eqref{khalfwithmeqn}.
\end{proof}

\begin{proof}[Corollary \textup{\ref{ram}}][]
Let $a=-2m-1$ in Theorem \ref{extendedid} to get
{\allowdisplaybreaks\begin{align}\label{khalfnb}
&\sum_{n=1}^\infty  \sigma_{-2m-1}(n)e^{-ny}+\frac{1}{2}\left(\left(\frac{2\pi}{y}\right)^{-2m}(-1)^{m+1}+1\right)\zeta(2m+1)-\frac{\zeta(2m+2)}{y}\nonumber\\
&=\frac{2\sqrt{2\pi}}{y^{\frac{1}{2}-m}}\sum_{n=1}^\infty\sigma_{-2m-1}(n)n^{\frac{2m+1}{2}}\Bigg\{{}_{\frac{1}{2}}{K}_{\frac{-2m-1}{2}}\left(\frac{4\pi^2n}{y},0\right)-(-1)^{m+1}\pi2^{\frac{1}{2}-2m}\left(\frac{4\pi^2n}{y}\right)^{\frac{2m-3}{2}}\nonumber\\
&\hspace{1mm} \times A_m\left(\frac{1}{2},\frac{-2m-1}{2},0;\frac{4\pi^2n}{y}\right)\Bigg\}-(-1)^{m+1}y(2\pi)^{2m-2}\sum_{k=0}^m\frac{\zeta(2k-2m+1)\zeta(2k+2)}{\Gamma(2m-2k)}\left(\frac{4\pi^2}{y}\right)^{-2k}.
\end{align}}
By using \eqref{khalfwithmeqn} with $z=4\pi^2n/y$ and the definition of $A_m$ in \eqref{am}, we see that
\begin{align}\label{khalfexp}
&{}_{\frac{1}{2}}{K}_{\frac{-2m-1}{2}}\left(\frac{4\pi^2n}{y},0\right)-(-1)^{m+1}\pi2^{\frac{1}{2}-2m}\left(\frac{4\pi^2n}{y}\right)^{\frac{2m-3}{2}} A_m\left(\frac{1}{2},\frac{-2m-1}{2},0;\frac{4\pi^2n}{y}\right)\nonumber\\
&=\frac{(-1)^m\sqrt{\pi}}{\sqrt{2}(ny)^{m+\frac{1}{2}}}\left(\frac{2\pi}{y}\right)^{-2m-1}\exp\left(-\frac{4\pi^2n}{y}\right).
\end{align}
From \eqref{khalfnb} and \eqref{khalfexp},
{\allowdisplaybreaks\begin{align}\label{idinzetazeta}
&\sum_{n=1}^\infty  \sigma_{-2m-1}(n)e^{-ny}+\frac{1}{2}\left(\left(\frac{2\pi}{y}\right)^{-2m}(-1)^{m+1}+1\right)\zeta(2m+1)-\frac{\zeta(2m+2)}{y}\nonumber\\
&=\left(\frac{2\pi}{y}\right)^{-2m}(-1)^m\sum_{n=1}^\infty\sigma_{-2m-1}(n)\exp\left(-\frac{4\pi^2n}{y}\right)\nonumber\\
&\qquad-(-1)^{m+1}y(2\pi)^{2m-2}\sum_{k=0}^m\frac{\zeta(2k-2m+1)\zeta(2k+2)}{\Gamma(2m-2k)}\left(\frac{4\pi^2}{y}\right)^{-2k}\nonumber\\
&=\left(\frac{2\pi}{y}\right)^{-2m}(-1)^m\sum_{n=1}^\infty\sigma_{-2m-1}(n)\exp\left(-\frac{4\pi^2n}{y}\right)\nonumber\\
&\qquad-\frac{y^{2m+1}}{2}\sum_{k=0}^m\frac{(-1)^{k}B_{2m+2-2k}B_{2k}}{(2m+2-2k)!(2k)!}\left(\frac{2\pi}{y}\right)^{2k},
\end{align}}%
where in the last step, we replaced $k$ by $m-k$ in the finite sum and then used \eqref{zeta(2m)} and \eqref{zeta1-2m}.
Now use \eqref{zeta(2m)} on the left-hand side of \eqref{idinzetazeta} to obtain 
{\allowdisplaybreaks\begin{align*}
&\sum_{n=1}^\infty  \sigma_{-2m-1}(n)e^{-ny}+\frac{1}{2}\zeta(2m+1)+\frac{(-1)^{m+1}}{2}\left(\frac{y}{2\pi}\right)^{2m}\zeta(2m+1)+\frac{(-1)^{m+1}(2\pi)^{2m+2}B_{2m+2}}{2y(2m+2)!}\nonumber\\
&=\left(\frac{2\pi}{y}\right)^{-2m}(-1)^m\sum_{n=1}^\infty\sigma_{-2m-1}(n)\exp\left(-\frac{4\pi^2n}{y}\right)-\frac{y^{2m+1}}{2}\sum_{k=0}^m\frac{(-1)^{k}B_{2m+2-2k}B_{2k}}{(2m+2-2k)!(2k)!}\left(\frac{2\pi}{y}\right)^{2k}.
\end{align*}}
Now multiply both sides of the above expression by $\left(\frac{y}{2}\right)^{-m}$ and simplify to see that
{\allowdisplaybreaks\begin{align*}
&\left(\frac{y}{2}\right)^{-m}\left\{\frac{1}{2}\zeta(2m+1)+\sum_{n=1}^\infty \sigma_{-2m-1}(n)e^{-ny}\right\}-\left(\frac{-y}{2\pi^2}\right)^m\left\{\frac{1}{2}\zeta(2m+1)+\sum_{n=1}^\infty\sigma_{-2m-1}(n)e^{-\frac{4\pi^2n}{y}}\right\}\nonumber\\
&=-2^{m-1}y^{m+1}\sum_{k=0}^{m+1}\frac{(-1)^{k}B_{2m+2-2k}B_{2k}}{(2m+2-2k)!(2k)!}\left(\frac{2\pi}{y}\right)^{2k}.
\end{align*}}
Finally let $y=2\alpha$ and $\alpha\beta=\pi^2$ and simplify to get \eqref{rameqn}.
\end{proof}

\subsection{Transformation formula for the logarithm of the Dedekind eta function}\label{a-1}

\begin{proof}[Corollary \textup{\ref{maineqn2cor}}][]
Let $m=0$ in Theorem \ref{extendedid} and then let $a\to-1$ so that 
\begin{align*}
&\sum_{n=1}^{\infty}\sigma_{-1}(n)e^{-ny}+\lim_{a\to-1}\left\{\left(\left(\frac{2\pi}{y}\right)^{1+a}\mathrm{cosec}\left(\frac{\pi a}{2}\right)+1\right)\zeta(-a)\right\}-\frac{1}{y}\zeta(2)\nonumber\\
&=\frac{2\sqrt{2\pi}}{\sqrt{y}}\sum_{n=1}^\infty \sigma_{-1}(n)\sqrt{n}{}_{\frac{1}{2}}{K}_{\frac{-1}{2}}\left(\frac{4\pi^2n}{y},0\right)+\frac{2\sqrt{2\pi}}{\sqrt{y}}\left(\frac{y^{3/2}}{24\sqrt{2\pi}}\lim_{a\to-1}\frac{\zeta(a+2)}{\G\left(\frac{-1-a}{2}\right)}\right).
\end{align*}
Using the limit evaluations $\displaystyle\lim_{a\to-1}\frac{\zeta(a+2)}{\G\left(\frac{-1-a}{2}\right)}=-\frac{1}{2}$ and 
\begin{align*}
\lim_{a\to-1}\left\{\left(\left(\frac{2\pi}{y}\right)^{1+a}\mathrm{cosec}\left(\frac{\pi a}{2}\right)+1\right)\zeta(-a)\right\}=\log\left(\frac{2\pi}{y}\right)
\end{align*}
in the above identity, invoking \eqref{mueq-nu} with $\nu=-1/2$ and using $\zeta(2)=\pi^2/6$, we are led to 
\begin{align*}
&\sum_{n=1}^\infty  \sigma_{-1}(n)e^{-ny}+\frac{1}{2}\log\left(\frac{2\pi}{y}\right)-\frac{\pi^2}{6y}+\frac{y}{24}=\sum_{n=1}^{\infty}\sigma_{-1}(n)e^{-\frac{4\pi^2n}{y}}.
\end{align*}
upon simplification. To obtain \eqref{maineqn2alt}, simply let $y=2\a$ and use the fact that $\a\b=\pi^2$ in the above identity.
\end{proof}

\section{Concluding Remarks and future directions}\label{cr}

Koshliakov \cite{kosh1938} studied an integral transform with respect to the $w=0$ case of Watson's kernel $\mathscr{G}_{\nu}(x, w)$, that is, the first Koshliakov kernel \eqref{kk1}, which motivated us to obtain an explicit transformation formula for $\sum_{n=1}^{\infty}\sigma_a(n)e^{-ny}$ for \emph{any} $a\in\mathbb{C}$ and Re$(y)>0$ that was missing in the literature. The wealth of information that this transformation contains is evident from the numerous corollaries derived from it in Section \ref{mr}. These include the modular properties of Eisenstein series on $\textup{SL}_{2}\left(\mathbb{Z}\right)$ as well as explicit transformations for the series $\sum_{n=1}^{\infty}\sigma_{2m}(n)e^{-ny}, m\in\mathbb{Z}$.  They include a novel companion of Ramanujan's famous formula for $\zeta(2m+1)$ given in Corollary \ref{companion}. 

Are there any applications of results such as \eqref{kanot} or \eqref{zeta3formula} from the point of view of transcendental number theory? Note that Erd\"{o}s \cite{erdos} has shown that the series $\sum_{n=1}^{\infty}d(n)q^{-n}$ is irrational for any integer $q$ with $|q|\geq 2$. Now observe that, if we let $y=\log (2)$ in \eqref{kanot}, the left-hand side becomes $
\sum_{n=1}^\infty\frac{1}{e^{n\log 2}-1}=\sum_{n=1}^\infty d(n)2^{-n}$
so that, Erd\"{o}s result implies that it is an irrational number and hence \eqref{kanot}, in turn, implies that $$\frac{\left(\gamma-\log(\log (2))\right)}{\log (2)}+\frac{2}{y}\sum_{n=1}^{\infty}\left\{\log\left(\frac{2\pi n}{\log (2)}\right)-\frac{1}{2}\left(\psi\left(\frac{2\pi in}{\log (2)}\right)+\psi\left(-\frac{2\pi in}{\log (2)}\right)\right)\right\}$$
is irrational.

The two-variable extension of the modified Bessel function, namely, ${}_{\mu}K_{\nu}(z, w)$, is instrumental in generalizing the well-known modular transformations (as well as modular-type) appearing in the results of Section \ref{mr}, and hence deserves further study. In this paper, we have restricted ourselves to obtaining only those properties of ${}_{\mu}K_{\nu}(z, w)$ relevant to deriving various transformations.

Ramanujan conceived an overarching generalization of his formula \eqref{rameqn}. See \cite[pp.~429-432]{bcbramfournote} for its rigorous formulation and a proof. In the same spirit, it may be interesting to look at the corresponding generalization of \eqref{companion}.

Our companion to Ramanujan's formula for $\zeta(2m+1)$, namely Equation \eqref{companioneqn}, contains the higher Herglotz function of Vlasenko and Zagier. We note that the Herglotz as well as the higher Herglotz functions are useful in algebraic as well as analytic number theory. Thus it would be of merit to study \eqref{companioneqn} from this viewpoint.

Letting $\a=\b=\pi$ in Corollary \ref{companion} implies that the hyperbolic cotangent Dirichlet series and $\sum_{n=1}^{\infty}\left(\psi(in)+\psi(-in)\right)n^{-2m}$ are intimately connected at even positive integers $s=2m$, that is,
\begin{align}\label{berndtanalogue}
\sum_{n=1}^{\infty}\frac{\textup{coth}(\pi n)}{n^{2m}}&=(-1)^{m+1}\frac{2\g}{\pi}\zeta(2m)+\frac{2}{\pi}\sum_{k=0}^{m-1}\frac{2^{2k-1}B_{2k}\zeta(2m-2k+1)\pi^{2k}}{(2k)!}\nonumber\\
&\quad+\frac{(-1)^{m+1}}{\pi}\sum_{n=1}^{\infty}\frac{\psi(in)+\psi(-in)}{n^{2m}}.
\end{align}
Equation \eqref{berndtanalogue} is an analogue of Lerch's formula \cite{lerch}, namely, when $m\in\mathbb{N}$ is odd, 
\begin{equation*}
\sum_{n=1}^{\infty}\frac{\textup{coth}(\pi n)}{n^{2m+1}}=2^{2m}\pi^{2m+1}\sum_{k=0}^{m+1}(-1)^{k+1}\frac{B_{2k}}{(2k)!}\frac{B_{2m+2-2k}}{(2m+2-2k)!}.
\end{equation*}
This suggests further work, in the setting of \eqref{berndtanalogue}, along the lines in \cite{komori} and \cite{straubramanujan}.

Finally, it may be worthwhile extending our Theorem \ref{koshlyakovg-wthm} wherein the kernel $\mathscr{G}_{\nu}(x, w)$, or equivalently $\varpi_{\mu, \nu}(x)$, is replaced by its generalizations, for example, those studied in \cite{bhatnagarbelgique}, \cite{bhatnagar1954} and \cite{olkharathie}.

\begin{center}
\textbf{Acknowledgements}
\end{center}

The authors sincerely thanks the referees for giving very nice suggestions which improved the exposition of the paper. They would also like to thank Olivia da Costa Maya, Professors Alexandru Zaharescu, Lin Jiu and Gaurav Dwivedi respectively for sending them the copies of references \cite{bhatnagarbelgique}, \cite{dahiyaromanian}, \cite{jekhowsky1} and \cite{mainrasingh}. They also thank Becky Burner, a library staff at the University of Illinois at Urbana-Champaign, for procuring the copies of \cite{bhatnagar1954} and \cite{singhrajasthan}, Dr. T. S. Kumbar, the librarian at IIT Gandhinagar for obtaining a copy of \cite{bhatnagarganita}, and Suresh Kumar, the librarian at the Harish-Chandra Research Institute, for arranging a copy of \cite{bhatnagar1954f}. The first author's research was supported by the CRG grant CRG/2020/002367. He sincerely thanks SERB for the support. The third author's research was supported by the grant IBS-R003-D1 of the IBS-CGP, POSTECH, South Korea and by IIT Gandhinagar. He sincerely thanks both the institutes for the support.

\textbf{Conflict of interest statement:} On behalf of all authors, the corresponding author states that there is no conflict of interest.

\textbf{Data availability statement:} Data sharing not applicable to this article as no datasets were generated or analysed during the current study.

\end{document}